\documentclass[a4paper,11pt]{article}
\usepackage{amsmath}
\usepackage{amsfonts}
\usepackage{amssymb}
\usepackage{mathrsfs}
\usepackage{amscd}
\usepackage{pb-diagram}
\usepackage{amsthm}
\usepackage{color}
\usepackage[all]{xy}
\usepackage{graphicx}
\usepackage{url}
\usepackage{enumerate}
\usepackage{tikz-cd}
\numberwithin{equation}{section} 
\usepackage{caption}
\DeclareUnicodeCharacter{2212}{-}
\usepackage{tikz}

\usepackage{titlesec}
\titleformat{\subsection}[runin]{\normalsize\bfseries}{\thesubsection}{5pt}{}

\usepackage{tikz}          
\usetikzlibrary{matrix, arrows, decorations.pathmorphing}

\newcommand{\moment}{
\textup{\textbf{J}}
}

\author{Mathieu Molitor
\\
\small{\it{e-mail:}}\,\,\url{pergame.mathieu@gmail.com}}
\title{Geometric spectral theory for K\"{a}hler functions}
\date{}

\begin{document}

\theoremstyle{definition}
\newtheorem{lemma}{Lemma}[section]
\newtheorem{definition}[lemma]{Definition}
\newtheorem{proposition}[lemma]{Proposition}
\newtheorem{corollary}[lemma]{Corollary}
\newtheorem{theorem}[lemma]{Theorem}
\newtheorem{remark}[lemma]{Remark}
\newtheorem{example}[lemma]{Example}
\newtheorem{setting}[lemma]{Setting }

\bibliographystyle{alpha}

\maketitle 


\begin{abstract}
	We consider K\"{a}hler toric manifolds $N$ that are torifications of statistical manifolds $\mathcal{E}$ in the sense of 
	$\cite{molitor-toric}$, and prove a geometric analogue of the spectral decomposition theorem in which Hermitian matrices 
	are replaced by K\"{a}hler functions on $N$. The notion of ``spectrum" of a K\"{a}hler function is defined, 
	and examples are presented. This paper is motivated by the geometrization program 
	of Quantum Mechanics that we pursued in previous 
	works \cite{Molitor2012,Molitor-exponential, Molitor2014,Molitor-2015,molitor-toric,molitor-weyl}.
\end{abstract}




\section{Introduction}

%

	Let $\mathbb{P}_{n}(c)$ be the complex projective space of complex dimension $n$ and 
	holomorphic sectional curvature $c>0$, and let $\mathcal{P}_{n+1}$
	be the set of probability functions $p$ defined on a finite set $\Omega=\{x_{1},...,x_{n+1}\}$. (Thus a function $p:\Omega\to \mathbb{R}$ is 
	in $\mathcal{P}_{n+1}$ if and only if $p(x)\geq 0$ for all $x\in\Omega$ and $\sum_{x\in \Omega}p(x)=1$.) 
	Let $K:\mathbb{P}_{n}(c)\to \mathcal{P}_{n+1}$ be the function defined by 
	\begin{eqnarray}\label{nfeknkefnkn}
		K([z])(x_{k})=\dfrac{|z_{k}|^{2}}{|z_{1}|^{2}+...+|z_{n+1}|^{2}}, 
	\end{eqnarray}
	where $[z]=[z_{1},...,z_{n+1}]\in \mathbb{P}_{n}(c)$ (homogeneous coordinates) and $x_{k}\in \Omega$. 
	Recall that a real-valued function $f$ on a K\"{a}hler manifold is said to be \textit{K\"{a}hler} if its 
	Hamiltonian vector field $X_{f}$ is a Killing vector field. One of the main objectives of this paper is 
	to generalize the following result. 

\begin{proposition}\label{nknefknkdnknk}
	A smooth function $f:\mathbb{P}_{n}(c)\to \mathbb{R}$ is K\"{a}hler 
	if and only if there exist a random variable $X:\Omega\to \mathbb{R}$ and a holomorphic isometry 
	$\phi$ of $\mathbb{P}_{n}(c)$ such that 
	\begin{eqnarray}\label{nfeknkefnkndfefdfk}
		f(p)=\mathbb{E}_{K(\phi(p))}(X)
	\end{eqnarray}
	for all $p\in \mathbb{P}_{n}(c)$, where the right-hand side of \eqref{nfeknkefnkndfefdfk} is the expectation of $X$ 
	with respect to the probability function $K(\phi(p))$. 
\end{proposition}

	This result is just the spectral decomposition theorem for Hermitian matrices in a geometric and probabilistic guise. 
	Indeed, let $\textup{Herm}_{n}$ be the space of $n\times n$ Hermitian matrices. 
	Given $H\in \textup{Herm}_{n+1}$, let $f_{H}:\mathbb{P}_{n}(c)\to \mathbb{R}$ be the 
	function defined by $f_{H}([z])=\tfrac{\langle z,Hz\rangle}{\langle z,z\rangle}$, 
	where $\langle z,w\rangle=\overline{z}_{1}w_{1}+...+\overline{z}_{n+1}w_{n+1}$ 
	is the usual Hermitian product on $\mathbb{C}^{n+1}$. It is well-known that 
	the correspondence $H\mapsto f_{H}$ defines a bijection between $\textup{Herm}_{n+1}$ and the space 
	of K\"{a}hler functions on $\mathbb{P}_{n}(c)$ (see, e.g., \cite{Cirelli-Quantum}). Therefore, one can 
	carry over the usual spectral decomposition theorem for Hermitian 
	matrices to the space of K\"{a}hler functions on $\mathbb{P}_{n}(c)$, which yields Propositon \ref{nknefknkdnknk}. 

	Our generalization of Proposition \ref{nknefknkdnknk}, to be discussed below, is based on the concept of ``torification" of 
	a statistical manifold, that we introduced in \cite{molitor-toric}. 
	This refers to a geometric construction that associates to certain statistical manifolds, a K\"{a}hler manifold equipped with 
	a torus action $\Phi:\mathbb{T}^{n}\times N\to N$, where $n$ is the complex dimension of $N$ and 
	$\mathbb{T}^{n}=\mathbb{R}^{n}/\mathbb{Z}^{n}$ is the $n$-dimensional real torus. 
	For example, if $c=1$, then $\mathbb{P}_{n}(c)$ is the torification of
	$\mathcal{P}_{n+1}^{\times}:=\{p\in \mathcal{P}_{n+1}\,\,\vert\,\,p(x)>0\,\,\forall\,x\in \Omega\}$ (for the corresponding 
	torus action, see Section \ref{nknkefnknk} and Proposition \ref{nfeknwknwknk}).

	In general, a torification $N$ of a statistical manifold $\mathcal{E}$ comes equipped with a family of maps 
	$\kappa:N^{\circ}\to \mathcal{E}$, we call them \textit{compatible projections}, that are defined on the set $N^{\circ}$ 
	of points $p\in N$ where the torus action is free. Under suitable conditions, we prove (see Section \ref{ndknvfknkfnvknk}) 
	that each $\kappa$ extends uniquely to a continuous map $K:N\to \overline{\mathcal{E}}$, where 
	$\overline{\mathcal{E}}$ is an appropriate superset of $\mathcal{E}$ formed by probability functions. 
	The map \eqref{nfeknkefnkn} is an example of such extension.

	The statistical manifolds $\mathcal{E}$ of our generalization are \textit{exponential families} 
	(see Definition \ref{def:5.3}). They form an important class of statistical manifolds which are found
	among the most common probability distributions: Bernoulli, beta, binomial, chi-square, Dirichlet, exponential, gamma,
	geometric, multinomial, normal, Poisson, to name but just a few.  The data that are needed to define an exponential 
	family are a base measure on $\Omega$ and a finite set of random variables on $\Omega$. Using these random variables, 
	we contruct, for each exponential family $\mathcal{E}$, a finite dimensional vector space 
	of random variables on $\Omega$, that we denote by $\mathcal{A}_{\mathcal{E}}$ (see Section \ref{nfeqnkdekwnn}). 
	This space plays an important role in our generalization. For example, 
	if $\mathcal{E}=\mathcal{P}_{n+1}^{\times}$, then $\mathcal{A}_{\mathcal{E}}$ is the space of all random variables on 
	$\Omega=\{x_{1},...,x_{n+1}\}$. If $\mathcal{E}=\mathcal{B}(n)$ is the set of all binomial distributions 
	$p(k)=\binom{n}{k}q^{k}(1-q)^{n-k}$, $q\in (0,1)$,
	defined over $\Omega=\{0,1,...,n\}$, then $\mathcal{A}_{\mathcal{E}}$ is the set of functions $f:\{0,...,n\}\to \mathbb{R}$ 
	of the form $f(k)=\alpha\,k+\beta$, where $\alpha$ and $\beta$ are real numbers.

	Now we are ready to state our generalization (Theorem \ref{newnknfkenkndvknfk}). 
	Let $N$ be a toric K\"{a}hler manifold with real-analytic K\"{a}hler metric $g$. Let $\textup{Aut}(N,g)^{0}$ be 
	the identity component of the group of holomorphic isometries of $N$. Suppose that $N$ is a torification of 
	an exponential family $\mathcal{E}$ defined over a finite set $\Omega=\{x_{1},...,x_{m}\}$. Let $K:N\to \overline{\mathcal{E}}$ 
	be the extension of some compatible projection. Then a function $f:N\to \mathbb{R}$ is K\"{a}hler if and only if there are 
	a random variable 
	$X\in \mathcal{A}_{\mathcal{E}}$ and $\phi\in \textup{Aut}(N,g)^{0}$ such that 
	$f(p)=\mathbb{E}_{K(\phi(p))}(X)$ for all $p\in N$. 

	Let us look at an example. Let $\mathcal{E}=\mathcal{B}(n)$ be the set of binomial distributions 
	defined over $\Omega=\{0,1,...,n\}$. Then $N=\mathbb{P}_{1}(\tfrac{1}{n})$ is a torification of $\mathcal{B}(n)$. 
	The superset $\overline{\mathcal{B}(n)}$ is the disjoint union $\mathcal{B}(n)\cup\{\delta_{0},\delta_{n}\}$, 
	where $\delta_{i}:\Omega\to \mathbb{R}$ is defined by $\delta_{i}(j)=1$ if $i=j$, 0 otherwise. 
	For simplicity, let us identify $\mathbb{P}_{1}(\tfrac{1}{n})$ with the unite sphere $S^{2}\subset \mathbb{R}^{3}$ 
	endowed with $n$ times the round metric induced from $\mathbb{R}^{3}$. Then the map $K:S^{2}\to \overline{\mathcal{B}(n)}$ defined 
	by $K(x,y,z)(k)=2^{-n}\binom{n}{k}(1+z)^{k}(1-z)^{n-k}$, $(x,y,z)\in S^{2}$, $k\in\{0,1,...,n\}$, 
	is the extension of some compatible projection (by convention $0^{0}=1$). It is easy to see that 
	a function $f:S^{2}\to \mathbb{R}$ is K\"{a}hler 
	if and only if it is of the form $f(p)=\langle u,p\rangle+c$, where $u\in \mathbb{R}^{3}$, $c\in \mathbb{R}$, 
	and where $\langle\,,\,\rangle$ is the Euclidean product on $\mathbb{R}^{3}$. Our result says that a function
	$f$ on $S^{2}$ is K\"{a}hler if and only if there are a random variable $X:\Omega\to \mathbb{R}$ of the form 
	$X(k)=\alpha k+\beta$ ($\alpha,\beta\in \mathbb{R}$) and $U=(U_{ij})_{1\leq i,j\leq 3}\in SO(3)$ such that 
	\begin{eqnarray*}
		f(x,y,z)=\dfrac{1}{2^{n}}\sum_{k=0}^{n}\binom{n}{k}(\alpha k+\beta) (1+z')^{k}(1-z')^{n-k}
	\end{eqnarray*}
	for all $(x,y,z)\in S^{2}$, where $\binom{n}{k}=\tfrac{n!}{(n-k)!k!}$ and $z':=U_{31}x+U_{32}y+U_{33}z$. 

	The proof of our generalization uses group theoretical and symplectic arguments. The first step is to prove that if $N$ is a toric 
	K\"{a}hler manifold with momentum map $\moment:N\to \mathbb{R}^{n}$, then a smooth function 
	$f:N\to \mathbb{R}$ is K\"{a}hler if and only if there are $\xi\in \mathbb{R}^{n}$ (= Lie algebra of $\mathbb{T}^{n}$), 
	a holomorphic isometry $\phi$ of $N$ and $r\in \mathbb{R}$ such that 
	$f(p)=\moment^{\xi}(\phi(p))+r$ for all $p\in N$, where $\moment^{\xi}:N\to \mathbb{R}$ is defined by $\moment^{\xi}(p)
	=\langle \moment(p),\xi\rangle$ (here $\langle\,,\,\rangle$ is the Euclidean metric). This result follows from 
	the infinitesimal version of the classical conjugation Theorem of Cartan for maximal tori, 
	and the fact that the identity component of 
	the group of holomorphic isometries of $N$ is a compact connected Lie group whose natural action on $N$ is Hamiltonian (see 
	Section \ref{nenefkwnknefkwndk} and Proposition \ref{nenekfnsddnekn}). The second step uses the fact that $N$ is a 
	torification of $\mathcal{E}$. We show (Lemma \ref{nfeknkrkefnk} and Proposition \ref{nfeknekfnkn}) that the momentum map 
	$\moment:N\to \mathbb{R}^{n}$ associated to the torus action can be written as 
	$\moment(p)=(\mathbb{E}_{K(p)}(X_{1}),...,\mathbb{E}_{K(p)}(X_{n}))$, where $X_{1},...,X_{n}$ are random variables on $\Omega$ 
	and $K:N\to \overline{\mathcal{E}}$ is the extension of some compatible projection. Combining these two steps yields the generalization 
	of Proposition \ref{nknefknkdnknk}. 

	The second objective of this paper, which is motivated by our previous works on the quantum formalism 
	\cite{Molitor2012,Molitor-exponential, Molitor2014,Molitor-2015,molitor-toric,molitor-weyl},
	is to show that the image of $X$ in the decomposition $f(p)=\mathbb{E}_{K(\phi(p))}(X)$ 
	is independent of $K$, $\phi$ and $X$, allowing us to define what we call the \textit{spectrum} of a  K\"{a}hler function $f$ as the set 
	\begin{eqnarray*}
		\textup{spec}(f):=\textup{image of}\,\,X.
	\end{eqnarray*}
	This result is more delicate than the generalization of Proposition \ref{nknefknkdnknk}. The proof is a generalization of the 
	following phenomenom. Let $G=U(n)$ be the unitary group and let $T\subset U(n)$ be the subgroup of diagonal matrices. Then 
	$T$ is a maximal torus in $U(n)$; let $W$ be the Weyl group of $U(n)$ associated to $T$. It is well-known that 
	$W$ is isomorphic to the group $\mathbb{S}_{n}$ of permutations of $n$ objects, and that its natural action on $T$ 
	is by permutations of the diagonal entries. Now, if $A$ is a skew-Hermitian matrix, regarded as 
	an element of the Lie algebra $\mathfrak{u}(n)$ of $U(n)$, then there are $U\in U(n)$ and a diagonal matrix $D\in T$ such that 
	$A=\textup{Ad}_{U}D=UDU^{*}$, where $U^{*}$ is the conjugate of $U$ (this follows from the infinitesimal conjugation Theorem 
	of Cartan, see Section \ref{nkennkfenknk}). The decompositon $A=UDU^{*}$ is not unique, but if $A=UDU^{*}=U'D'(U')^{*}$, then 
	$D$ and $D'$ differ by a permutation of the diagonal entries; this follows from the fact 
	that the intersection of an adjoint orbit of a connected compact Lie group with the Lie algebra of a maximal torus is a Weyl orbit. 
	Because of this, one could \textit{define} the spectrum of $A$ as the set of diagonal entries of $D$; this would be consistant with 
	the standard definition.

	Our proof that $\textup{spec}(f)$ is well-defined is similar, and uses the Weyl group $W$ of 
	$G=\textup{Aut}(N,g)^{0}$. First we show that the space of random variables 
	$\mathcal{A}_{\mathcal{E}}$ associated to $\mathcal{E}$ is naturally isomorphic to $\mathbb{R}^{n}\oplus \mathbb{R}$, where $\mathbb{R}^{n}$ 
	is regarded as the Lie algebra of the torus. From this, we obtain that $W$ acts on $\mathcal{A}_{\mathcal{E}}$, 
	and that if a K\"{a}hler function $f:N\to \mathbb{R}$ satisfies 
	$f(p)=\mathbb{E}_{K(\phi(p))}(X)=\mathbb{E}_{K'(\phi'(p))}(X')$ for all $p\in N$, where $X,X'\in \mathcal{A}_{\mathcal{E}}$, 
	then there is $w\in W$ such that $X=w\cdot X'$ (see Section \ref{nekenrkefnknk}). Now, the crucial argument, proven in 
	\cite{molitor-weyl}, is that the action of $W$ on $\mathcal{A}_{\mathcal{E}}$ is by permutations of $\Omega$ 
	(see Section \ref{nekenrkefnknk}). This implies that the images of $X=w\cdot X'$ and $X'$ coincide, from which it follows 
	that $\textup{spec}(f)$ is well-defined. 

	Let us look at two examples. If $\mathcal{E}=\mathcal{P}_{n+1}^{\times}$ and 
	$N=\mathbb{P}_{n}(1)$, then the spectrum of $f=f_{H}$ is precisely 
	the spectrum in the classical sense of the Hermitian matrix $H$ (see Proposition \ref{neeknereknfk}). Remarkably, 
	this spectrum coincides with the set of critical values of $f$ (see Proposition \ref{nfeknkrnefknk}), and therefore 
	our definition of $\textup{spec}(f)$ agrees with the standard definition of the spectrum of a K\"{a}hler function 
	on $\mathbb{P}_{n}(1)$ as given in Geometric Quantum Mechanics (see, e.g., \cite{Ashtekar}). 

	If $\mathcal{E}=\mathcal{B}(n)$ and $N=S^{2}$, then the spectrum of $f(p)=\langle u,p\rangle+c$ ($u\in \mathbb{R}^{3}$, 
	$c\in \mathbb{R}$) is the set 
	$\{\lambda_{k}\,\,\vert\,\,k=0,1,...,n\}$, where 
	$\lambda_{k}=\tfrac{2}{n}\|u\|\big(-\tfrac{n}{2}+k\big)+c$ (see Proposition \ref{nfekwnkefnknkn}). In particular, 
	if $\|u\|=j:=\tfrac{n}{2}$ and $c=0$, then 
	\begin{eqnarray}\label{nfeknkewknk}
		\textup{spec}(f)=\{-j,-j+1,...,j-1,j\}.	
	\end{eqnarray}
	We recognize the spectrum of the usual spin operator $S_{u}$ (angular momentum) of Quantum Mechanics in the direction $u$.
	This is not a coincidence, and the precise connection between our definition of $\textup{spec}(f)$
	and the unitary representations of $\mathfrak{su}(2)$ (used by physicists 
	to describe the spin of a particle) is discussed in detail in \cite{Molitor-exponential} (see also below).

	The last result of this paper is a formula that relates the Riemannian gradient $\textup{grad}(f)$ of a K\"{a}hler function 
	$f:N\to \mathbb{R}$ with the variance of a certain random variable. More precisely, if 
	$f(p)=\mathbb{E}_{K(\phi(p))}(X)$, where $N$ is the torification of an exponential family 
	$\mathcal{E}$ defined over a finite set $\Omega$, then 
	\begin{eqnarray}\label{nekwnkefnknk}
		\mathbb{V}_{K(\phi(p))}(X)=\|\textup{grad}(f)\|^{2}
	\end{eqnarray}
	for all $p\in N$, where the left side of \eqref{nekwnkefnknk} is the variance of $X$ with respect to $K(\phi(p))$ (see 
	Theorem \ref{nceknkenknknk} for a precise statement). We give two applications: (1) The set of critical values of $f$ 
	is a subset of $\textup{spec}(f)$ (Proposition \ref{neknkendkwnkn}), and $(2)$ a point $p\in N$ is a fixed point 
	for the torus action if and only if $K(p)$ is a Dirac function, i.e., its support is a singleton 
	(Proposition \ref{ndknkefnknk}). 

	We conclude this introduction with a few remarks. If $N$ is a toric K\"{a}hler manifold with 
	real-analytic K\"{a}hler metric, and if 
	$N$ is a torification of an exponential family $\mathcal{E}$ defined over a finite set $\Omega=\{x_{1},...,x_{m+1}\}$, 
	then there is automatically a K\"{a}hler immersion 
	$j:N\to \mathbb{P}_{m}(1)$ (this is a consequence of a lifting procedure described in Section \ref{neknkenkfnkn}). 
	If $j$ is full in the sense of Calabi (see \cite{Calabi}), then one can use the rigidity theorem 
	to extend holomorphic isometries on $N$ to holomorphic isometries on 
	$\mathbb{P}_{m}(1)$, making $j$ an equivariant map. Then one can use the decompositon $f(p)=\mathbb{E}_{K(\phi(p))}(X)$ 
	to show that every K\"{a}hler function 
	$f$ on $N$ extends uniquely to a K\"{a}hler function $\hat{f}$ on $\mathbb{P}_{m}(1)$, and that the map $f\mapsto \hat{f}$ is a Poisson 
	homomorphism. Since every K\"{a}hler function on $\mathbb{P}_{m}(1)$ is of the form $f_{H}$ for some Hermitian matrix $H$, 
	we see that there is a canonical way to quantize K\"{a}hler functions on $N$. It would be interesting 
	to compare our definition of $\textup{spec}(f)$ to the spectra of the corresponding quantum operators. 
	The author wishes to explore further this topic in a future publication.\\

	The paper starts with a brief review of the rudiments of 
	symplectic geometry used throughout the paper. In Section \ref{nknkneknknfwknk}, we give a fairly detailed discussion on the
	concept of torification, with which not all readers may be familiar. The material is taken from \cite{molitor-toric}, except 
	for Section \ref{nkennkfenknk}, which is based on \cite{molitor-weyl}. In Section \ref{nfenwkeknwkndk}, we consider 
	a toric K\"{a}hler manifold and define a cocycle that measures 
	the non-equivariance of a momentum map $\moment$ with respect to the action of a certain Weyl group. In 
	Section \ref{nenefkwnknefkwndk}, we investigate the relationships between the space of K\"{a}hler functions 
	on a K\"{a}hler toric manifold $N$ and the momentum map associated to the torus action. In Section \ref{nfeqnkdekwnn}, 
	we specialize to the case in which $N$ is a torification of an exponential family $\mathcal{E}$ defined over a finite set $\Omega$ and 
	prove our main result (Theorem \ref{newnknfkenkndvknfk}). In Section \ref{nfeknkefnknkn}, we prove Formula 
	\eqref{nekwnkefnknk} and give two applications. The paper ends with an appendix in which we discuss the Lie group structure 
	of the group of holomorphic isometries of a connected K\"{a}hler manifold. 

	\textbf{}\\
	\noindent\textbf{Notation.} Let $f:M\to N$ be a smooth map between the manifolds $M$ and $N$. The derivative of $f$ at 
	$p$ is denoted by $f_{*_{p}}:T_{p}M\to T_{f(p)}N$. When $N=\mathbb{R}$ and $u\in T_{p}M$, we sometimes 
	write $f_{*_{p}}u=df_{p}(u)$ and regard the correspondence $p\mapsto df_{p}$ as a 1-form on $M$, that we denote by $df$. 
	We usually use $\Phi$ to denote a left Lie group action of a Lie group $G$ on a manifold $M$. 
	If $g\in G,$ then $\Phi_{g}$ is the map $M\to M$, $p\mapsto \Phi(g,p)$. Momentum maps are usually denoted by 
	$\moment:M\to \mathfrak{g}^{*}$. 

\section{Symplectic preliminaries}\label{nfeknkenfkndkn}

	In this section we recall the rudiments of symplectic geometry used throughout the paper. 

	Let $G$ be a Lie group with Lie algebra $\textup{Lie}(G)=\mathfrak{g}$. Let $\mathfrak{g}^{*}$ be the dual of $\mathfrak{g}$. 
	Given $g\in G$, we denote by $\textup{Ad}_{g}:\mathfrak{g}\to \mathfrak{g}$ and $\textup{Ad}_{g}^{*}:\mathfrak{g}^{*}\to 
	\mathfrak{g}^{*}$ the adjoint and coadjoint representations of $G$, respectively; they are related as follows: 
	$\langle \textup{Ad}_{g}^{*}\alpha, \xi\rangle=\langle \alpha,\textup{Ad}_{g^{-1}}\xi\rangle$, 
	where $\xi\in \mathfrak{g}$, $\alpha\in \mathfrak{g}^{*}$ and $\langle\,,\,\rangle$ 
	is the natural pairing between $\mathfrak{g}$ and $\mathfrak{g}^{*}$. 

	Let $\Phi:G\times M\to M$ be a Lie group action of $G$ on a manifold $M$. Given $g\in G$, let $\Phi_{g}$ denote
	the diffeomorphism of $M$ defined by $\Phi_{g}(p)=\Phi(g,p)$. The \textit{fundamental vector field} 
	associated to $\xi\in \mathfrak{g}$ is the vector field on $M$, denoted by $\xi_{M}$, defined by 
	$(\xi_{M})(p):=\tfrac{d}{dt}\big\vert_{0}\Phi(c(t),p),$ 
	where $p\in M$ and $c(t)$ is a smooth curve in $G$ satisfying $c(0)=e$ (neutral element) and $\dot{c}(0)=\xi$. 

	Suppose that $\omega$ is a symplectic form on $M$. The action $\Phi:G\times M\to M$ is called \textit{symplectic} 
	if $(\Phi_{g})^{*}\omega=\omega$ for all $g\in G$. The action $\Phi$ is called \textit{Hamiltonian} if it is symplectic and if 
	there exists a smooth map $\moment:M\to \mathfrak{g}^{*}$, called \textit{momentum map}, such that:
	\begin{enumerate}[(i)]
	\item $\moment \circ \Phi_{g}=\textup{Ad}_{g}^{*}\circ \moment$ for all $g\in G$, and 
	\item $\omega(\xi_{M},\,.\,)=d\moment^{\xi}(.)$ (equality of 1-forms) for all $\xi\in \mathfrak{g}$, where 
		$\moment^{\xi}:M\to \mathbb{R}$ is the function defined by $\moment^{\xi}(p):=\langle \moment(p),\xi\rangle.$
	\end{enumerate}


	When $G=\mathbb{T}^{n}=\mathbb{R}^{n}/\mathbb{Z}^{n}$ is a torus, it is convenient to identify 
	the Lie algebra of the torus $\mathbb{T}^{n}$ with $\mathbb{R}^{n}$ via the derivative at 
	$0\in \mathbb{R}^{n}$ of the quotient map $\mathbb{R}^{n}\to \mathbb{R}^{n}/\mathbb{Z}^{n}$, and 
	to identify $\mathbb{R}^{n}$ and its dual $(\mathbb{R}^{n})^{*}$ via the Euclidean metric.
	Upon these identifications, a momentum map for a Hamiltonian torus action 
	$\mathbb{T}^{n}\times M\to M$ is a map $\moment:M\to \mathbb{R}^{n}$. 
	Moreover, since the coadjoint action of a commutative group is trivial, 
	the equivariance condition reduces to $\moment \circ \Phi_{g}=\moment$ for all $g\in \mathbb{T}^{n}$. 

	A \textit{symplectic toric manifold} is a quadruplet $(M,\omega,\Phi,\moment)$, where $(M,\omega)$ 
	is a connected compact symplectic manifold of dimension $2n$, $\Phi:\mathbb{T}^{n}\times M\to M$ is an 
	effective Hamiltonian torus action and $\moment:M\to \mathbb{R}^{n}$ is an equivariant momentum map. 
	A symplectic toric manifold $(M,\omega,\Phi,\moment)$ is called a \textit{K\"{a}hler toric manifold} 
	if $M$ is a K\"{a}hler manifold whose K\"{a}hler form coincides with $\omega$, and if the torus acts 
	by holomorphic and isometric transformations on $M.$ By abuse of language, we will often say that the torus 
	action $\Phi:\mathbb{T}^{n}\times M\to M$ is a K\"{a}hler toric manifold. 

	For later reference, we collect some elementary facts from symplectic geometry. Let $\Phi:G\times M\to M$ be a 
	symplectic action of a Lie group $G$ on a symplectic manifold $(M,\omega)$. 
	Let $M^{\circ}$ be the set of points where the action is free, that is, $M^{\circ} =\{p\in M\,\,\vert\,\,\Phi(a,p)=p\Rightarrow a=e\}$. 

\begin{description}
	\item[\textbf{(F1)}] If $G$ is compact and the first de Rham cohomology group $H^{1}_{dR}(M,\mathbb{R})$ is trivial, 
		then the action is Hamiltonian (see \cite{Ortega}, Propositions 4.5.17 and 4.5.19). 
	\item[\textbf{(F2)}] If $G=T$ is a $m$-dimensional real torus and $\Phi$ is effective and Hamiltonian, 
		then $\textup{dim}\,M\geq 2m$ (see, e.g., \cite[Theorem 27.3]{Ana-lectures}).
	\item[\textbf{(F3)}] If $M$ is a symplectic toric manifold, then it is simply connected (see \cite{Delzant}). In particular, 
		$H^{1}_{dR}(M,\mathbb{R})=\{0\}$. 
	\item[\textbf{(F4)}] If $G=T$ is a $m$-dimensional real torus, 
		$M$ is connected and $\Phi$ effective, then $M^{\circ}$ is open, connected and dense in $M$ 
		(this follows from \cite[Corollary B.48]{Guillemin}).
\end{description}

\section{Torification of dually flat manifolds}\label{nknkneknknfwknk}
	In this section, we discuss the concept of torification, which is used throughout this paper. This concept 
	is a combination of two ingredients: (1) \textit{Dombrowski's construction}, 
	which implies that the tangent bundle of a dually flat manifold is naturally a K\"{a}hler manifold \cite{Dombrowski}, 
	and (2) \textit{parallel lattices}, which are used to implement torus actions. Further properties are discussed as well, 
	such as a lifting property. Examples from Information Geometry are presented. 
	The material is mostly taken from \cite{molitor-toric}.

\subsection{Dombrowski's construction.}\label{nkwwnkwnknknfkn}
	Let $M$	be a connected manifold of dimension $n$, endowed with a Riemannian metric $h$ 
	and affine connection $\nabla$ ($\nabla$ is not necessarily 
	the Levi-Civita connection). The \textit{dual connection} of $\nabla$, denoted by $\nabla^{*}$, is the only connection satisfying 
	$X(h(Y,Z))=h(\nabla_{X}Y,Z)+h(Y,\nabla^{*}_{X}Z)$ for all vector fields $X,Y,Z$ on $M$. 
	When both $\nabla$ and $\nabla^{*}$ are flat (i.e., the curvature tensor and torsion are zero), 
	we say that the triple $(M,h,\nabla)$ is a \textit{dually flat manifold}. 

	Let $\pi:TM\to M$ denote the canonical projection. Given a local coordinate system $(x_{1},...,x_{n})$ 
	on $U\subseteq M$, we can define a coordinate system $(q,r)=(q_{1},...,q_{n},r_{1},...,r_{n})$ 
	on $\pi^{-1}(U)\subseteq TM$ by letting $(q,r)(\sum_{j=1}^{n}a_{j}\tfrac{\partial}{\partial x_{j}}\big\vert_{p})=
	(x_{1}(p),...,x_{n}(p),a_{1},...,a_{n})$, where $p\in M$ and $a_{1},...,a_{n}\in \mathbb{R}$. 
	Write $(z_{1},...,z_{n})=(q_{1}+ir_{1},...,q_{n}+ir_{n})$, where $i=\sqrt{-1}$. When $\nabla$ is flat, Dombrowski 
	\cite{Dombrowski} showed that the family of complex coordinate systems $(z_{1},...,z_{n})$ 
	on $TM$ (obtained from affine coordinates on $M$) form a holomorphic atlas on $TM$. Thus, when $\nabla$ is flat, 
	$TM$ is naturally a complex manifold. If in addition $\nabla^{*}$ is flat, then $TM$ has a natural 
	K\"{a}hler metric $g$ whose local expression in the coordinates $(q,r)$ is given by $g(q,r)=
	\big[\begin{smallmatrix}
		h(x)  &  0\\
		0     &   h(x)
	\end{smallmatrix}
	\big]$, where $h(x)$ is the matrix representation of $h$ in the affine coordinates $x=(x_{1},...,x_{n})$. 
	It follows that the tangent bundle of a dually flat manifold is naturally a K\"{a}hler manifold. In this paper, 
	we will refer to this K\"{a}hler structure as the \textit{K\"{a}hler structure associated to Dombrowski's 
	construction}.

\subsection{Parallel lattices.}\label{nekwnkefkwnkk}

	Let $(M,h,\nabla)$ be a dually flat manifold of dimension $n$. 
	A subset $L\subset TM$ is said to be a \textit{parallel lattice} with respect to $\nabla$ if there are $n$ vector 
	fields $X_{1},...,X_{n}$ on $M$ that are parallel with respect to $\nabla$ and such that: $(i)$ $\{X_{1}(p),...,X_{n}(p)\}$ is a basis 
	for $T_{p}M$ for every $p\in M$, and $(ii)$ $L=\{k_{1}X_{1}(p)+...+k_{n}X_{n}(p)\,\,\vert\,\,k_{1},...,k_{n}\in \mathbb{Z},\,\,p\in M\}$. 
	In this case, we say that the frame $X=(X_{1},...,X_{n})$ is a 
	\textit{generator} for $L$. We will denote the set of generators for $L$ by 
	$\textup{gen}(L)$. 

	Given a parallel lattice $L\subset TM$ with respect to $\nabla$, and $X\in \textup{gen}(L)$ , 
	we will denote by $\Gamma(L)$ the set of 
	transformations of $TM$ of the form $u\mapsto u+k_{1}X_{1}+...+k_{n}X_{n}$, where 
	$u\in TM$ and $k_{1},...,k_{n}\in \mathbb{Z}$. The group
	$\Gamma(L)$ is independent of the choice of $X\in \textup{gen}(L)$ and is isomorphic to $\mathbb{Z}^{n}$. 
	Moreover, the natural action of $\Gamma(L)$ 
	on $TM$ is free and proper. Thus the quotient $M_{L}=TM/\Gamma(L)$ is a 
	smooth manifold and the quotient map $q_{L}:TM\to M_{L}$ is a covering map 
	whose Deck transformation group is $\Gamma(L)$. Since 
	$\pi\circ \gamma=\pi$ for all $\gamma\in \Gamma(L)$, 
	there exists a surjective submersion $\pi_{L}:M_{L}\to M$ such that $\pi=\pi_{L}\circ q_{L}$.

	Let $\mathbb{T}^{n}=\mathbb{R}^{n}/\mathbb{Z}^{n}$ be the $n$-dimensional torus. 
	Given $t=(t_{1},...,t_{n})\in \mathbb{R}^{n}$, we will denote 
	by $[t]=[t_{1},...,t_{n}]$ the corresponding equivalence class in $\mathbb{R}^{n}/\mathbb{Z}^{n}$. 
	Given $X\in \textup{gen}(L)$, we will denote by 
	$\Phi_{X}:\mathbb{T}^{n}\times M_{L}\to M_{L}$ the effective torus action defined by 
	\begin{eqnarray}\label{nkdnknewknek}
		\Phi_{X}([t],q_{L}(u))=q_{L}(u+t_{1}X_{1}+...+t_{n}X_{n}). 
	\end{eqnarray}

	The manifold $M_{L}=TM/\Gamma(L)$ is naturally a K\"{a}hler manifold (this follows from the fact that 
	every $\gamma\in \Gamma(L)$ is a holomorphic and isometric map with respect the K\"{a}hler structure 
	associated to Dombrowski's construction). Moreover, for each $a\in \mathbb{T}^{n}$, the map 
	$(\Phi_{X})_{a}:M_{L}\to M_{L}$, $p\mapsto \Phi_{X}(a,p)$ is a holomorphic isometry. 

\subsection{Torification.}\label{nfkenkfejdefdknkn} 
	Let $(M,h,\nabla)$ be a connected dually flat manifold of dimension $n$ and $N$ a connected K\"{a}hler 
	manifold of complex dimension $n$, equipped with an effective holomorphic and isometric torus 
	action $\Phi:\mathbb{T}^{n}\times N\to N$. Let $N^{\circ}$ denote the set of points $p\in N$ 
	where $\Phi$ is free. 

\begin{definition}[\textbf{Torification}]
	We shall say that $N$ is a \textit{torification} of $M$ 
	if there exist a parallel lattice $L\subset TM$ with respect to 
	$\nabla$, $X\in \textup{gen}(L)$ and a holomorphic and isometric diffeomorphism $F:M_{L}\to N^{\circ}$
	satisfying $F\circ (\Phi_{X})_{a}=\Phi_{a}\circ F$ for all $a\in \mathbb{T}^{n}$. 
\end{definition}

	By abuse of language, we will often say that the torus action $\Phi:\mathbb{T}^{n}\times N\to N$ is a torification of $M$. 
	We shall say that a K\"{a}hler manifold $N$ is \textit{regular} if it is connected, 
	simply connected, complete and if the K\"{a}hler metric is real analytic. 
	A torification $\Phi:\mathbb{T}^{n}\times N\to N$ is said to be \textit{regular} if $N$ is regular. 
	In this paper, we are mostly interested 
	in regular torifications. 
\begin{remark}\label{ndkdnknknknk}
	A torification $N$ is not necessarily a K\"{a}hler toric manifold. For example, $N$ may not be compact 
	(see Proposition \ref{nfeknwknwknk}). 
\end{remark}

	Two torifications $\Phi:\mathbb{T}^{n}\times N\to N$ and $\Phi':\mathbb{T}^{n}\times N'\to N'$ 
	of the same connected dually flat manifold $(M,h,\nabla)$ are said to be \textit{equivalent} if there exists 
	a K\"{a}hler isomorphism $f:N\to N'$ and a Lie group isomorphism $\rho:\mathbb{T}^{n}\to \mathbb{T}^{n}$ such that 
	$f\circ \Phi_{a}=\Phi'_{\rho(a)}\circ f$ for all $a\in \mathbb{T}^{n}$.
\begin{theorem}[\textbf{Equivalence of regular torifications}]\label{newdnkekfwndknk}
	Regular torifications of a connected dually flat manifold $(M,h,\nabla)$ are equivalent. 
\end{theorem}

	A connected dually flat manifold $(M,h,\nabla)$ is said to be \textit{toric} if it has a regular torification $N$. 
	In this case, we will often refer to $N$ as ``the regular torification of $M$", and keep in mind that it is 
	only defined up to an equivariant K\"{a}hler isomorphism and reparametrization of the torus. 

	For later reference, we give the following technical definition.
\begin{definition}\label{fjenkrkefnkn}
	Suppose $\Phi:\mathbb{T}^{n}\times N\to N$ is a torification of $(M,h,\nabla)$. 
	\begin{enumerate}[(1)]
	\item A \textit{toric parametrization} is a triple $(L,X,F)$, where $L\subset TM$ is parallel lattice with respect to $\nabla$, 
		$X\in \textup{gen}(L)$ and $F:M_{L}\to N^{\circ}$ is a holomorphic and isometric diffeomorphism satisfying 
		$F\circ (\Phi_{X})_{a}=\Phi_{a}\circ F$ for all $a\in \mathbb{T}^{n}$.
	\item  Let $\tau:TM\to N^{\circ}$ and $\kappa:N^{\circ}\to M$ be smooth maps. We say that 
		the pair $(\tau,\kappa)$ is a \textit{toric factorization} 
		if there exists a toric parametrization $(L,X,F)$ such that $\tau=F\circ q_{L}$ and $\kappa=\pi_{L}\circ F^{-1}$. 
	In this case, we say that $(\tau, \kappa)$ is \textit{induced by the toric parametrization} $(L,X,F)$. Note that $\pi=\kappa\circ \tau.$
\item We say that $\kappa:N^{\circ}\to M$ is the \textit{compatible projection induced by the toric parametrization} $(L,X,F)$ 
	if there exists a map $\tau:TM\to N^{\circ}$ such that $(\tau,\kappa)$ is the toric factorization induced by $(L,X,F)$. 
	When it is not necessary to mention $(L,X,F)$ explicitly, we just say 
	that $\kappa$ is a \textit{compatible projection}. Analogously, one defines a \textit{compatible covering map} $\tau:TM\to N^{\circ}$. 
\end{enumerate}
\end{definition}
	\noindent By abuse of language, we will often say that the formula $\pi=\kappa\circ \tau$ is a toric factorization. 

	If $\Phi:\mathbb{T}^{n}\times N\to N$ is a regular torification of $M$, and 
	if $\kappa,\kappa':N^{\circ}\to M$ are compatible projections, 
	then there is a holomorphic isometry $\varphi:N\to N$ and a Lie group isomorphism 
	$\rho:\mathbb{T}^{n}\to \mathbb{T}^{n}$ such that $\varphi\circ \Phi_{a}=\Phi_{\rho(a)}\circ \varphi$ 
	for all $a\in \mathbb{T}^{n}$ and $\kappa'=\kappa\circ \varphi$.

\subsection{Lifting procedure.}\label{neknkenkfnkn}
	Let $\Phi:\mathbb{T}^{n}\times N\to N$ and $\Phi':\mathbb{T}^{d}\times N'\to N'$ be torifications of the dually flat manifolds 
	$(M,h,\nabla)$ and $(M',h',\nabla')$, respectively. 

\begin{definition}
	Let $f:M\to M'$ and $m:N\to N'$ be smooth maps. We say that $m$ is a \textit{lift of} $f$ 
	if there are compatible covering maps $\tau:TM\to N^{\circ}$ and $\tau':TM'\to (N')^{\circ}$ such that 
	$m\circ \tau= \tau'\circ f_{*}$. In this case, we say that $m$ \textit{is a lift of} $f$ 
	\textit{with respect to} $\tau$ and $\tau'$.  
\end{definition}

	If $m$ is a lift of $f$ with respect to $\tau$ and $\tau'$, and if $\pi=\kappa\circ \tau$ and $\pi'=\kappa'\circ \tau'$ are toric factorizations, then 
	$\kappa'\circ m =f\circ \kappa.$ Therefore the following diagram commutes:

\begin{eqnarray}\label{knkdfnknkndknk}
	\begin{tikzcd}
		TM \arrow{rrr}{\displaystyle f_{*}} \arrow[dd,swap,"\displaystyle\pi"] \arrow{rd}{\displaystyle \tau} &    &   & 
			TM' \arrow[swap]{ld}{\displaystyle \tau'} \arrow[dd,"\displaystyle\pi'"]  \\
		& N^{\circ} \arrow{r}{\displaystyle m} \arrow{ld}{\displaystyle \kappa}  
			& (N')^{\circ} \arrow[swap]{dr}{\displaystyle \kappa'}  & \\
		M \arrow[swap]{rrr}{\displaystyle f}  &     && M'
	\end{tikzcd}
\end{eqnarray}

	If $m:N\to N'$ is a lift of $f$, then $m$ is a K\"{a}hler immersion if and only if $f:(M,\nabla)\to (M',\nabla')$ is an affine immersion 
	satisfying $f^{*}h'=h$. 
	In this case, there exists a unique Lie group homomorphism 
	$\rho:\mathbb{T}^{n}\to \mathbb{T}^{d}$ with finite kernel such that $m\circ \Phi_{a}=\Phi_{\rho(a)}'\circ m$ 
	for all $a\in \mathbb{T}^{n}$.

\begin{theorem}[\textbf{Existence of lifts}]\label{ncnkwknknk}
	Suppose $\Phi:\mathbb{T}^{n}\times N\to N$ and $\Phi':\mathbb{T}^{d}\times N'\to N'$ are 
	regular torifications of $(M,h,\nabla)$ and $(M',h',\nabla')$, respectively. 
	Let $\tau:TM\to N^{\circ}$ and $\tau':TM'\to (N')^{\circ}$ be compatible covering maps. Then every isometric affine immersion 
	$f:M\to M'$ has a unique lift $m:N\to N'$ with respect to $\tau$ and $\tau'$. 
\end{theorem}

	Let $\textup{Diff}(M,h,\nabla)$ be the group of isometries of $(M,h)$ that are also affine with respect to $\nabla$, 
	and let $\textup{Aut}(N,g)$ be the group of holomorphic isometries of $N$. It follows from Theorem \ref{ncnkwknknk} that 
	every $\psi\in \textup{Diff}(M,h,\nabla)$ has a lift with respect to $\tau$ and $\tau'=\tau$, that we will denote by 
	$\textup{lift}_{\tau}(\psi)$. The following result is a consequence of \cite[Proposition 9.9]{molitor-toric}.

\begin{proposition}\label{nfeknwkneknk}
	The map $\textup{Diff}(M,h,\nabla)\to \textup{Aut}(N,g)$, $\psi\mapsto \textup{lift}_{\tau}(\psi)$ is a group homomorphism. 
\end{proposition}

\subsection{Examples from Information Geometry.}\label{nknkefnknk} 

	The concept of torification was motivated, in the first place, by the 
	connection between K\"{a}hler geometry, Information Geometry and Quantum Mechanics.
	In this section, we illustrate this connection with a few examples. The reader interested in 
	Information Geometry may consult \cite{Jost2,Amari-Nagaoka,Murray}.

\begin{definition}\label{def:5.1}
	A \textit{statistical manifold} is a pair $(S,j)$, where $S$ is a manifold and where $j$ 
	is an injective map from $S$ to the space of all probability density functions $p$ 
	defined on a fixed measure space $(\Omega,dx)$: 
	\begin{eqnarray*}
		j:S\hookrightarrow\Bigl\{ p:\Omega\to\mathbb{R}\;\Bigl|\; p\:\textup{is measurable, }p\geq 
		0\textup{ and }\int_{\Omega}p(x)dx=1\Bigr\}.
	\end{eqnarray*}
\end{definition}

	If $\xi=(\xi_{1},...,\xi_{n})$ is a coordinate system on a statistical manifold $S$, then we shall 
	indistinctly write $p(x;\xi)$ or $p_{\xi}(x)$ for the probability density function determined by $\xi$.

	Given a ``reasonable" statistical manifold $S$, it is possible to define a metric $h_{F}$ and a 
	family of connections $\nabla^{(\alpha)}$ on $S$ $(\alpha\in\mathbb{R})$ in the following way: 
	for a chart $\xi=(\xi_{1},...,\xi_{n})$ of $S$, define
	\begin{alignat*}{1}
		\bigl(h_F\bigr)_{\xi}\bigl(\partial_{i},\partial_{j}\bigr) & :=
		\mathbb{E}_{p_{\xi}}\bigl(\partial_{i}\ln\bigl(p_{\xi}\bigr)\cdotp\partial_{j}\ln\bigl(p_{\xi}\bigr)\bigr),
		\nonumber\\
		\Gamma_{ij,k}^{(\alpha)}\bigl(\xi\bigr) & :=
		\mathbb{E}_{p_{\xi}}\bigl[\bigl(\partial_{i}\partial_{j}\ln\bigl(p_{\xi}\bigr)
		+\tfrac{1-\alpha}{2}\partial_{i}\ln\bigl(p_{\xi}\bigr)\cdotp\partial_{j}\ln\bigl(p_{\xi}\bigr)\bigr)
		\partial_{k}\ln\bigl(p_{\xi}\bigr)\bigr],\label{eq:48}\nonumber
	\end{alignat*}
	where $\mathbb{E}_{p_{\xi}}$ denotes the mean, or expectation, with respect to the probability 
	$p_{\xi}dx$, and where $\partial_{i}$ is a shorthand for $\tfrac{\partial}{\partial\xi_{i}}$. 
	In the formulas above, it is assumed that the function $p_{\xi}(x)$ is smooth with respect to 
	$\xi$ and that the expectations are finite. When the first formula above defines a smooth metric $h_{F}$ on $S$, 
	it is then called \textit{Fisher metric}. In this case, the 
	$\Gamma_{ij,k}^{(\alpha)}$'s define a connection $\nabla^{(\alpha)}$ on $S$ via the formula 
	$\Gamma_{ij,k}^{(\alpha)}(\xi)=(h_F)_{\xi}(\nabla_{\partial_{i}}^{(\alpha)}\partial_{j},\partial_{k})$, 
	which is called the \textit{$\alpha$-connection}.

%
	Among the $\alpha$-connections, the $1$-connection is particularly important and 
	is usually referred to as the \textit{exponential connection}, also denoted by $\nabla^{(e)}$. 
	In this paper, we will only consider statistical manifolds $S$ for which the Fisher metric $h_{F}$ and
	exponential connection $\nabla^{(e)}$ are well defined.


	We now recall the definition of an exponential family. 
\begin{definition}\label{def:5.3} 
	An \textit{exponential family} $\mathcal{E}$ on a measure space $(\Omega,dx)$ is a set of probability 
	density functions $p(x;\theta)$ of the form 
	$p(x;\theta)=\exp\big\{ C(x)+\sum_{i=1}^{n}\theta_{i}F_{i}(x)-\psi(\theta)\big\},$ 
	where $C,F_1...,F_n$ are measurable functions on $\Omega$, $\theta=(\theta_{1},...,\theta_{n})$ 
	is a vector varying in an open subset $\Theta$ of $\mathbb{R}^{n}$ and where $\psi$ is a function defined on $\Theta$.
\end{definition}
	In the above definition, it is assumed that the family of functions $\{1,F_1,...,$ $F_n\}$ 
	is linearly independent, so that the map $p(x,\theta)\mapsto\theta$ becomes a bijection, hence defining a global 
	chart for $\mathcal{E}$. The parameters $\theta_{1},...,\theta_{n}$ are called the 
	\textit{natural} or \textit{canonical parameters} of the exponential family $\mathcal{E}$. 
	
\begin{example}[\textbf{Binomial distribution}]\label{exa:5.6}
	Let $\mathcal{B}(n)$ be the set of Binomial distributions $p(k)=\binom{n}{k}q^{k}\bigl(1-q\bigr)^{n-k}$, $q\in (0,1)$, 
	defined over $\Omega=\{0,...,n\}$. It is a 1-dimensional exponential family, because 
	$p(k)=\exp\big\{ C(k)+\theta F(k)-\psi(\theta)\big\}$, where $\theta=\ln\big(\frac{q}{1-q}\big)$, $C(k)= \ln\binom{n}{k}$, 
	$F(k)=k$, $k\in \Omega$, and $\psi(\theta)=n\ln\big(1+\exp(\theta)\big)$. 
\end{example}
\begin{example}[\textbf{Categorical distribution}]\label{exa:5.5}
	Let $\Omega=\{ x_{1},...,x_{n}\}$ be a finite set and let $\mathcal{P}_{n}^{\times}$ be the set of maps
	$p:\Omega\to \mathbb{R}$ satisfying $p(x)>0$ for all $x\in \Omega$ and $\sum_{x\in \Omega}p(x)=1$. Then 
	$\mathcal{P}_{n}^{\times}$ is an exponential family of dimension $n-1$. Indeed, elements in $\mathcal{P}_{n}^{\times}$ 
	can be parametrized as follows: $p(x;\theta)=\exp\big\{\sum_{i=1}^{n-1}\theta_{i}F_{i}(x)-\psi(\theta)\big\}$,
	where $x\in \Omega$, $\theta=(\theta_{1},...,\theta_{n-1})\in\mathbb{R}^{n-1},$ 
	$F_{i}(x_{j})=\delta_{ij}$ \textup{(Kronecker delta)} and $\psi(\theta)=\ln\big(1+\sum_{i=1}^{n-1}e^{\theta_{i}}\big)$. 
%
%
%
%
\end{example}

\begin{example}[\textbf{Poisson distribution}]
	A Poisson distribution is a distribution over $\Omega=\mathbb{N}=\{0,1,...\}$ of the form 
	$p(k;\lambda)=e^{-\lambda}\tfrac{\lambda^{k}}{k!}$, 
	where $k\in \mathbb{N}$ and $\lambda>0$. Let $\mathscr{P}$ denote the set of all Poisson distributions 
	$p(\,.\,;\lambda)$, $\lambda>0$. The set $\mathscr{P}$ is an exponential family, because 
	$p(k,\lambda)=\textup{exp}\big(C(k)+F(k)\theta-\psi(\theta)\big),$ where 
	$C(k)=-\ln(k!)$, $F(k)=k$, $\theta=\ln(\lambda)$, and $\psi(\theta)=\lambda=e^{\theta}.$
\end{example}

	Under mild assumtions, it can be proved that an exponential family $\mathcal{E}$ endowed with 
	the Fisher metric $h_{F}$ and exponential connection $\nabla^{(e)}$ is a dually flat manifold 
	(see \cite{Amari-Nagaoka}). This is the case, for example, if $\Omega$ is a finite set endowed 
	with the counting measure (see \cite{shima}, Chapter 6). In the sequel, we will always regard 
	an exponential family  $\mathcal{E}$ as a dually flat manifold. 

	Since an exponential family is dually flat, it is natural to ask whether it is toric. 
	Below we describe three examples. Let $\mathbb{P}_{n}(c)$ be the complex projective space of complex dimension $n$, 
	endowed with the Fubini-Study metric normalized 
	in such a way that the holomorphic sectional curvature is $c>0$. Let $\Phi_{n}$ be the action of 
	$\mathbb{T}^{n}$ on $\mathbb{P}_{n}(c)$ defined by 
	\begin{eqnarray}\label{neknknknknfdkn}
		\Phi_{n}([t],[z])= [e^{2i\pi t_{1}}z_{1},...,e^{2i\pi t_{n}}z_{n},z_{n+1}]. 
		\,\,\,\,\,\,\,(\textup{homogeneous coordinates})
	\end{eqnarray}
\begin{proposition}\label{nfeknwknwknk}
	In each case below, the torus action is a regular torification of the indicated exponential family $\mathcal{E}$. \\

	\begin{tabular}{lll}
		$(a)$ &  $\mathcal{E}=\mathcal{P}_{n+1}^{\times}$,    &    
						$\Phi_{n}:\mathbb{T}^{n}\times \mathbb{P}_{n}(1)\to \mathbb{P}_{n}(1)$.\\[0.4em]
		$(b)$ & $\mathcal{E}=\mathcal{B}(n)$,                &    $\Phi_{1}:\mathbb{T}^{1}\times 
								\mathbb{P}_{1}(\tfrac{1}{n})\to \mathbb{P}_{1}(\tfrac{1}{n})$.\\[0.4em]
		$(c)$ & $\mathcal{E}=\mathscr{P}$,                   & $\mathbb{T}^{1}\times \mathbb{C}\to \mathbb{C}$, $([t],z)\mapsto e^{2i\pi t}z$.	
	\end{tabular}
\end{proposition}

	\noindent For a proof and more examples, see \cite{molitor-toric}. 

	In what follows, we will identify the tangent bundle of $\mathcal{E}\in \{\mathcal{P}_{n+1}^{\times},\mathcal{B}(n),
	\mathscr{P}\}$ 
	with $\mathbb{C}^{\textup{dim}(\mathcal{E})}$ via the map $T\mathcal{E}\to \mathbb{C}^{\textup{dim}(\mathcal{E})}$, 
	$\sum\dot{\theta}_{k}\tfrac{\partial}{\partial \theta_{k}}\big\vert_{p(.,\theta)}\mapsto 
	(z_{1},...,z_{\textup{dim}(\mathcal{E})})$, where the $\theta_{k}$'s are natural parameters
	and $z_{k}=\theta_{k}+i\dot{\theta}_{k}$.

\begin{proposition}[\textbf{Complement of Proposition \ref{nfeknwknwknk}}]\label{njewndknknknk}
	In each case below, the pair $(\tau,\kappa)$ is a toric factorization of the indicated exponential family  $\mathcal{E}$. 
	\begin{enumerate}[(a)]
	\item $\mathcal{E}=\mathcal{P}_{n+1}^{\times},$ \,\,\,\,\,\, 
		$\mathbb{C}^{n}=T\mathcal{P}_{n+1}^{\times}\overset{\tau}{\longrightarrow} \mathbb{P}_{n}(1)^{\circ}
		\overset{\kappa}{\longrightarrow} \mathcal{P}_{n+1}^{\times},$
		\begin{eqnarray*}
			\tau(z)=\big[e^{z_{1}/2},...,e^{z_{n}/2},1\,\big], 
			\,\,\,\,\,\,\,\,\,\,\,\,\,\,\,\,\,\,\,\,\kappa([z])(x_{k})=
			\dfrac{|z_{k}|^{2}}{|z_{1}|^{2}+...+|z_{n+1}|^{2}},\,\,\,k=1,2...,n+1.
		\end{eqnarray*}
		\item $\mathcal{E}=\mathcal{B}(n),$ \,\,\,\,\,\, 
			$\mathbb{C}=T\mathcal{B}(n)\overset{\tau}{\longrightarrow} \mathbb{P}_{1}(\tfrac{1}{n})^{\circ}
			\overset{\kappa}{\longrightarrow}\mathcal{B}(n),$
			\begin{eqnarray*}
				\tau(z)	= \big[e^{z/2},1\big], \,\,\,\,\,\,\,\,\,\,\,\,\,\,\,\,\,\,\,\,
				\kappa([z_{1},z_{2}])(k)=
				\binom{n}{k}\dfrac{|z_{1}^{k}z_{2}^{n-k}|^{2}}{(|z_{1}|^{2}+|z_{2}|^{2})^{n}},\,\,\,k=0,1,2...,n.
			\end{eqnarray*}
		\item $\mathcal{E}=\mathscr{P},$ \,\,\,\,\,\,  
			$\mathbb{C}=T\mathscr{P}\overset{\tau}{\longrightarrow} 
			\mathbb{C}^{*}\overset{\kappa}{\longrightarrow} \mathscr{P}$,
			\begin{eqnarray*}
				\tau(z)	= 2e^{z/2}, \,\,\,\,\,\,\,\,\,\,\,\,\,\,\,\,\,\,\,\,
				\kappa(z)(k)=e^{-|z|^{2}/4}\dfrac{1}{k!}\bigg(\dfrac{|z|^{2}}{4}\bigg)^{k},\,\,\,k=0,1,2...  
			\end{eqnarray*}
	\end{enumerate}
\end{proposition}

%
%

\subsection{Weyl groups.}\label{nkennkfenknk}  In this section, we discuss Weyl groups in the context of torification. We begin 
	by recalling some basic definitions and properties. 
	Let $G$ be a compact Lie group with Lie algebra $\mathfrak{g}$. A subgroup $T$ of $G$ is called a \textit{torus}
	if it is compact, connected and Abelian. A torus $T\subset G$ is said to be \textit{maximal} 
	if $T\subseteq T'$, with $T'$ a torus, implies $T=T'$. 
	Let $G^{0}$ denote the identity component of $G$, and let $T$ be a maximal torus in $G^{0}$
	with Lie algebra $\mathfrak{t}$. The \textit{Weyl group} of $G^{0}$ 
	with respect to $T$, denoted by $W(T)$, is the quotient group $N(T)/T$, where 
	$N(T)=\{g\in G^{0}\,\,\vert\,\,gTg^{-1} = T\}$ is the normalizer of $T$ in $G^{0}$ (since $T$ is normal 
	in $N(T)$, the quotient $N(T)/T$ is naturally a group). If $T$ and $T'$ are maximal tori in $G^{0}$, then the groups 
	$W(T)$ and $W(T')$ are isomorphic. Thus the Weyl group is independent of the choice of a maximal torus. 
	The Weyl group is a finite group that acts on $T$ and $\mathfrak{t}$ as follows: if $n\in N(T)$ is a representative of 
	$w\in W(T)$, then 
        \begin{eqnarray*}
		w\cdot a= nan^{-1} \,\,\,\,\,\,\,\,\textup{and} \,\,\,\,\,\,\,\,w\cdot \xi = \textup{Ad}_{n}\xi, 
	\end{eqnarray*}
	where $a\in T$ and $\xi\in \mathfrak{t}$. Moreover, the following properties hold (see, e.g., \cite{kolk}, Theorem 3.7.1(iv) and 
	\cite{bott}, p.74):
	\begin{description}
		\item[(P1)] $\mathfrak{g}=\cup_{g\in G^{0}}\textup{Ad}_{g}(\mathfrak{t})$. 
		\item[(P2)] Let $\xi\in \mathfrak{g}$. 
			If $\xi=\textup{Ad}_{g}\eta=\textup{Ad}_{g'}\eta'$, 
			where $g,g'\in G^{0}$ and $\eta,\eta'\in \mathfrak{t}$, then 
			there is $w\in W$ such that $\eta=w\cdot \eta'$. 
	\end{description}
	Property (P1) is sometimes referred to as the \textit{infinitesimal conjugation Theorem}, since it is an 
	infinitesimal version of the classical conjugation Theorem of Cartan for maximal tori. 

	For our purposes, it is convenient to consider a sightly more general definition of the Weyl group, which due to Segal
	\cite{segal}. A closed subgroup $S\subseteq G$ is called a \textit{Cartan subgroup}
	if it contains an element whose powers are dense in $S$, and if $S$ is of finite index in its normalizer 
	$N(S)=\{g\in G\,\,\vert\,\,gSg^{-1}=S\}$. The finite group $W(S)=N(S)/S$ is, by definition, the \textit{Weyl group} of $S$ 
	(in the sense of Segal). When $G=G^{0}$, that is, when $G$ is connected, 
	Cartan subgroups are precisely the maximal tori of $G$, and thus, Segal's definition 
	generalizes the classical one to the nonconnected case. Note, however, that Segal's definition \textit{depends} 
	on the choice of a Cartan subgroup. 

	We now discuss Weyl groups in the context of torification. Let $\Phi:\mathbb{T}^{n}\times N\to N$ be a 
	K\"{a}hler toric manifold, with K\"{a}hler metric $g$. Let $\textup{Aut}(N,g)$ be the group of holomorphic isometries of $N$
	endowed with the compact-open topology; it is a compact Lie group whose natural action on $N$ is smooth 
	(see Proposition \ref{nfekwnknefknk}). We assume that $N$, equipped with $\Phi$, is a regular torification of a dually flat 
	manifold $(M,h,\nabla)$. Let $\textup{Diff}(M,h,\nabla)$ be the group of affine diffeomorphisms of $(M,\nabla)$ 
	that are also isometries with respect to $h$. Let $\tau:TM\to N^{\circ}$ be a compatible covering map 
	(see Definition \ref{fjenkrkefnkn}). Given $\psi\in \textup{Diff}(M,h,\nabla)$, we will denote by 
	$\textup{lift}_{\tau}(\psi)$ the lift of $\psi$ with respect to $\psi$ (see Section \ref{neknkenkfnkn}). 
	Let $S\subset \textup{Aut}(N,g)$ be the image of $\mathbb{T}^{n}$ under the map 
	$a\mapsto \Phi_{a}$. 	

	The following properties are proven in \cite{molitor-weyl}.
	\begin{description}
		\item[(P3)] $S$ is a maximal torus and a Cartan subgroup of $\textup{Aut}(N,g)$. 
	\end{description}

	Let $W(S)=N(S)/S$ be the Weyl group of $\textup{Aut}(N,g)$ associated to $S$ (in the sense of Segal). 
	We will denote by $[\psi]$ the equivalence class of $\psi\in N(S)$ in $W(S)$. 

	\begin{description}
		\item[(P4)] $\textup{Diff}(M,h,\nabla)\to W(S),\,\,\,\, \psi\mapsto \big[\textup{lift}_{\tau}(\psi)\big]$
		is a group isomorphism. 
	\end{description}

	Let $\textup{Proj}$ denote the set of compatible projections $\kappa:N^{\circ}\to M$ (see Definition \ref{fjenkrkefnkn}). 

	\begin{description}
		\item[(P5)] The map $\textup{Diff}(M,h,\nabla)\times \textup{Proj}\to \textup{Proj}$,
			$(\psi,\kappa) \mapsto \psi\circ \kappa$, is a free and transitive group action. 
	\end{description}

	We now specialize to the case in which $M=\mathcal{E}$ is an exponential family defined over a finite set 
	$\Omega=\{x_{1},...,x_{m}\}$, with elements of the form 
	$p(x;\theta)=e^{C(x)+\langle \theta,F(x)\rangle-\psi(\theta)},$
	where $\theta\in \mathbb{R}^{n}$, 
	$C:\Omega\to \mathbb{R}$, $F=(F_{1},...,F_{n}):\Omega\to \mathbb{R}^{n}$, 
	$\langle u,v\rangle=u_{1}v_{1}+...+u_{n}v_{n}$ is the Euclidean pairing in 
	$\mathbb{R}^{n}$ and $\psi:\mathbb{R}^{n}\to \mathbb{R}$. 
	It is assumed that the functions $1,F_{1},...,F_{n}$ are linearly independent. 
	Let $h_{F}$ and $\nabla^{(e)}$ be, respectively, the Fisher metric and exponential connection on $\mathcal{E}$.
	We continue to assume that $\Phi:\mathbb{T}^{n}\times N\to N$ is a K\"{a}hler toric manifold and a regular 
	torification of $(\mathcal{E},h_{F},\nabla^{(e)})$. Let $\mathbb{S}_{m}$ be the group of permutations of $\{1,...,m\}$. 
	Given $\sigma\in \mathbb{S}_{m}$, let $T_{\sigma}$ denote the diffeomorphism of $\mathcal{E}$ defined 
	by $T_{\sigma}(p)(x_{k})=p(x_{\sigma(k)})$, where $p\in \mathcal{E}$ and $k\in \{1,...,m\}$. 
	We shall say that a diffeomorphism $\varphi$ of $\mathcal{E}$ is a \textit{permutation} 
	if there is $\sigma\in \mathbb{S}_{m}$ such that $\varphi=T_{\sigma}$. 
	Let $\textup{Perm}(\mathcal{E})$ be the group of permutations of $\mathcal{E}$. 
	When $F$ is injective, it is easy to see that the map $\textup{Perm}(\mathcal{E})\to \mathbb{S}_{m}$, 
	$T_{\sigma}\to \sigma^{-1}$ defines a group isomorphism from 
	$\textup{Perm}(\mathcal{E})$ to a subgroup of $\mathbb{S}_{m}$. In this case, we will often 
	regard $\textup{Perm}(\mathcal{E})$ as a subgroup of $\mathbb{S}_{m}$. 

	\begin{description}
	\item[(P6)] If $F:\Omega\to \mathbb{R}^{n}$ is injective, then 
		$\textup{Diff}(\mathcal{E},h_{F},\nabla^{(e)})=\textup{Perm}(\mathcal{E})$. 
	\end{description}

	In the case in which $F$ is injective, (P4) and (P6) imply that 
	the Weyl group of $\textup{Aut}(N,g)$ associated to $S$ is isomorphic to the group of 
	permutations of $\mathcal{E}$. For instance, the group of permutations of $\mathcal{P}_{m}^{\times}$ (see Example \ref{exa:5.5}) 
	is obviously isomorphic to the group $\mathbb{S}_{m}$. The complex projective space $\mathbb{P}_{m-1}(c)$ 
	of complex dimension $m-1$ and holomorphic sectional curvature $c=1$ is a regular torification 
	of $\mathcal{P}_{m}^{\times}$. The group of holomorphic isometries of $\mathbb{P}_{m-1}(c)$ is the projective 
	unitary group $\textup{PU}(m)$, which is connected. Therefore the Weyl group of 
	$\textup{PU}(m)$ is isomorphic to $\mathbb{S}_{m}$ (see also Proposition \ref{neknknekfnknn}(ii)). 

	Let $\mathcal{A}_{\mathcal{E}}$ be the vector space of real-valued functions $X$ on $\Omega$ spanned by $F_{1},...,F_{n}$ 
	and the constant function 1.

	\begin{description}
	\item[(P7)] If $F:\Omega\to \mathbb{R}^{n}$ is injective, then the formula 
		$(\sigma\cdot X)(x_{k})=X(x_{\sigma^{-1}(k)})$ defines a linear action of $\textup{Perm}(\mathcal{E})$ on 
		$\mathcal{A}_{\mathcal{E}}$. 
	\end{description}

	Finally, it follows from (P5) and (P6) that: 

\begin{description}
	\item[(P8)] Suppose $F:\Omega\to \mathbb{R}^{n}$ injective. If $\kappa$ and $\kappa'$ are compatible 
		projections, then there is $\sigma\in \textup{Perm}(\mathcal{E})\subset\mathbb{S}_{m}$ such that 
		$\kappa(p)(x_{k})=\kappa'(p)(x_{\sigma(k)})$ for all $p\in N^{\circ}$ and all $k\in \{1,...,m\}$. 
\end{description}

\section{Non-equivariance one-cocycle}\label{nfenwkeknwkndk}

	The purpose of this section is to define a cocycle that measures the non-equivariance of 
	the momentum map $\moment$ of a K\"{a}hler toric manifold with respect to the natural action of a 
	certain Weyl group (see Definition \ref{neknwkenknk}).

	Let $\Phi:\mathbb{T}^{n}\times N\to N$ be a K\"{a}hler toric manifold, with K\"{a}hler metric $g$ and momentum map 
	$\moment:N\to \mathbb{R}^{n}$. Let $\textup{Aut}(N,g)$ be the group of holomorphic isometries of $N$ endowed with the
	compact-open topology; it is a compact Lie group whose natural action on $N$ is smooth 
	(see Proposition \ref{nfekwnknefknk}). Below, we will use the same notation $\Phi$ to denote either the map $\mathbb{T}^{n}\times N\to N$ 
	or the corresponding group homomorphism $\mathbb{T}^{n}\to \textup{Aut}(N,g)$. Let $\mathfrak{aut}(N,g)$ be the Lie 
	algebra of $\textup{Aut}(N,g)$. It follows immediately from (F1) and (F3) that if $G$ 
	is a closed subgroup of $\textup{Aut}(N,g)$, then its natural action on $N$ 
	is Hamiltonian (see also Proposition \ref{ndkndkneksdnknk}). 
	In particular, the natural action of $\textup{Aut}(N,g)$ on $N$ has an equivariant momentum map 
	that we will denote by $\overline{\moment}:N\to \mathfrak{aut}(N,g)^{*}$. 

\begin{lemma}\label{nefkkwknefknwknk}
	The map $\Phi:\mathbb{T}^{n}\to \textup{Aut}(N,g)$ defines a Lie group 
	isomorphism between $\mathbb{T}^{n}$ and a maximal torus $T$ in $\textup{Aut}(N,g)$. 
\end{lemma}
\begin{proof}
	Let $T=\Phi(\mathbb{T}^{n})\subset \textup{Aut}(N,g)$.
	By Lemma \ref{nknkdnknskndknk}, $T$ is a closed Lie subgroup of $\textup{Aut}(N,g)$ and 
	the map $\mathbb{T}^{n}\to T$, $a\mapsto \Phi_{a}$, is a Lie group isomorphism. Thus 
	$T$ is a torus of dimension $n$. To see that it is maximal, let $T'$ be a torus in 
	$\textup{Aut}(N,g)$ such that $T\subseteq T'$. By the discussion above, the
	natural action of $T'$ on $N$ is Hamiltonian. By (F2), this forces 
	$\textup{dim}(T')\leq \textup{dim}(N)/2=n=\textup{dim}(T)$. 
	It follows from this and the inclusion $T\subseteq T'$ that $T=T'.$
\end{proof}

	Since $T=\Phi(\mathbb{T}^{n})$ is maximal in $\textup{Aut}(N,g)$, it is also maximal in the 
	identity component $\textup{Aut}(N,g)^{0}$. Let $W=W(T)$ be the Weyl group of $\textup{Aut}(N,g)^{0}$ associated to $T$ 
	(thus $W$ is the quotient group $N(T)/T$, where $N(T)$ is the normalizer of $T$ in $\textup{Aut}(N,g)^{0}$). 
	We will denote by $[\phi]$ the equivalence class of $\phi\in N(T)\subset \textup{Aut}(N,g)^{0}$ in $W$.

	Let $\Phi_{*}:\mathbb{R}^{n}\to \mathfrak{aut}(N,g)$ be the 
	derivative of $\Phi$ at the identity element. By Lemma \ref{nefkkwknefknwknk}, $\Phi_{*}$ is a 
	bijection between $\mathbb{R}^{n}$ and the Lie algebra 
	$\mathfrak{t}$ of $T$. Since $W$ acts linearly on $\mathfrak{t}$, there is a unique 
	linear action of $W$ on $\mathbb{R}^{n}$ such that 
	\begin{eqnarray*}
		\Phi_{*}(w\cdot \xi)=w\cdot (\Phi_{*}\xi)
	\end{eqnarray*}
	for all $w\in W$ and $\xi\in \mathbb{R}^{n}$. Let $\Phi^{*}:\mathfrak{aut}(N,g)^{*}\to \mathbb{R}^{n}$ be the map defined by 
	\begin{eqnarray*}
		\langle \Phi^{*}(\alpha), \xi\rangle=\alpha(\Phi_{*}\xi)
	\end{eqnarray*}
	for all $\alpha\in \mathfrak{aut}(N,g)^{*}$ and $\xi\in \mathbb{R}^{n}$, where $\langle\,,\, \rangle$ is the Euclidean pairing 
	in $\mathbb{R}^{n}$. 

	The maps $\moment$ and $\Phi^{*}\circ \overline{\moment}$ are both momentum maps for the same torus action. Since 
	$N$ is connected, there is $C\in \mathbb{R}^{n}$ such that 
	\begin{eqnarray}\label{neknknekwnknk}
		\moment=\Phi^{*}\circ \overline{\moment}+C.
	\end{eqnarray}

\begin{lemma}\label{nkekndkneknk}
	For every $\xi\in \mathbb{R}^{n}$ and $w=[\phi]\in W$, we have: 
	\begin{eqnarray*}
		\moment^{\xi}\circ \phi^{-1}-\moment^{w\cdot \xi}=\langle C,\xi-w\cdot \xi\rangle. 
	\end{eqnarray*}
\end{lemma}
\begin{proof}
	It follows from \eqref{neknknekwnknk} and the equivariance of $\overline{\moment}$ that 
	\begin{eqnarray*}
		\moment^{w\cdot \xi} &=& \overline{\moment}^{w\cdot \Phi_{*}\xi} +\langle C,w\cdot \xi\rangle \\
		&=& \overline{\moment}^{\textup{Ad}_{\phi}\Phi_{*}\xi}+\langle C,w\cdot \xi\rangle \\
		&=& \overline{\moment}^{\Phi_{*}\xi}\circ \phi^{-1}+\langle C,w\cdot \xi\rangle \\
		&=& (\moment^{\xi}-\langle C,\xi\rangle)\circ \phi^{-1}+\langle C,w\cdot \xi\rangle \\
		&=&	 \moment^{\xi}\circ \phi^{-1}-\langle C,\xi-w\cdot \xi\rangle. 
	\end{eqnarray*}
	The result follows. 
\end{proof}

\begin{definition}\label{neknwkenknk}
	The \textit{non-equivariance one-cocycle} (or simply \textit{one-cocycle}) 
	of $\moment$ is the map $\sigma_{\moment}:W\to (\mathbb{R}^{n})^{*}$ defined by
	\begin{eqnarray}\label{nekdnkfnnkenefkkn}
		\sigma_{\moment}(w)(\xi)=\moment^{\xi}\circ \phi^{-1}-\moment^{w\cdot \xi}, 
	\end{eqnarray}
	where $\xi\in \mathbb{R}^{n}$ and $w=[\phi]\in W$. 
\end{definition}
	
	It follows from Lemma \ref{nkekndkneknk} that the one-cocycle $\sigma_{\moment}$ is well-defined and 
	can be computed easily from the momentum map $\overline{\moment}$. 
	We call $\sigma_{\moment}$ a ``cocycle" because 
	\begin{eqnarray}\label{neknwkenkfnknk}
		\sigma_{\moment}(ww')(\xi)=\sigma_{\moment}(w)(w'\cdot \xi)+\sigma_{\moment}(w')(\xi)
	\end{eqnarray}
	for all $w,w'\in W$ and $\xi\in \mathbb{R}^{n}$. 
	Note that $\sigma_{\moment-C}$ is identically zero. Thus it is always possible to add 
	a constant to $\moment$ so that the corresponding one-cocycle is zero. \\

	The rest of this section is devoted to the explicit calculation of $\sigma_{\moment}$ when 
	$N=\mathbb{P}_{n}(c)$ is the complex projective space of complex dimension $n$ and holomorphic 
	sectional curvature $c>0$, equipped with the torus action $\Phi:\mathbb{T}^{n}\times \mathbb{P}_{n}(c)\to \mathbb{P}_{n}(c)$
	defined by 
	\begin{eqnarray*}
		\Phi([t],[z])=[e^{2i\pi t_{1}}z_{1},....,e^{2i\pi t_{n}}z_{n},z_{n+1}],
	\end{eqnarray*}
	where $[t]=[t_{1},...,t_{n}]\in \mathbb{T}^{n}=\mathbb{R}^{n}/\mathbb{Z}^{n}$ and $[z]=[z_{1},...,z_{n+1}]\in \mathbb{P}_{n}(c)$ 
	(homogeneous coordinates). The action $\Phi$ is effective, isometric and Hamiltonian, with equivariant momentum map 
	$\moment_{c}:\mathbb{P}_{n}(c)\to \mathbb{R}^{n}$ given by 
	\begin{eqnarray*}
		\moment_{c}([z])=-\dfrac{4\pi}{c}\bigg(\dfrac{|z_{1}|^{2}}{\langle z,z\rangle},...,
		\dfrac{|z_{n}|^{2}}{\langle z,z\rangle}\bigg), 
	\end{eqnarray*}
	where $\langle z,w\rangle = \overline{z}_{1}w_{1}+...+\overline{z}_{n+1}w_{n+1}$ is the standard Hermitian product 
	on $\mathbb{C}^{n+1}$. Below, we will regard $\mathbb{P}_{n}(c)$ as a K\"{a}hler toric manifold for the torus 
	action $\Phi$ and momentum map $\moment_{c}$. 

	Let $PU(n+1)$ be the projective unitary group that we identify with the 
	quotient group $SU(n+1)/\mathbb{Z}_{n+1}$, where 
	$\mathbb{Z}_{n+1}=\{\varepsilon \textup{I}_{n+1}\,\,\vert\,\,\varepsilon\in \mathbb{C},\,\,\varepsilon^{n+1}=1\}$ and 
	$I_{n+1}$ is the identity matrix. 
	The Lie algebra of $PU(n+1)$ is naturally identified with $\mathfrak{su}(n+1):=\textup{Lie}(SU(n+1))$ via the derivative 
	at the neutral element of the quotient map $SU(n)\to PU(n)$, $A\mapsto \mathbb{Z}_{n+1}A$ (here 
	$\mathbb{Z}_{n+1}A$ denotes the equivalence class of $A\in SU(n+1)$ in $PU(n+1)$).

	Let $\textup{Aut}(\mathbb{P}_{n}(c))$ denote the group of holomorphic isometries of $\mathbb{P}_{n}(c)$. 
	Given a complex matrix $A$ of size $(n+1)\times (n+1)$, we will denote by 
	$\phi_{A}$ the map $\mathbb{P}_{n}(c)\to \mathbb{P}_{n}(c)$ defined by $\phi_{A}([z])=[Az]$.
	It is well-known that the map 
	\begin{eqnarray*}
		\left\lbrace
		\begin{array}{lll}
			PU(n+1)	            &\to     & \textup{Aut}(\mathbb{P}_{n}(c)), \\[0.5em]
			\mathbb{Z}_{n+1}A   &\mapsto & \phi_{A},
		\end{array}
		\right.
	\end{eqnarray*}
	is a Lie group isomorphism. The natural action of $\textup{Aut}(\mathbb{P}_{n}(c))$ on 
	$\mathbb{P}_{n}(c)$ is Hamiltonian. An equivariant momentum map 
	$\overline{\moment}_{c}:\mathbb{P}_{n}(c)\to \mathfrak{su}(n+1)^{*}$ is given by
	\begin{eqnarray*}
		\big\langle\,\overline{\moment}_{c}([z]),\xi\,\big\rangle 
		= -\dfrac{2i}{c}\dfrac{\langle \xi(z),z\rangle}{\langle z,z\rangle},
	\end{eqnarray*}
	where $[z]\in \mathbb{P}_{n}(c)$, $\xi\in \mathfrak{su}(n+1)$, and where we use the same notation $\langle\,,\,\rangle$ 
	to denote the usual Hermitian product on $\mathbb{C}^{n+1}$ (linear in the second entry) and the natural pairing 
	between $\mathfrak{su}(n+1)$ and its dual. 
	
	Let $T\subset PU(n+1)$ be the image of $\mathbb{T}^{n}$ under the map $\Phi$ (it is a maximal torus in $PU(n+1)$ by 
	Lemma \ref{nefkkwknefknwknk}). Let $N(T)$ be the normalizer of $T$ in $PU(n+1)$ and let $W=W(T)=N(T)/T$ be the corresponding Weyl group. 
	Given $A\in SU(n+1)$ such that $\mathbb{Z}_{n+1}A\in N(T)$, we will denote by $T(\mathbb{Z}_{n+1}A)$ 
	the equivalence class of $\mathbb{Z}_{n}A$ in $W$.

	Let $\mathbb{S}_{n+1}$ be the group of permutations of $\{1,...,n+1\}$. Given $\sigma\in \mathbb{S}_{n+1}$, we will 
	denote by $E_{\sigma}\in SU(n+1)$ the matrix defined by 
	\begin{eqnarray}\label{nvenknkenwknekn}
		(E_{\sigma})_{ab}= 
		\textup{sign}(\sigma)^{\delta_{a\,n+1}}\delta_{a\sigma(b)}=
		\left\lbrace
	\begin{array}{lll}
		\delta_{a\sigma(b)}	                 &\textup{if} & a<n+1,\\
		\textup{sign}(\sigma)\delta_{a\sigma(b)} &\textup{if} & a=n+1,
	\end{array}
	\right.
	\end{eqnarray}
	where $\textup{sign}(\sigma)\in \{-1,1\}$ is the sign of $\sigma$ and $\delta_{ij}$ denotes the Kronecker Delta. 
	For instance, when $n=2$, there are $\textup{Card}\,\mathbb{S}_{n+1}=6$ matrices of the form $E_{\sigma}$:
	\begin{eqnarray*}
		&&
		\begin{bmatrix}	
			1 & 0 & 0 \\
			0 & 1 & 0 \\
			0 & 0 & 1
		\end{bmatrix},
	\,\,\,\,\,\,\,\,
	\begin{bmatrix}	
		0 & 1 & 0\\
		1 & 0 & 0\\
		0 & 0 &  -1
	\end{bmatrix},
	\,\,\,\,\,\,\,\,
	\begin{bmatrix}	
		0 & 0 & 1\\
		0 & 1 & 0\\
		-1 & 0 & 0
	\end{bmatrix},
	\,\,\,\,\,\,\,\,
	\begin{bmatrix}	
		1 & 0 & 0\\
		0 & 0 & 1\\
		0 & -1 &0
	\end{bmatrix},\\[0.5em]
		&&
	\begin{bmatrix}	
		0 & 1 & 0\\
		0 & 0 & 1\\
		1 & 0 & 0
	\end{bmatrix},
	\,\,\,\,\,\,\,\,
	\begin{bmatrix}	
		0 & 0 & 1\\
		1 & 0 & 0\\
		0 & 1 & 0
	\end{bmatrix}.
	\end{eqnarray*}

	Below we will adopt the following convention: if $\xi=(\xi_{1},...,\xi_{n})\in \mathbb{R}^{n}$, then $\xi_{n+1}:=0$. 
	Thus, if $\sigma\in \mathbb{S}_{n+1}$, it makes sense to write $\xi_{\sigma(k)}$ for all $k\in \{1,...,n+1\}$. 

\begin{proposition}\label{neknknekfnknn}
	Let $\mathbb{P}_{n}(c)$ be the complex projective space, regarded as a K\"{a}hler toric manifold for the torus action $\Phi$ and momentum map 
	$\moment_{c}$ (see above). Let $W$ be the Weyl group of $\textup{Aut}(\mathbb{P}_{n}(c))=PU(n+1)$ associated to the 
	maximal torus $T=\Phi(\mathbb{T}^{n})$. The following hold. 
	\begin{enumerate}[(i)]
		\item $\moment_{c}=\Phi^{*}\circ \overline{\moment}_{c}+C_{c}$, 
			where $C_{c}:=-\tfrac{4\pi}{c}\tfrac{1}{n+1}(1,...,1)\in \mathbb{R}^{n}$.
		\item $\mathbb{S}_{n+1}\to W$, $\sigma\mapsto T(\mathbb{Z}_{n}E_{\sigma})$ is a group isomorphism. 
		\item $\sigma\cdot \xi=(\xi_{\sigma^{-1}(1)}-\xi_{\sigma^{-1}(n+1)},...,\xi_{\sigma^{-1}(n)}
			-\xi_{\sigma^{-1}(n+1)})$ for all $\sigma\in \mathbb{S}_{n+1}=W$ and $\xi=(\xi_{1},...,\xi_{n})\in \mathbb{R}^{n}$. 
		\item $\sigma_{\moment_{c}}(\sigma)(\xi)=-\tfrac{4\pi}{c}\xi_{\sigma^{-1}(n+1)}$ for all 
			$\sigma\in \mathbb{S}_{n+1}=W$ and $\xi=(\xi_{1},...,\xi_{n})\in \mathbb{R}^{n}$.
	\end{enumerate}
\end{proposition}

	We will show the proposition with a series of lemmas. 
	
\begin{lemma}\label{neknknefknkdn}
	Let $S:\mathbb{R}^{n}\to \mathbb{R}$, $(\xi_{1},...,\xi_{n})\mapsto \xi_{1}+...+\xi_{n}$. 
	Given $[t]=[t_{1},...,t_{n}]\in \mathbb{T}^{n}$ and $\xi=(\xi_{1},...,\xi_{n})\in \mathbb{R}^{n}=\textup{Lie}(\mathbb{T}^{n})$, 
	we have:
	\begin{eqnarray*}
		\Phi([t])&=&\mathbb{Z}_{n+1}\,
		\textup{diag}\bigg[e^{2i\pi \big(t_{1}-\tfrac{S(t)}{n+1}\big)},...,
		e^{2i\pi \big(t_{n}-\tfrac{S(t)}{n+1}\big)}, e^{-2i\pi \tfrac{S(t)}{n+1}}\bigg]\in PU(n+1),\\
		\Phi_{*}\xi&=& 2i\pi\, \textup{diag}\Big[\xi_{1}-\tfrac{S(\xi)}{n+1},...,
		\xi_{n}-\tfrac{S(\xi)}{n+1},-\tfrac{S(\xi)}{n+1}\Big]\in \mathfrak{su}(n+1),
	\end{eqnarray*}
	where the notation $\textup{diag}(\lambda_{1},...,\lambda_{n+1})$ denotes the 
	matrix of size $n+1$ whose $(i,j)$-entry is $\lambda_{i}\delta_{ij}$ ($\delta_{ij}=$ Kronecker Delta). 
\end{lemma}
\begin{proof}
	By a direct calculation.
\end{proof}
	By inspection of Lemma \ref{neknknefknkdn}, 
	we see that 
	\begin{eqnarray*}
		T=\Phi(\mathbb{T}^{n})
		=\big\{\mathbb{Z}_{n}\,\textup{diag}[\lambda_{1},...,\lambda_{n+1}]\,\,\vert\,\,\lambda_{j}\in \mathbb{C}\,\,\textup{and}\,\,
		|\lambda_{j}|=1\,\,\forall\,j,
		\lambda_{1}...\lambda_{n+1}=1\big\}. 
	\end{eqnarray*}

	Note that $T$ is the image under the projection $SU(n+1)\to PU(n+1)$ of the group $S\subseteq SU(n+1)$ of diagonal matrices. 
	It is well-known that $S$ is a maximal torus in $SU(n+1)$. 
\begin{lemma}\label{nekwneknfkk}
	The following hold.
	\begin{enumerate}[(i)]
		\item A complex matrix $A=(A_{ij})$ of size $n+1$ is in the normalizer $N(S)$ of $S$ if and only if there are complex numbers 
			$\lambda_{1},...,\lambda_{n+1}$ and a permutation $\sigma\in \mathbb{S}_{n+1}$ such that:
			\begin{enumerate}[(a)]
				\item $|\lambda_{1}|=...=|\lambda_{n+1}|=1$ and $\lambda_{1}...\lambda_{n+1}=\textup{sign}(\sigma)$. 
				\item $A_{ab}=\lambda_{b}\delta_{a\sigma(b)}$ for all $a,b\in \{1,...,n+1\}$. 
			\end{enumerate}
		\item Given $\sigma,\tau\in \mathbb{S}_{n+1}$, there is $A\in S$ such that $E_{\sigma}E_{\tau}=AE_{\sigma \circ \tau}$. 
		\item $N(S)=\sqcup_{\sigma\in \mathbb{S}_{n+1}}SE_{\sigma}$ (disjoint union of cosets). 
		\item The map $\mathbb{S}_{n+1}\to W(S),$ $\sigma\mapsto SE_{\sigma}$ is a 
			group isomorphism, where $W(S)$ is the Weyl group of $SU(n+1)$ associated to $S$. 
	\end{enumerate}
\end{lemma}
\begin{proof}
	By a direct calculation.
\end{proof}
	The next lemma is proven in Chapter IV of \cite{tammo}, Theorem 2.9. 
\begin{lemma}\label{nekwnkenfknk}
	Let $f:G\to H$ be a surjective homomorphism of compact connected Lie groups, and let 
	$T\subset G$ be a maximal torus. Then $f(T)$ is a maximal torus in $H$. 
	Furthermore, if $\textup{ker}(f)\subset T$, then $f$ induces an isomorphism of the Weyl groups. 
\end{lemma}

\begin{proof}[Proof of Proposition \ref{neknknekfnknn}]
	(i) Follows from a direct computation using Lemma \ref{neknknefknkdn}. (ii) Follows from
	Lemmas \ref{nekwneknfkk} and \ref{nekwnkenfknk}. Let us show (iii). 
	Let $\sigma\in \mathbb{S}_{n+1}$ and $\xi\in \mathbb{R}^{n}$. By definition, 
	\begin{eqnarray}\label{njjwnknvrnefwdkfekj}
		\Phi_{*}(\sigma\cdot \xi)=\textup{Ad}_{\mathbb{Z}_{n+1}E_{\sigma}}\Phi_{*}\xi.
	\end{eqnarray}
	Under the natural identification $\textup{Lie}(PU(n+1))=\mathfrak{su}(n+1)$, it is easy to see that 
	$\textup{Ad}_{\mathbb{Z}_{n+1}A}\eta=\textup{Ad}_{A}\eta$ for all $\eta\in \mathfrak{su}(n+1)$ and all $A\in SU(n+1)$. Thus 
	\eqref{njjwnknvrnefwdkfekj} can be rewritten as 
	\begin{eqnarray}\label{nekwnkenfkwndk}
		\Phi_{*}(\sigma\cdot \xi)=\textup{Ad}_{E_{\sigma}}\Phi_{*}\xi.
	\end{eqnarray}
	Let $a,b\in \{1,...,n+1\}$ be arbitrary. In view of Lemma \ref{neknknefknkdn}, we have 
	\begin{eqnarray*}
		\Big(\textup{Ad}_{E_{\sigma}}\Phi_{*}\xi\Big)_{ab}&=& \Big(E_{\sigma} (\Phi_{*}\xi) (E_{\sigma})^{T}\Big)_{ab}\\
		&=& \sum_{k,l} \big(E_{\sigma}\big)_{ak} (\Phi_{*}\xi)_{kl} (E_{\sigma})_{bl}\\
		&=& 2i\pi\,\sum_{k,l} \textup{sign}(\sigma)^{\delta_{a(n+1)}}\delta_{a\sigma(k)}
		\delta_{kl}\big(\xi_{k}-\tfrac{S(\xi)}{n+1}\big) \textup{sign}(\sigma)^{\delta_{b(n+1)}} \delta_{b\sigma(l)}\\
		&=& 2i\pi\,\sum_{k} \underbrace{\textup{sign}(\sigma)^{\delta_{a(n+1)}}\delta_{a\sigma(k)}
			\big(\xi_{k}-\tfrac{S(\xi)}{n+1}\big) \textup{sign}(\sigma)^{\delta_{b(n+1)}} \delta_{b\sigma(k)}}_{=0\,\,\textup{if}\,\,
			a\neq b}\\
		&=& 2i\pi\,\sum_{k} \delta_{ab} \delta_{a\sigma(k)}
			\big(\xi_{k}-\tfrac{S(\xi)}{n+1}\big) \big(\textup{sign}(\sigma)^{\delta_{a(n+1)}}\big)^{2} \delta_{a\sigma(k)}\\
		&=& 2i\pi\,\sum_{k} \delta_{ab} \delta_{a\sigma(k)}
			\big(\xi_{k}-\tfrac{S(\xi)}{n+1}\big)\delta_{a\sigma(k)}\\
		&=& 2i\pi\,\delta_{ab} \big(\xi_{\sigma^{-1}(a)}-\tfrac{S(\xi)}{n+1}\big).
	\end{eqnarray*}
	Thus 
	\begin{eqnarray}\label{nknknekknkndk}
		\textup{Ad}_{E_{\sigma}}\Phi_{*}\xi=2i\pi \textup{diag}\Big[\xi_{\sigma^{-1}(1)}-\tfrac{S(\xi)}{n+1},...,
		\xi_{\sigma^{-1}(n+1)}-\tfrac{S(\xi)}{n+1}\Big].
	\end{eqnarray}
	Let $\xi'=(\xi'_{1},...,\xi'_{n})\in \mathbb{R}^{n}$ be defined by $\xi'_{k}=\xi_{\sigma^{-1}(k)}-\xi_{\sigma^{-1}(n+1)}$, where 
	$k\in \{1,...,n\}$. We claim that $\textup{Ad}_{E_{\sigma}}\Phi_{*}\xi=\Phi_{*}\xi'$. Indeed, 
	\begin{eqnarray*}
		S(\xi') &=&\xi'_{1}+...+\xi'_{n}\\
		&=& (\xi_{\sigma^{-1}(1)}+...+\xi_{\sigma^{-1}(n)}) -n\xi_{\sigma^{-1}(n+1)}\\
		&=& S(\xi) -\xi_{\sigma^{-1}(n+1)} -n\xi_{\sigma^{-1}(n+1)}\\
		&=& S(\xi) -(n+1)\xi_{\sigma^{-1}(n+1)}
	\end{eqnarray*}
	and so $-\tfrac{S(\xi')}{n+1}=\xi_{\sigma^{-1}(n+1)}-\tfrac{S(\xi)}{n+1}.$ 
	Thus we can rewrite \eqref{nknknekknkndk} as 
	\begin{eqnarray}\label{nfekwknekfnejdnwknekwnkkd}
		&&\textup{Ad}_{E_{\sigma}}\Phi_{*}\xi= 2i\pi\,\textup{diag}\Big[
			\xi'_{1}-\tfrac{S(\xi')}{n+1},...,
		\xi'_{n}-\tfrac{S(\xi')}{n+1},-\tfrac{S(\xi')}{n+1}\Big]=\Phi_{*}\xi',
	\end{eqnarray}
	which shows the claim. It follows from the claim and \eqref{nekwnkenfkwndk} that $\sigma\cdot \xi=\xi'$. 
	This completes the proof of (iii). (iv) Follows from a direct computation using (i), (iii) and Lemma \ref{nkekndkneknk}. 
\end{proof}

\section{K\"{a}hler functions on toric manifolds}\label{nenefkwnknefkwndk}

	In this section, we investigate the relationships between the space of K\"{a}hler functions 
	on a K\"{a}hler toric manifold and the momentum map associated to the torus action. 
	We start by recalling the following definition.

\begin{definition}\label{fdkjfekgjrkgj}
	Let $N$ be a K\"{a}hler manifold with K\"{a}hler metric $g$ and K\"{a}hler form $\omega$. 
	A smooth function $f\,:\,N\rightarrow \mathbb{R}$ is said to be 
	\textit{K\"{a}hler} if it satisfies $\mathscr{L}_{X_{f}}g=0$,
	where $X_{f}$ is the Hamiltonian vector field associated to $f$ 
	(i.e. $\omega(X_{f},\,.\,)=df(.)$) and where $\mathscr{L}_{X_{f}}$ is the Lie derivative 
	in the direction of $X_{f}.$
\end{definition}

	Following \cite{Cirelli-Quantum}, we will denote by 
	$\mathscr{K}(N)$ the space of K\"{a}hler functions 
	on $N$ endowed with the Poisson bracket $\{f,h\}:=\omega(X_{f},X_{h})$.
	When $N$ has a finite number of connected components, $\mathscr{K}(N)$
	is finite dimensional\footnote{This is a consequence of the following result: 
	if $(M,h)$ is a connected Riemannian manifold, 
	then its space of Killing vector fields $\mathfrak{X}_{\textup{Kill}}(M):=
	\{X\in \mathfrak{X}(M)\,\big\vert\,\mathscr{L}_{X}h=0\}$ 
	is finite dimensional (see for example \cite{Jost}). }.

	Let $N$ be a K\"{a}hler manifold with K\"{a}hler metric $g$. We will denote by $\textup{Aut}(N,g)$ 
	the group of holomorphic isometries of $N$ endowed with the compact-open topology. 
	When $N$ is connected, $\textup{Aut}(N,g)$ is a Lie group whose natural action on $N$ is smooth 
	(see Proposition \ref{nfekwnknefknk}). In this case, we shall denote by $\mathfrak{aut}(N,g)$ the Lie algebra 
	of $\textup{Aut}(N,g)$ and by $[\,,\,]_{\mathfrak{aut}}$ the corresponding Lie bracket. 
	If in addition $N$ is compact, then $\textup{Aut}(N,g)$ is also compact,
	since it is a closed subgroup of the group of isometries of $N$, which is compact (see Theorem \ref{dkfjkejfkegjkr}). 

	Let $\mathfrak{X}_{\textup{hol}}(N,g)$ be the space of complete Killing vector fields $X$ on $N$ whose flows 
	consist of biholomorphisms, endowed with minus the usual Lie bracket for vector fields, that we denote by $[,]_{-}$. 
	When $N$ is connected, the map 
	\begin{eqnarray}\label{unfekefknkfknk}
		(\mathfrak{aut}(N,g),[,]_{\mathfrak{aut}})\mapsto (\mathfrak{X}_{\textup{hol}}(N,g),[,]_{-}), \,\,\,\,\,\,\xi\mapsto 
		\xi_{N},
	\end{eqnarray}
	is well-defined and is a Lie algebra isomorphism, where $\xi_{N}$ denotes the fundamental vector field associated to $\xi$ 
	(see Propositions \ref{necknknknfsknfk} and \ref{nfekwnknefknk}).

\begin{lemma}\label{ndkndkneksdnknk}
	Let $N$ be a connected, compact K\"{a}hler manifold with K\"{a}hler metric $g.$ Suppose that 
	the first de Rham cohomology group $H^{1}_{dR}(N,\mathbb{R})$ is trivial 
	(for example, if $N$ is simply connected). Then the natural action of $\textup{Aut}(N,g)$ on $N$ 
	is Hamiltonian. Furthermore, if $\overline{\moment}:N\to \mathfrak{aut}(N,g)^{*}$ is an equivariant momentum map, 
	then the map
	\begin{eqnarray}\label{nenkrkfenknekk}
		\begin{tabular}{cll}
			$\mathfrak{aut}(N,g)\oplus \mathbb{R}$   &$\to$& $\mathscr{K}(N)$, \\[0.5em]
			$(\xi,r)$                         &$\mapsto$& $\overline{\moment}^{\xi}+r$,
		\end{tabular}
	\end{eqnarray}
	is a Lie algebra isomorphism, where $\mathfrak{aut}(N,g)\oplus \mathbb{R}$ is endowed 
	with the Lie bracket $[(\xi,r),(\eta,s)]=([\xi,\eta]_{\mathfrak{aut}},0)$. 
\end{lemma}
\begin{proof}
	The fact that the natural action of $\textup{Aut}(N,g)$ on $N$ is Hamiltonian follows from (F1).
	Let $\overline{\moment}:N\to \mathfrak{aut}(N,g)^{*}$ be an equivariant momentum map and let
	$\phi:\mathfrak{aut}(N,g)\oplus \mathbb{R}\to C^{\infty}(N,\mathbb{R})$ be defined by $\phi(\xi,r)=\overline{\moment}^{\xi}+r$, 
	where $C^{\infty}(N,\mathbb{R})$ is the space of smooth functions on $N$.
	Clearly, $\phi$	is linear and its image satisfies $\textup{Im}(\phi)\subseteq \mathscr{K}(N)$. 
	Let $\{f,h\}=\omega(X_{f},X_{h})$ denote the Poisson bracket on $C^{\infty}(N,\mathbb{R})$ 
	associated to the K\"{a}hler form $\omega.$ To prove the lemma, it suffices to show that:
	\begin{enumerate}[(a)]
	\item $\textup{Im}(\phi)=\mathscr{K}(N).$ 
	\item $\phi$ is injective.
	\item $\phi([(\xi,r),(\xi',r')])=\{\phi(\xi,r),\phi(\xi',r')\}$ for all $\xi,\xi'\in \mathfrak{aut}(N,g)$ and all 
		$r,r'\in \mathbb{R}$. 
	\end{enumerate}
	\noindent (a) It suffices to show that $\mathscr{K}(N)\subseteq \textup{Im}(\phi)$. 
	So let $f\in \mathscr{K}(N)$ be arbitrary and let $X_{f}$ be the corresponding Hamiltonian vector field. 
	Note that $X_{f}$ is complete, since $N$ is compact. Because $X_{f}$ is Hamiltonian and Killing, 
	its flow preserves $g$ and $\omega$. Since $g,\omega$ and the complex structure $J$ are compatible, 
	$X_{f}$ also preserves $J$, which means that 
	the flow of $X_{f}$ consists of holomorphic transformations. Thus $X_{f}\in \mathfrak{X}_{\textup{hol}}(N,g)$. 
	By \eqref{unfekefknkfknk}, $X_{f}$ is the fundamental vector 
	field associated to some $\xi\in \mathfrak{aut}(N,g)$, that is, $X_{f}=\xi_{N}$. 
	It follows that for every $u\in TN$, 
	\begin{eqnarray*}
		df(u)=\omega(X_{f},u)=\omega(\xi_{N},u)=d\overline{\moment}^{\xi}(u),
	\end{eqnarray*}
	and so $df=d\overline{\moment}^{\xi}$ on $N$. Since $N$ is connected, there is $r\in \mathbb{R}$ such 
	that $f=\overline{\moment}^{\xi}+r=\phi(\xi,r)$. This shows that 
	$\mathscr{K}(N)\subseteq \textup{Im}(\phi)$. 

	\noindent (b) Suppose that $\phi(\xi,r)=0$. Thus $\overline{\moment}^{\xi}\equiv -r$ on $N$ and hence 
	$\omega(\xi_{N},u)=d\overline{\moment}^{\xi}(u)=d(-r)(u)=0$ for all $u\in TN$. It follows that $\xi_{N}=0$ on $N$. 
	Since the natural action of $\textup{Aut}(N,g)$ on $N$ is effective, this implies 
	$\xi=0$ (see Lemma \ref{wkjkjkefjwkjk}), which in turn implies that $r\equiv-\overline{\moment}^{0}=0$. 
	Thus $\xi=0$ and $r=0$. 

	\noindent (c) Since $\overline{\moment}$ 
	is $\textup{Aut}(N,g)$-equivariant, it is also infinitesimally equivariant, which means that 	
	$\overline{\moment}^{[\xi,\eta]_{\mathfrak{aut}}}=\{\overline{\moment}^{\xi},\overline{\moment}^{\eta}\}$ 
	for all $\xi,\eta\in \mathfrak{aut}(N,g)$. Thus
	\begin{eqnarray*}
		\phi([(\xi,r),(\xi',r')])&=& \phi([\xi,\xi']_{\mathfrak{aut}},0)=\overline{\moment}^{[\xi,\xi']_{\mathfrak{aut}}}=
		\{\overline{\moment}^{\xi},\overline{\moment}^{\xi'}\}\\
		&=&\{\overline{\moment}^{\xi}+r,\overline{\moment}^{\xi'}+r'\}=\{\phi(\xi,r),\phi(\xi',r')\}. 
	\end{eqnarray*}
	This completes the proof. 
\end{proof}

	We now focus our attention to the case in which $N$ is a K\"{a}hler toric manifold, with torus action 
	$\Phi:\mathbb{T}^{n}\times N\to N$ and momentum map $\moment:N\to \mathbb{R}^{n}$. By Lemma \ref{nefkkwknefknwknk}, 
	$T=\Phi(\mathbb{T}^{n})$ is a maximal torus in $\textup{Aut}(N,g)^{0}$. Let $W=W(T)$ be the Weyl group 
	of $\textup{Aut}(N,g)^{0}$ associated to $T$, and let $\sigma_{\moment}:W\to (\mathbb{R}^{n})^{*}$ 
	be the one-cocycle associated to $\moment$. 

	In view of \eqref{neknwkenkfnknk}, one can easily check that the formula
        \begin{eqnarray}\label{nceknfkfneknk}
		w\cdot (\xi,r)=(w\cdot \xi, r+\sigma_{\moment}(w)(\xi)) \,\,\,\,\,\,\,\,(w\in W,\xi\in\mathbb{R}^{n},r\in \mathbb{R})
	\end{eqnarray}
	defines an action of $W$ on $\mathbb{R}^{n}\oplus \mathbb{R}$. 
	(Recall that the Lie algebra $\mathfrak{t}$ of $T$ is identified with $\mathbb{R}^{n}$ via $\Phi_{*}$).

\begin{proposition}[\textbf{Diagonalization of K\"{a}hler functions}]\label{nenekfnsddnekn}
	Let $\Phi:\mathbb{T}^{n}\times N\to N$ be a K\"{a}hler toric manifold, with 
	K\"{a}hler metric $g$ and momentum map $\moment:N\to \mathbb{R}^{n}$. 
	Let $W=W(T)$ be the Weyl group of $\textup{Aut}(N,g)^{0}$ associated to the maximal torus $T=\Phi(\mathbb{T}^{n})$. 
	Given a smooth function $f:N\to \mathbb{R}$, the following hold.
	\begin{enumerate}[(a)]
	\item $f$ is K\"{a}hler if and only if there exist $\xi\in \mathbb{R}^{n}$, $\varphi\in \textup{Aut}(N,g)^{0}$ and 
		$r\in \mathbb{R}$ such that $f=\moment^{\xi}\circ \varphi+r$.
	\item If $f=\moment^{\xi}\circ \varphi+r=\moment^{\xi'}\circ \varphi'+r'$ on $N$, where $\xi,\xi'\in 
		\mathbb{R}^{n}$, $\varphi,\varphi'\in \textup{Aut}(N,g)^{0}$ and $r,r'\in \mathbb{R}$, 
		then there is $w\in W$ such that $(\xi,r)=w\cdot (\xi',r')$, where the action of $W$ on $\mathbb{R}^{n}\oplus \mathbb{R}$ 
		is given by \eqref{nceknfkfneknk}. 
	\end{enumerate}
\end{proposition}
\begin{proof}
	By Lemma \ref{ndkndkneksdnknk}, the natural action of $\textup{Aut}(N,g)$ on $N$ is Hamiltonian; let 
	$\overline{\moment}:N\to \mathfrak{aut}(N,g)^{*}$ be an equivariant momentum map. Since $\moment$ and 
	$\Phi^{*}\circ \overline{\moment}$ are both momentum maps for the same torus action, there is $C\in \mathbb{R}^{n}$ 
	such that $\moment=\Phi^{*}\circ \overline{\moment}+C$, or equivalently, such that $\moment^{\xi}=\overline{\moment}^{\Phi_{*}\xi}+
	\langle C,\xi\rangle$ for all $\xi\in \mathbb{R}^{n}$, where $\langle\,,\,\rangle$ is the Euclidean product on $\mathbb{R}^{n}$.

	Now let $f:N\to \mathbb{R}$ be a function of the form $f=\overline{\moment}^{\xi}+r$, where 
	$\xi\in \mathfrak{aut}(N,g)$ and $r\in \mathbb{R}$. By (P1) and the fact that $\Phi_{*}:\mathbb{R}^{n}\to \mathfrak{t}\subset 
	\mathfrak{aut}(N,g)$ is a bijection (Lemma \ref{nefkkwknefknwknk}), there are $\xi'\in \mathbb{R}^{n}$ and 
	$\phi\in \textup{Aut}(N,g)^{0}$ such that $\xi=\textup{Ad}_{\phi}\Phi_{*}\xi'$, and thus 
	\begin{eqnarray*}
		f&=&\overline{\moment}^{\xi}+r=\overline{\moment}^{\textup{Ad}_{\phi}\Phi_{*}\xi'}+r =
		\overline{\moment}^{\Phi_{*}\xi'}\circ \phi^{-1}+r \\
		&=& (\moment^{\xi'}-\langle C,\xi'\rangle)\circ \phi^{-1}+r\\
		&=&\moment^{\xi'}\circ \phi^{-1}-\langle C,\xi'\rangle+r, 
	\end{eqnarray*}
	where, in the first line, we have used the fact that $\overline{\moment}$ is equivariant. Set 
	$\varphi'=\phi^{-1}$ and $r'=-\langle C,\xi'\rangle+r$. Then $f=\moment^{\xi'}\circ \varphi'+r'$, 
	where $\xi'\in \mathbb{R}^{n}$, $\varphi'\in \textup{Aut}(N,g)^{0}$ and $r'\in \mathbb{R}$. In view of 
	Lemma \ref{ndkndkneksdnknk}, this shows (a). To see (b), suppose that $\moment^{\xi}\circ \varphi+r=\moment^{\xi'}\circ \varphi'+r'$, 
	where $\xi,\xi'\in \mathbb{R}^{n}$, $\varphi,\varphi'\in \textup{Aut}(N,g)^{0}$ and $r,r'\in \mathbb{R}$. 
	An argument analogous to the one above shows that $\moment^{\xi}\circ \varphi+r
	=\overline{\moment}^{\textup{Ad}_{\varphi^{-1}}\Phi_{*}\xi}+\langle C,\xi\rangle+r$ and hence 
	\begin{eqnarray*}
		\overline{\moment}^{\textup{Ad}_{\varphi^{-1}}\Phi_{*}\xi}+\langle C,\xi\rangle+r
		=\overline{\moment}^{\textup{Ad}_{(\varphi')^{-1}}\Phi_{*}\xi'}+\langle C,\xi'\rangle+r'.
	\end{eqnarray*}
	It follows from this and Lemma \ref{ndkndkneksdnknk} that 
	$\textup{Ad}_{\varphi^{-1}}\Phi_{*}\xi= \textup{Ad}_{(\varphi')^{-1}}\Phi_{*}\xi'$
	and $\langle C,\xi\rangle+r=\langle C,\xi'\rangle+r'$. The first equation shows that 
	$\xi,\xi'\in\mathbb{R}^{n}\cong \mathfrak{t}$ are in the same $\textup{Aut}(N,g)^{0}$-orbit. By (P2), this implies 
	that there is $w\in W$ such that $\xi=w\cdot \xi'$. Thus 
	\begin{eqnarray*}
		(\xi,r)&=&(w\cdot \xi', \langle C,\xi'-\xi\rangle+r') =(w\cdot \xi', \langle C,\xi'-w\cdot \xi'\rangle+r')\\
			&=&(w\cdot \xi', \sigma_{\moment}(w)(\xi')+r') = w\cdot (\xi',r'). 
	\end{eqnarray*}
	The proposition follows. 
\end{proof}

	We end this section with a simple but important example. 
	Let $\textup{Herm}_{n}$ denote the space of Hermitian matrices of size $n$ and let $\mathscr{K}(\mathbb{P}_{n}(c))$ be the 
	space of K\"{a}hler functions on the complex projective space $\mathbb{P}_{n}(c)$. 
	Given $H\in \textup{Herm}_{n+1}$, we will denote by $f_{H}:\mathbb{P}_{n}(c)\to \mathbb{R}$ the function defined by 
	\begin{eqnarray*}
		f_{H}([z])=\dfrac{\langle z,Hz\rangle}{\langle z,z\rangle},
	\end{eqnarray*}
	where $[z]=[z_{1},...,z_{n+1}]\in \mathbb{P}_{n}(c)$ (homogeneous 
	coordinates) and $\langle z,w\rangle=\overline{z}_{1}w_{1}+...+\overline{z}_{n+1}w_{n+1}$ 
	is the usual Hermitian product on $\mathbb{C}^{n+1}$. It is well-known that the map 
	$\textup{Herm}_{n+1}\to \mathscr{K}(\mathbb{P}_{n}(c))$ is a bijection (see, e.g., \cite{Cirelli-Quantum}).

	Given a complex matrix $U$, we will denote by $U^{*}$ its conjugate. When $U\in SU(n+1)$, we will denote 
	by $\phi_{U}$ the holomorphic isometry of 
	$\mathbb{P}_{n}(c)$ defined by $\phi_{U}([z])=[Uz]$ (recall that every holomorphic isometry of $\mathbb{P}_{n}(c)$ is of this 
	form). 
	
\begin{proposition}\label{neknkefnwknfekn}
	Let $\mathbb{P}_{n}(c)$, $\Phi$ and $\moment_{c}$ be as in Proposition \ref{neknknekfnknn}. 
	Then for every $\xi\in \mathbb{R}^{n}$, $U\in SU(n+1)$ and $r\in \mathbb{R}$, we have 
	$\moment_{c}^{\xi}\circ \phi_{U}+r=f_{H}$, where $H$ is the Hermitian matrix defined by $H=U^{*}DU$,
	and 
	\begin{eqnarray*}
		D=\begin{bmatrix}
			r-\tfrac{4\pi\xi_{1}}{c}   &  & 0 & \\
					      &  &   & \\
			                           & r-\tfrac{4\pi\xi_{n}}{c} &  &  \\
			0  && r&
		\end{bmatrix}.
	\end{eqnarray*}
\end{proposition}
\begin{proof}
	By a direct calculation.
\end{proof}

\section{Spectral theory on toric exponential families}\label{nfeqnkdekwnn}

	In this section, we specialize to the case in which a toric K\"{a}hler manifold $\Phi:\mathbb{T}^{n}\times N\to N$ 
	is the regular torification of an exponential family $\mathcal{E}$ defined over a finite set, and prove a 
	strengthened version of Proposition \ref{nenekfnsddnekn}. The proof requires some preliminary results that 
	we present in separate sections. Some results are technical, while others are of independent interest. 
	We begin by specifying our setting. 

\begin{setting}\label{mnekmnwknefknk}
	\textbf{}
	\begin{enumerate}[(a)]
	\item Let $\Omega=\{x_{1},...,x_{m}\}$ be a finite set endowed with the counting measure. Let
	$\mathcal{E}$ be an exponential family of dimension $n$ defined over $\Omega$, with elements of the form 
	$p(x;\theta)=p_{\theta}(x)=\textup{exp}\big[C(x)+\sum_{k=1}^{n}\theta_{k}F_{k}(x)-\psi(\theta)\big]$, where 
	\begin{eqnarray*}
		x\in \Omega,\,\,\,\,\,\,\theta\in \mathbb{R}^{n},\,\,\,\,\,\,
		C,F_{1},...,F_{n}:\Omega\to \mathbb{R},\,\,\,\,\,\,\psi:\mathbb{R}^{n}\to \mathbb{R}. 
	\end{eqnarray*}
	We assume that the functions $F_{1},...,F_{n},1$ are linearly independent. 

	\item Let $\Phi:\mathbb{T}^{n}\times N\to N$ be a K\"{a}hler toric manifold, with K\"{a}hler metric $g$ and momentum map 
		$\moment:N\to \mathbb{R}^{n}$ (thus $N$ is compact, connected, $\Phi$ is effective, holomorphic and isometric,
		and $\textup{dim}_{\mathbb{R}}N=2n$). 
		We assume that $N$, equipped with $\Phi$, is a regular torification of $\mathcal{E}$. 
	\end{enumerate}
\end{setting}

	As a matter of notation, we shall use $\mathcal{A}_{\mathcal{E}}$ to denote the 
	vector space of real-valued functions $f$ on $\Omega$ spanned by $F_{1},...,F_{n},1$ 
	(thus $\textup{dim}\mathcal{A}_{\mathcal{E}}=n+1$). 
	The map $F:\Omega\to \mathbb{R}^{n}$, defined by $F(x)=(F_{1}(x),...,F_{n}(x))$, 
	is interpreted in the statistical literature as an unbiased estimator for the model $\mathcal{E}$. 
	We shall call it the \textit{estimator map}. 
	The expectation of a random variable $X:\Omega\to \mathbb{R}$, with respect to 
	a probability function $p:\Omega\to \mathbb{R}$, will be denoted by $\mathbb{E}_{p}(X)$.

\begin{remark}\label{nemwwnkefnkwnk}
	The estimator map $F$ is injective if and only if $\mathcal{A}_{\mathcal{E}}$ 
	separates the points of $\Omega$, that is, 
	if for every pair of distinct points $x$ and $y$ in $\Omega$, there is $X\in \mathcal{A}_{\mathcal{E}}$ 
	such that $X(x)\neq X(y)$.
\end{remark}

%

\subsection{Extending compatible projections.}\label{ndknvfknkfnvknk}
	Let $N^{\circ}$ be the set of points $p\in N$ where the torus action is free; it is a connected open dense subset of $N$ 
	by (F4). The purpose of this section is to show that every compatible projection $\kappa:N^{\circ}\to \mathcal{E}$ 
	(see Definition \ref{fjenkrkefnkn}) extends continuously to a map $K:N\to \overline{\mathcal{E}}$, where 
	$\overline{\mathcal{E}}$ is an appropriate superset of $\mathcal{E}$.
	
	Let $\mathcal{P}_{m}$ be the set of all probability functions on $\Omega$. Thus a function 
	$p:\Omega\to \mathbb{R}$ is in $\mathcal{P}_{m}$ if and only if $p(x)\geq 0$ for all $x\in \Omega$ 
	and $\sum_{x\in \Omega}p(x)=1$. We give $\mathcal{P}_{m}$ the topology for which the map $\mathcal{P}_{m}\to \mathbb{R}^{m}$, 
	$p\mapsto (p(x_{1}),...,p(x_{m}))$ is a homeomorphism onto its image in the subspace topology.
	Thus a sequence $(p_{j})_{j\in \mathbb{N}}$ of points $p_{j}\in \mathcal{P}_{m}$ converges to some $p\in \mathcal{P}_{m}$ 
	if and only if $p_{j}(x)\to p(x)$ for all $x\in \Omega$. 
	Since $\mathcal{E}$ is naturally a subset of $\mathcal{P}_{m}$, we may consider its closure 
	$\overline{\mathcal{E}}$ in $\mathcal{P}_{m}$. To determine $\overline{\mathcal{E}}$, it is often convenient 
	to use the following property of closed maps, whose proof is left to the reader. 

\begin{lemma}\label{nfekenkrnefkn}
	Let $f:X\to Y$ be a continuous map between topological spaces. Suppose that $f$ is closed (in the sense that 
	the image of every closed set is closed). Then $f(\overline{A})=\overline{f(A)}$ for every set $A\subseteq X$
	(here the closure of a set $B$ is denoted by $\overline{B}$). 
\end{lemma}

	Note that a continuous map $f:X\to Y$ from a compact space $X$ to a Hausdorff space $Y$ is automatically closed. 

\begin{example}
	The set $\mathcal{B}(n)$ of Binomial distributions defined over $\Omega=\{0,...,n\}$ 
	is the image of the continuous map $f:(0,1)\to \mathcal{P}_{n+1}$ 
	defined by $f(q)(k)=\binom{n}{k}q^{k}\bigl(1-q\bigr)^{n-k}$, 
	where $q\in (0,1)$ and $k\in \{0,...,n\}$. The map $f$ extends continuously 
	to a closed map $f:[0,1]\to \mathcal{P}_{n+1}$, where $f(0)=\delta_{0}$ and $f(1)=\delta_{n}$ are defined by 
	\begin{eqnarray*}
		\delta_{i}(j)=
		\left\lbrace
		\begin{array}{lll}
			1  & \textup{if} & i=j,\\
			0  & \textup{if} & i\neq j.
		\end{array}
		\right. \,\,\,\,\,\,\,\,\,\,(i\in \{0,n\},\,\,j\in \Omega)
	\end{eqnarray*}
	By Lemma \ref{nfekenkrnefkn}, we have $\overline{\mathcal{B}(n)}=\overline{f((0,1))}=f(\overline{(0,1)})=f([0,1])$, and so 
	\begin{eqnarray*}
		\overline{\mathcal{B}(n)}=\mathcal{B}(n)\cup \{\delta_{0},\delta_{n}\}. 
	\end{eqnarray*}
\end{example}
\begin{example}
	Let $\Delta_{n}=\{(y_{1},...,y_{n})\in \mathbb{R}^{n}\,\,\vert\,\,y_{k}\geq 0\,\,\forall k,\,\,\,y_{1}+...+y_{n}\leq 1\}$,
	and let $f:\Delta_{n}\to \mathcal{P}_{n+1}$ be the continuous closed surjective map defined by 
	\begin{eqnarray*}
		f(y)(x_{k})=
		\left\lbrace
		\begin{array}{lll}
			y_{k}                & \textup{if} & k<n+1,\\
			1-(y_{1}+...+y_{n})  & \textup{if} & k=n+1,
		\end{array}
		\right. 
	\end{eqnarray*}
	where $y=(y_{1},...,y_{n})\in \Delta_{n}$ and $x_{k}\in \Omega=\{x_{1},...,x_{n+1}\}$. A straightforward verification 
	shows that $f$ maps the topological interior of $\Delta_{n}$ onto
	$\mathcal{P}_{n+1}^{\times}=\{p\in \mathcal{P}_{n+1}\,\,\vert\,\, p(x)>0\,\,\forall x\in \Omega\}$, that is, 
	$f(\Delta_{n}^{\circ})=\mathcal{P}_{n+1}^{\times}$. 
	Because $\Delta_{n}$ is a closed convex set whose topological interior is nonempty, we have $\overline{\Delta_{n}^{\circ}}=\Delta_{n}$ 
	(see \cite{Rockafellar}, Theorem 6.3). It follows from this and Lemma \ref{nfekenkrnefkn} 
	that $\overline{\mathcal{P}_{n+1}^{\times}}= \overline{f(\Delta_{n}^{\circ})}=f(\overline{\Delta_{n}^{\circ}})=f(\Delta_{n})$, and so
	\begin{eqnarray*}
		\overline{\mathcal{P}_{n+1}^{\times}}=\mathcal{P}_{n+1}. 
	\end{eqnarray*}
\end{example}

	Throughout the rest of this section, we let $\pi=\kappa\circ \tau:T\mathcal{E}\to \mathcal{E}$ be a toric factorization. 
	Let $j:\mathcal{E}\to \mathcal{P}_{m}^{\times}$ be the inclusion map. Recall that the complex projective space $\mathbb{P}_{m-1}(1)$,
	equipped with the torus action $\Phi_{m-1}$, is a regular torification of 
	$\mathcal{P}_{m}^{\times}$ and that the map $\tau_{\textup{can}}:T\mathcal{P}_{m}^{\times}=\mathbb{C}^{m-1}
	\to \mathbb{P}_{m-1}(1)^{\circ}$ defined 
	by $\tau_{\textup{can}}(z)=\big[e^{z_{1}/2},...,e^{z_{m-1}/2},1\,\big]$ is a compatible covering map 
	(see Propositions \ref{nfeknwknwknk} and \ref{njewndknknknk}).
	Let $\kappa_{\textup{can}}:\mathbb{P}_{m-1}(1)^{\circ}\to \mathcal{P}_{m}^{\times}$ be the map defined by 
	$\kappa_{\textup{can}}([z])(x_{k})=\tfrac{|z_{k}|^{2}}{|z_{1}|^{2}+...+|z_{m}|^{2}}$. Then $(\kappa_{\textup{can}},\tau_{\textup{can}})$ 
	is a toric factorization. Let $\textup{lift}(j):N\to \mathbb{P}_{m-1}(1)$ be the lift of $j$ with respect to $\tau$ and 
	$\tau_{\textup{can}}$ (see Section \ref{neknkenkfnkn}). Let $K_{\textup{can}}:\mathbb{P}_{m-1}(1)\to \mathcal{P}_{m}$ 
	be the continuous map defined by 
	\begin{eqnarray}\label{nfewnkefnwkk}
		K_{\textup{can}}([z])(x_{k})=\dfrac{|z_{k}|^{2}}{|z_{1}|^{2}+...+|z_{m}|^{2}}, 
	\end{eqnarray}
	where $[z]=[z_{1},...,z_{m}]\in \mathbb{P}_{m-1}(1)$ and $x_{k}\in \Omega$. The following commutative diagram 
	summarizes the situation:
	\begin{eqnarray}\label{nfeknkrnkenk}
	\begin{tikzcd}
		N^{\circ} \arrow{r}{\displaystyle \textup{lift}(j)} \arrow[swap]{d}{\displaystyle \kappa} 
		&   \mathbb{P}_{m-1}(1)^{\circ}\arrow{d}{\displaystyle \kappa_{\textup{can}}} \arrow{r}{\displaystyle i_{1}}  
			&   \mathbb{P}_{m-1}(1)  
			\arrow{d}{\displaystyle K_{\textup{can}}}     \\
		 \mathcal{E} \arrow[swap]{r}{\displaystyle j} & \mathcal{P}_{m}^{\times} \arrow[swap]{r}{\displaystyle i_{2}}  
		 &   \mathcal{P}_{m}
	\end{tikzcd},
	\end{eqnarray}
	where $i_{1}$ and $i_{2}$ are inclusion maps. The next lemma follows immediately from \eqref{nfeknkrnkenk} and the fact that 
	$N^{\circ}$ is dense in $N$.

\begin{lemma}\label{nfekwnkefnkwdnk}
	$(K_{\textup{can}}\circ \textup{lift}(j))(N)\subseteq \overline{\mathcal{E}}$. 
\end{lemma}
	Thus, the composition $K_{\textup{can}}\circ \textup{lift}(j)$ defines a continuous map 
	$N\to \overline{\mathcal{E}}$ that we will denote by $K$.
	By construction, the following diagram commutes,
	\begin{eqnarray}\label{knknkmmm}
	\begin{tikzcd}
		N^{\circ} \arrow{r}{\displaystyle } \arrow[swap]{d}{\displaystyle \kappa} 
		&   N \arrow{d}{\displaystyle K}  \\
		 \mathcal{E} \arrow{r}{\displaystyle } & \overline{\mathcal{E}}
	\end{tikzcd},
	\end{eqnarray}
	where horizontal maps are inclusion maps. Furthermore, since $N$ is compact and $N^{\circ}$ is dense in $N$, Lemma \ref{nfekwnkefnkwdnk} 
	implies that 
	\begin{eqnarray*}
		K(N)=\overline{K(N^{\circ})}=\overline{\kappa(N^{\circ})}=\overline{\mathcal{E}}.
	\end{eqnarray*}
	Thus $K(N)=\overline{\mathcal{E}}$, that is, $K$ is surjective.

\begin{definition}
	Let $\mathcal{E},N,\kappa$ and $K$ be as above. 
	We shall say that the surjective map $K:N\to \overline{\mathcal{E}}$ is the \textit{natural extension} of 
	the compatible projection $\kappa.$
\end{definition}

	Clearly, $K$ is the unique continuous map $N\to \overline{\mathcal{E}}$ 
	satisfying $\textup{ev}_{x}\circ K=\textup{ev}_{x}\circ \kappa$ on $N^{\circ}$
	for all $x\in \Omega$, where $\textup{ev}_{x}:\mathcal{P}_{m}\to \mathbb{R}$ is the evaluation map, 
	that is, $\textup{ev}_{x}(p)=p(x)$. It is also worth noting that $\textup{ev}_{x}\circ K:N\to \mathbb{R}$ is 
	smooth for all $x\in \Omega$. 

%

\begin{example}
	If $\mathcal{E}=\mathcal{P}_{m+1}^{\times}$, then $\overline{\mathcal{E}}=\mathcal{P}_{m+1}$ and the 
	natural extension of $\kappa_{\textup{can}}$ is $K_{\textup{can}}$. 
\end{example}

\begin{example}\label{nekwnkefnk}
	If $\mathcal{E}=\mathcal{B}(n)$, then $\overline{\mathcal{E}}=\mathcal{B}(n)\cup\{\delta_{0},\delta_{n}\}$. 
	A regular torification of $\mathcal{B}(n)$ is $\mathbb{P}_{1}(\tfrac{1}{n})$, endowed with the torus action $\Phi_{1}$. 
	The map $\kappa:\mathbb{P}_{1}(\tfrac{1}{n})^{\circ}\to \mathcal{B}(n)$, defined by 
	$\kappa([z_{1},z_{2}])(k) =\binom{n}{k}\tfrac{|z_{1}^{k}z_{2}^{n-k}|^{2}}{(|z_{1}|^{2}+|z_{2}|^{2})^{n}}$, is a compatible projection. 
	Let $K:\mathbb{P}_{1}(\tfrac{1}{n})\to \overline{\mathcal{B}(n)}$ be the natural extension of $\kappa.$ 
	On the set $\mathbb{P}_{1}(\tfrac{1}{n})\backslash\mathbb{P}_{1}(\tfrac{1}{n})^{\circ}=\{[1,0],[0,1]\}$ where the torus action 
	is not free, $K$ is given by 
	\begin{eqnarray*}
		K([1,0])=\delta_{n}, \,\,\,\,\,\,\,\,\,K([0,1])=\delta_{0}. 
	\end{eqnarray*}
\end{example}

	We end this section with the following result, which follows immediately from (P8) and 
	the fact that $N^{\circ}$ is dense in $N$. 

\begin{lemma}\label{nfeknkrenknfk}
	Suppose that the estimator map $F:\Omega\to \mathbb{R}^{n}$ is injective. If $K$ and $K'$ are the natural 
	extensions of some compatible projections, then there is a permutation $\sigma\in \textup{Perm}(\mathcal{E})\subseteq 
	\mathbb{S}_{m}$ such that $K(p)(x_{k})=K'(p)(x_{\sigma(k)})$ for all $p\in N$ and all $k\in \{1,...,m\}$. 
\end{lemma}

\subsection{Momentum map.} In this section, we focus our attention on the momentum map $\moment:N\to \mathbb{R}^{n}$ associated to 
	the torus action $\Phi:\mathbb{T}^{n}\times N\to N$. Our purpose is to show Proposition \ref{nfeknekfnkn} below, which 
	expresses the fact that the coordinate functions of $\moment$ are expectations (in the probabilistic sense)
	of appropriate random variables on $\Omega=\{x_{1},...,x_{m}\}$. For the convenience of the reader, 
	we begin with a discussion of the notion of momentum map in the context of torification. The material
	is mostly taken from \cite{molitor-toric}.
	
	Let $\kappa:N^{\circ}\to \mathcal{E}$ be a compatible projection. 
	We assume that $\kappa$ is induced by the toric parametrization $(\mathscr{L},Z,F)$ 
	(see Definition \ref{fjenkrkefnkn}). Write $Z=(Z_{1},...,Z_{n})$, where each $Z_{i}$ is a vector field on $\mathcal{E}$, 
	and define the $n\times n$ invertible matrix $A=(a_{ij})$ via the formulas
	\begin{eqnarray*}
		Z_{i}=\sum_{j=1}^{n}a_{ij}\dfrac{\partial}{\partial \theta_{j}}, \,\,\,\,\,\,\,i\in \{1,...,n\}, 
	\end{eqnarray*}
	where the $\theta_{j}$'s are the natural parameters of $\mathcal{E}$ (see Setting \ref{mnekmnwknefknk}). Below 
	we will regard $A$ as a linear bijection from $\mathbb{R}^{n}$ to $\mathbb{R}^{n}$. 

	In \cite{molitor-toric}, it is proven that there is $C=(C_{1},...,C_{n})\in \mathbb{R}^{n}$ such that the following holds on $N^{\circ}$:
	\begin{eqnarray}\label{nfekdwnkrnfekn}
		\moment=-A\circ \eta\circ \kappa+C, 
	\end{eqnarray}
	where $\eta=(\eta_{1},...,\eta_{n}):\mathcal{E}\to \mathbb{R}^{n}$ is the map whose coordinate functions are 
	defined by 
	\begin{eqnarray*}
		\eta_{i}(p)=\mathbb{E}_{p}(F_{i})=\sum_{x\in \Omega}F_{i}(x)p(x). 
	\end{eqnarray*}
	In the statistical literature, the functions $\eta_{i}$'s are called \textit{expectation parameters} (see, e.g, 
	\cite{Amari-Nagaoka}). 

	Clearly, $\eta$ extends to a continuous map $\mathcal{P}_{m}\to\mathbb{R}^{n}$, also denoted by $\eta$, 
	by letting $\eta(p)=(\mathbb{E}_{p}(F_{1}),...,\mathbb{E}_{p}(F_{n}))$. Let $K:N\to \overline{\mathcal{E}}$ 
	be the natural extension of $\kappa$. In view of \eqref{nfekdwnkrnfekn}, the continuous 
	functions $\moment$ and $-A\circ \eta\circ K+C$ coincide on $N^{\circ}$, which is dense in $N$, and thus they 
	coincide on $N$. Thus, for any $p\in N$,  
	\begin{eqnarray*}
		\moment(p)=-(A\circ \eta\circ K)(p)+C=\big(\mathbb{E}_{K(p)}(X_{1}),...,\mathbb{E}_{K(p)}(X_{n})\big),
	\end{eqnarray*}
	where $X_{j}$ is the random variable on $\Omega$ defined by 
	\begin{eqnarray*}
		X_{j}:=-\sum_{k=1}^{n}a_{jk}F_{k}+C_{j}.
	\end{eqnarray*}
	Note that $X_{j}\in \mathcal{A}_{\mathcal{E}}$ for every $j\in \{1,...,n\}$, and that 
	the $X_{j}$'s are linearly independent. We thus have proved the existence part of the following lemma.

\begin{lemma}\label{nfeknkrkefnk}
	There are random variables $X_{1},...,X_{n}\in \mathcal{A}_{\mathcal{E}}$ such that 
	\begin{eqnarray}\label{nfekknrkefnwknfk}
		\moment(p)=\big(\mathbb{E}_{K(p)}(X_{1}),...,\mathbb{E}_{K(p)}(X_{n})\big) 
	\end{eqnarray}
	for all $p\in N$. Furthermore, the $X_{j}$'s in \eqref{nfekknrkefnwknfk} are unique. 
\end{lemma}

	The uniqueness assertion of Lemma \ref{nfeknkrkefnk} follows immediately 
	from Lemma \ref{newnkfenkwnk} (see below) and the well-known fact that the map $\eta$ 
	is a diffeomorphism from $\mathcal{E}$ onto the topological interior $\Delta^{\circ}$ of the 
	convex hull $\Delta\subset \mathbb{R}^{n}$ of the image of the estimator map $F:\Omega\to \mathbb{R}^{n}$
	(see, for example, \cite{Martin}, Proposition 3.2 and Theorem 3.3).

\begin{lemma}\label{newnkfenkwnk}
	For $p\in \mathcal{E}$, let $\alpha_{p}:\mathcal{A}_{\mathcal{E}}\to \mathbb{R}$ be the linear 
	form defined by $\alpha_{p}(X)=\mathbb{E}_{p}(X)$. Then the family $\{\alpha_{p}\}_{p\in \mathcal{E}}$ 
	separates the points of $\mathcal{A}_{\mathcal{E}}$, that is, if $\alpha_{p}(X)=0$ for all $p\in \mathcal{E}$, 
	where $X\in \mathcal{A}_{\mathcal{E}}$, then $X=0.$
\end{lemma}
\begin{proof}
	Let $X\in \mathcal{A}_{\mathcal{E}}$ be arbitrary. Since $\mathcal{A}_{\mathcal{E}}$ is spanned by $F_{1},...,F_{n},1$, 
	there are real numbers $a_{0},...,a_{n}$ such that $X=a_{0}+\sum_{i=1}^{n}a_{i}F_{i}$, and thus
	\begin{eqnarray*}
		\alpha_{p}(X)=a_{0}+\sum_{i=1}^{n}a_{i}\mathbb{E}_{p}(F_{i})=a_{0}+\sum_{i=1}^{n}a_{i}\eta_{i}(p)=
		a_{0}+\langle a,\eta(p)\rangle
	\end{eqnarray*}
	for all $p\in \mathcal{E}$, where $a:=(a_{1},...,a_{n})\in \mathbb{R}^{n}$ and 
	$\langle\,,\,\rangle$ is the Euclidean pairing in $\mathbb{R}^{n}$. 
	Now assume that $\alpha_{p}(X)=0$ for all $p\in \mathcal{E}$. Then 
	\begin{eqnarray}\label{fnewnjrkenjnfj}
		\langle a,u\rangle=-a_{0}
	\end{eqnarray}
	for all $u$ in $\eta(\mathcal{E})=\Delta^{\circ}$. Let $u_{0}\in \Delta^{\circ}$ and $v\in \mathbb{R}^{n}$ 
	be arbitrary. By continuity, $u_{0}+tv$ is in $\Delta^{\circ}$ for all $t$ in a 
	neighborhood of $0$ in $\mathbb{R}$. Using \eqref{fnewnjrkenjnfj} with $u_{0}+tv$ 
	in place of $u$, and taking the derivative with respect to $t$ at $t=0$, we obtain 
	$\langle a, v\rangle=0$ for all $v\in \mathbb{R}^{n}$. Thus $a=0.$ If $a=0$, then \eqref{fnewnjrkenjnfj} implies that $a_{0}=0.$ 
	The lemma follows. 
\end{proof}

\begin{example}\label{neknkenfkn}
	Recall that $\mathbb{P}_{n}(1)$, equipped with the torus action \eqref{neknknknknfdkn}, is a regular torification of 
	$\mathcal{P}_{n+1}^{\times}$. A momentum map is given by $\moment([z_{1},...,z_{n+1}])=
	-\tfrac{4\pi}{\langle z,z\rangle}(|z_{1}|^{2},...,|z_{n}|^{2})$, where $\langle\,,\,\rangle$ is the Hermitian 
	product in $\mathbb{C}^{n+1}$ (linear in the second entry). Let $K=K_{\textup{can}}:\mathbb{P}_{n}(1)\to \mathcal{P}_{n+1}$, 
	be defined by $K([z])(x_{k})=\tfrac{|z_{k}|^{2}}{\langle z,z\rangle}$ (it is the natural extension of some compatible projection). 
	Then 
	\begin{eqnarray*}
		\moment([z])&=& -4\pi \big(K([z])(x_{1}),...,K([z])(x_{n})\big)\\
		&=& -4\pi \big(\mathbb{E}_{K([z])}(\delta_{x_{1}}),...,\mathbb{E}_{K([z])}(\delta_{x_{n}})\big), 
	\end{eqnarray*}
	where $\delta_{x_{j}}:\Omega\to \mathbb{R}$ is defined by $\delta_{x_{j}}(x_{i})=1$ if $j=i$, 0 otherwise. Therefore 
	the $X_{j}$'s are given by $X_{j}=-4\pi \delta_{x_{j}}$. 
\end{example}

\begin{example}\label{nfekwnknfekkn}
	Recall that $\mathbb{P}_{1}(\tfrac{1}{n})$, equipped with the torus action \eqref{neknknknknfdkn}, 
	is a regular torification of $\mathcal{B}(n)$. A momentum map is given by 
	$\moment([z_{1},z_{2}])=-4\pi n \tfrac{|z_{1}|^{2}}{\langle z,z\rangle}$. 
	Let $K:\mathbb{P}_{1}(\tfrac{1}{n})\to \overline{\mathcal{B}_{n}}$, be the natural extension of the compatible projection
	$\kappa([z])(k) =\binom{n}{k}\tfrac{|z_{1}^{k}z_{2}^{n-k}|^{2}}{(|z_{1}|^{2}+|z_{2}|^{2})^{n}}$, where 
	$[z]=[z_{1},z_{2}]\in \mathbb{P}_{1}(\tfrac{1}{n})^{\circ}$ and $k\in \Omega=\{0,1,...,n\}$. Clearly, 
	$K([z])(k)= \binom{n}{k}p^{k}(1-p)^{n-k}$, where 
	$p=\tfrac{|z_{1}|^{2}}{(|z_{1}|^{2}+|z_{2}|^{2})}$. Thus $K([z])$ is a binomial distribution 
	with parameter $p$. In particular, the mean of the identity function $\textup{Id}$ on $\Omega$ with respect to 
	$K([z])$ is given by 
	\begin{eqnarray*}
		\mathbb{E}_{K([z])}(\textup{Id})=np=\dfrac{n|z_{1}|^{2}}{(|z_{1}|^{2}+|z_{2}|^{2})}=
		\dfrac{n|z_{1}|^{2}}{\langle z,z\rangle}=-\dfrac{1}{4\pi}\moment([z]).
	\end{eqnarray*}
	It follows that $\moment([z])=\mathbb{E}_{K([z])}(X)$, where $X=-4\pi \textup{Id}$. 
\end{example}

	For our purposes, it is convenient to reformulate Lemma \ref{newnkfenkwnk} in the form:
\begin{proposition}\label{nfeknekfnkn}
	Given a compatible projection $\kappa:N^{\circ}\to \mathcal{E}$, with natural extension $K$, there is a unique linear bijection 
	$Z:\mathbb{R}^{n}\oplus \mathbb{R}\to \mathcal{A}_{\mathcal{E}}$ such that 
	\begin{eqnarray}\label{jnekjneknfknk}
		\moment^{\xi}(p)+r= \mathbb{E}_{K(p)}\big(Z(\xi,r)\big)
	\end{eqnarray}
	for all $p\in N$, $\xi\in \mathbb{R}^{n}$ and $r\in \mathbb{R}$. 
	If $X_{1},...,X_{n}$ are the random variables described in Lemma \ref{nfeknkrkefnk}, then 
	$Z(\xi,r)=\sum_{j=1}^{n}\xi_{j}X_{j}+r$.
\end{proposition}

\begin{definition}
	We shall call the linear bijection $Z:\mathbb{R}^{n}\oplus \mathbb{R}\to \mathcal{A}_{\mathcal{E}}$ of 
	Proposition \ref{nfeknekfnkn} the \textit{kernel of} $\moment$ \textit{relative to} $K$.
\end{definition}

\begin{example}\label{nfewnkefnkwnken}
	Suppose $N=\mathbb{P}_{n}(1)$ and $\mathcal{E}=\mathcal{P}_{n+1}^{\times}$. In this case, 
	the kernel of $\moment([z])=-\tfrac{4\pi}{\langle z,z\rangle}(|z_{1}|^{2},...,|z_{n}|^{2})$ relative to 
	$K([z])(x_{k})=\tfrac{|z_{k}|^{2}}{\langle z,z\rangle}$ is given by 
	\begin{eqnarray*}
		Z(\xi,r)(x_{k})=\left\lbrace
		\begin{array}{lll}
			-4\pi\xi_{k}+r  & \textup{if}  & k<n+1, \\
			r               & \textup{if}  & k=n+1,
		\end{array}
		\right.
	\end{eqnarray*}
	where $(\xi,r)\in \mathbb{R}^{n}\oplus\mathbb{R}$ and $x_{k}\in \{x_{1},...,x_{n+1}\}$. 
\end{example}

\begin{example}\label{jkwwjefjwkdjk}
	Suppose $N=\mathbb{P}_{1}(\tfrac{1}{n})$ and $\mathcal{E}=\mathcal{B}(n)$. In this case, 
	the kernel of $\moment([z_{1},z_{2}])=-4\pi n \tfrac{|z_{1}|^{2}}{\langle z,z\rangle}$ relative to 
	$K([z_{1},z_{2}])(k) =\binom{n}{k}\tfrac{|z_{1}^{k}z_{2}^{n-k}|^{2}}{(|z_{1}|^{2}+|z_{2}|^{2})^{n}}$ is given by 
	\begin{eqnarray*}
		Z(\xi,r)(k)=-4\pi \xi k+r,	
	\end{eqnarray*}
	where $(\xi,r)\in \mathbb{R}\oplus\mathbb{R}$ and $k\in \{0,1,...,n\}$. 
\end{example}

	Consider the group action of $\mathbb{S}_{m}$ on the space of real-valued functions $h$ on $\Omega$ given by 
	$(\sigma\cdot h)(x_{k})=h(x_{\sigma^{-1}(k)})$. It is easy to see that 
	\begin{eqnarray*}
		\mathbb{E}_{\sigma\cdot p}(X)=\mathbb{E}_{p}(\sigma^{-1}\cdot X)
	\end{eqnarray*}
	for all probability functions $p$ on $\Omega$ and all random variables $X:\Omega\to \mathbb{R}$. Using this and 
	Lemma \ref{newnkfenkwnk}, a straightforward computation shows the following result. 

\begin{lemma}\label{nefkwnefnknkn}
	Suppose that the estimator map $F$ is injective. 
	Let $Z$ and $Z'$ be the kernels of $\moment$ relative to $K$ and $K'$, respectively, and let 
	$\sigma\in \textup{Perm}(\mathcal{E})\subseteq \mathbb{S}_{m}$ be a permutation such that 
	$K'(p)=\sigma\cdot K(p)$ for all $p\in N$ (see Lemma \ref{nfeknkrenknfk}). 
	Then $Z(\xi,r)=\sigma^{-1}\cdot Z'(\xi,r)$ for all $(\xi,r)\in \mathbb{R}^{n}\oplus \mathbb{R}$. 
\end{lemma}

	Therefore, if $F$ is injective, then the images of $Z(\xi,r)$ and $Z'(\xi,r)$ coincide for all 
	$(\xi,r)\in \mathbb{R}^{n}\oplus \mathbb{R}$. This fact will be important in the definition of the spectrum of 
	a K\"{a}hler function (Definition \ref{nfeknkenkenkf}).

\subsection{Weyl group of $\textup{Aut}(N,g)^{0}$ and equivariance of $Z$.}\label{nekenrkefnknk}

	We now discuss symmetry properties. Throughout this section, we assume that the 
	estimator map $F:\Omega\to \mathbb{R}^{n}$ is injective, or equivalently that 
	$\mathcal{A}_{\mathcal{E}}$ separates the points of $\Omega$ (see Remark \ref{nemwwnkefnkwnk}). Thus, by (P6), the group 
	$\textup{Diff}(\mathcal{E},h_{F},\nabla^{(e)})$ of affine isometries of $\mathcal{E}$ 
	coincides with the group $\textup{Perm}(\mathcal{E})$ of permutations of $\mathcal{E}$ 
	(see Section \ref{nkennkfenknk}). Recall that every element of $\textup{Perm}(\mathcal{E})$ is of the form 
	$T_{\sigma}$ for some $\sigma\in \mathbb{S}_{m}$, where $T_{\sigma}:\mathcal{E}\to \mathcal{E}$ 
	is defined by $T_{\sigma}(p)(x_{k})=p(x_{\sigma(k)}).$ 
	Since $F$ is injective, the map $\textup{Perm}(\mathcal{E})\to \mathbb{S}_{m}$, $T_{\sigma}\to \sigma^{-1}$ 
	defines a group isomorphism from $\textup{Perm}(\mathcal{E})$ to a subgroup of $\mathbb{S}_{m}$. Because of this, we will often 
	regard $\textup{Perm}(\mathcal{E})$ as a subgroup of $\mathbb{S}_{m}$. 
	
	Recall that $T=\Phi(\mathbb{T}^{n})$ is a maximal torus and a Cartan subgroup of $\textup{Aut}(N,g)$ (see Section \ref{nkennkfenknk}). 
	Let $W^{S}$ be the Weyl group (in the sense of Segal) of $\textup{Aut}(N,g)$ associated to $T$. 
	Let $\pi=\kappa\circ \tau$ be a toric factorization. By (P4) and (P6), 
	the map 
	\begin{eqnarray}\label{nfnfkenknknk}
		\textup{Perm}(\mathcal{E})\to W^{S}, \,\,\,\,\,\,\,\, 
		\sigma\mapsto [\textup{lift}_{\tau}(T_{\sigma^{-1}})],
	\end{eqnarray}
	is a group isomorphism, where $[\textup{lift}_{\tau}(T_{\sigma^{-1}})]$ 
	denotes the equivalence class of $\textup{lift}_{\tau}(T_{\sigma^{-1}})\in N(T)
	\subset \textup{Aut}(N,g)$ in $W^{S}$. Let $W$ be the Weyl group (in the classical sense) 
	of $\textup{Aut}(N,g)^{0}$ associated to $T$. Then $W$ is naturally a subgroup of $W^{S}$. 
	Moreover, if $\textup{Perm}(\mathcal{E})^{0}$ denotes the subgroup of $\textup{Perm}(\mathcal{E})$ 
	consisting of elements $\psi\in \textup{Perm}(\mathcal{E})$ satisfying $\textup{lift}_{\tau}(\psi)\in \textup{Aut}(N,g)^{0}$, 
	then \eqref{nfnfkenknknk} restricts to a group isomorphism $\textup{Perm}(\mathcal{E})^{0}\to W$, making the following 
	diagram commutative:
\begin{eqnarray}\label{knkdfnknkndknk}
	\begin{tikzcd}
		\textup{Perm}(\mathcal{E})^{0} \arrow{r}{\displaystyle \cong } \arrow{d}{\displaystyle } &  W \arrow{d}{\displaystyle } \\
		 \textup{Perm}(\mathcal{E}) \arrow[swap]{r}{\displaystyle \cong}  &  W^{S}
	\end{tikzcd}
\end{eqnarray}
	where vertical arrows are inclusion maps. 

	Let $Z:\mathbb{R}^{n}\oplus \mathbb{R}\to \mathcal{A}_{\mathcal{E}}$ be the kernel of $\moment$
	relative to $K:N\to \overline{\mathcal{E}}$. Recall that $W$ acts linearly on $\mathbb{R}^{n}\oplus \mathbb{R}$ 
	via the formula $w\cdot (\xi,r)=(w\cdot \xi,r+\sigma_{\moment}(w)(\xi))$, where $\sigma_{\moment}:W\to (\mathbb{R}^{n})^{*}$ 
	is the one-cocycle of $\moment$. Since $Z$ is a linear bijection, there is a unique linear action of $W$ on 
	$\mathcal{A}_{\mathcal{E}}$ such that 
	\begin{eqnarray}\label{ndefwnkefndknkenk}
		Z\big(w\cdot (\xi,r)\big)=w\cdot Z(\xi,r)
	\end{eqnarray}
	for all $w\in W$ and $(\xi,r)\in \mathbb{R}^{n}\oplus \mathbb{R}$. 

\begin{lemma}\label{ncdknknfeknk}
	The action of the Weyl group $W\subseteq \mathbb{S}_{m}$ on $\mathcal{A}_{\mathcal{E}}$ is given by:
	\begin{eqnarray*}
		(\sigma\cdot f)(x_{k})=f(x_{\sigma^{-1}(k)}),
	\end{eqnarray*}
	where $\sigma\in W$, $f\in \mathcal{A}_{\mathcal{E}}$ and $x_{k}\in \Omega$. 
\end{lemma}
\begin{proof}
	Let $w=[\phi]\in W$ be arbitrary, where $\phi\in N(T)\subseteq \textup{Aut}(N,g)^{0}$. 
	Suppose that $p\in N$ and $(\xi,r)\in \mathbb{R}^{n}\oplus \mathbb{R}$ are arbitrary. 
	Using \eqref{nekdnkfnnkenefkkn} and the fact that $Z$ is a kernel, we compute:
	\begin{eqnarray}
		\mathbb{E}_{K(p)}\big(Z(w\cdot(\xi,r))\big) &=&  \mathbb{E}_{K(p)}\big(Z(w\cdot \xi,r+\sigma_{\moment}(w)(\xi))\big)\nonumber \\
		&=& \moment^{w\cdot \xi}(p)+r+\sigma_{\moment}(w)(\xi))\nonumber\\
		&=& \moment^{\xi}(\phi^{-1}(p))+r\nonumber\\
		&=& \mathbb{E}_{K(\phi^{-1})(p)}(Z(\xi,r)).\label{nkwnkefnwknk}
	\end{eqnarray}
	By the discussion above, there is a permutation $\sigma\in \mathbb{S}_{m}$ such that $\phi=\textup{lift}_{\tau}(T_{\sigma^{-1}})$ and 
	$T_{\sigma^{-1}}\in \textup{Perm}(\mathcal{E})^{0}$. By Proposition \ref{nfeknwkneknk}, the map 
	$\textup{Diff}(\mathcal{E},h_{F},\nabla^{(e)})\to \textup{Aut}(N,g)$, $\psi\mapsto \textup{lift}_{\tau}(\psi)$, is a 
	group homomorphism, and thus
	\begin{eqnarray*}
		\phi^{-1}=\big(\textup{lift}_{\tau}(T_{\sigma^{-1}})\big)^{-1}=\textup{lift}_{\tau}\big((T_{\sigma^{-1}})^{-1}\big)
		= \textup{lift}_{\tau}(T_{\sigma}). 
	\end{eqnarray*}
	It follows from this and the fact that $K$ is the natural extension of a compatible projection that 
	\begin{eqnarray*}
		K(\phi^{-1}(p)) =K\big(\textup{lift}_{\tau}(T_{\sigma})(p)) = (T_{\sigma}\circ K)(p).
	\end{eqnarray*}
	Thus
	\begin{eqnarray}
		\mathbb{E}_{K(\phi^{-1})(p)}(Z(\xi,r))&=& \mathbb{E}_{(T_{\sigma}\circ K)(p)}(Z(\xi,r))\nonumber\\
		&=& \sum_{k=1}^{m}Z(\xi,r)(x_{k})(T_{\sigma}\circ K)(p)(x_{k})\nonumber\\
		&=& \sum_{k=1}^{m}Z(\xi,r)(x_{k})K(p)(x_{\sigma(k)})\nonumber\\
		&=& \sum_{k=1}^{m}Z(\xi,r)(x_{\sigma^{-1}(k)})K(p)(x_{k})\nonumber\\
		&=& \mathbb{E}_{K(p)}(Z^{\sigma}(\xi,r)), \label{nfekdkenfknk}
	\end{eqnarray}
	where $Z^{\sigma}(\xi,r)$ is defined by $Z^{\sigma}(\xi,r)(x_{k})=Z(\xi,r)(x_{\sigma^{-1}(k)})$. 
	It follows from \eqref{nkwnkefnwknk} and \eqref{nfekdkenfknk} that 
	\begin{eqnarray}\label{nefknknfeknk}
		\mathbb{E}_{K(p)}\big(Z(w\cdot(\xi,r))\big)=\mathbb{E}_{K(p)}(Z^{\sigma}(\xi,r)).
	\end{eqnarray}
	By (P7), we have $Z^{\sigma}\in \mathcal{A}_{\mathcal{E}}$, which implies, in view of 
	\eqref{nefknknfeknk} and Lemma \ref{newnkfenkwnk}, that $Z(w\cdot(\xi,r))=Z^{\sigma}(\xi,r)$. This completes the proof.
\end{proof}

\subsection{Main result.} In this section, we state and prove the main result of this paper, and define the spectrum 
	of a K\"{a}hler function.

\begin{theorem}\label{newnknfkenkndvknfk}
	Suppose that $\Phi:\mathbb{T}^{n}\times N\to N$ is a K\"{a}hler toric manifold and a regular torification 
	of an exponential family $\mathcal{E}$ defined over a finite set $\Omega=\{x_{1},...,x_{m}\}$, as in 
	Setting \ref{mnekmnwknefknk}. Let $K:N\to \overline{\mathcal{E}}$ be the natural extension of some 
	compatible projection. Given a smooth function $f:N\to \mathbb{R}$, the following hold. 
	\begin{enumerate}[(a)]
		\item $f$ is K\"{a}hler if and only if there is a random variable $X\in \mathcal{A}_{\mathcal{E}}$ 
			and a holomorphic isometry $\phi\in \textup{Aut}(N,g)^{0}$ such that 
			$f(p)=\mathbb{E}_{K(\phi(p))}(X)$ for all $p\in N$. 
		\item Suppose that the estimator map $F:\Omega\to \mathbb{R}^{n}$ is injective. 
			If $f(p)=\mathbb{E}_{K(\phi(p))}(X)=\mathbb{E}_{K(\phi'(p))}(X')$ for all $p\in N$, where 
			$\phi,\phi'\in \textup{Aut}(N,g)^{0}$ and $X,X'\in \mathcal{A}_{\mathcal{E}}$, then 
			there is a permutation $\sigma\in \textup{Perm}(\mathcal{E})\subseteq \mathbb{S}_{m}$ such that 
			$X(x_{k})=X'(x_{\sigma(k)})$ for all $k\in \{1,...,m\}$. 
	\end{enumerate}
\end{theorem}
\begin{proof}
	(a) This follows from Proposition \ref{nenekfnsddnekn} and Lemma \ref{nfeknekfnkn}. (b) Let
	$Z$ be the kernel of $\moment$ relative to $K$. 
	Write $X=Z(\xi,r)$ and $X'=Z(\xi',r')$, where $(\xi,r)$ and $(\xi',r')$ are in $\mathbb{R}^{n}\oplus \mathbb{R}$.
	Then
	\begin{eqnarray*}
		\mathbb{E}_{K(\phi(p))}(X)=\mathbb{E}_{K(\phi'(p))}(X') \,\,\,\,\,\,\,&\Rightarrow& \,\,\,\,\,\,\,
		\mathbb{E}_{K(\phi(p))}(Z(\xi,r))=\mathbb{E}_{K(\phi'(p))}(Z(\xi',r'))\\
		&\Rightarrow& \,\,\,\,\,\,\,\moment^{\xi}(\phi(p))+r=\moment^{\xi'}(\phi'(p))+r'.
	\end{eqnarray*}
	It follows from this and Proposition \ref{nenekfnsddnekn}(b) that 
	there is $w\in W$ such that $(\xi,r)=w\cdot (\xi',r')$, 
	where $W$ is the Weyl group of $\textup{Aut}(N,g)^{0}$ associated to the maximal torus $T=\Phi(\mathbb{T}^{n})$. 
	Since $Z$ is $W$-equivariant (see \eqref{ndefwnkefndknkenk}), we get 
	\begin{eqnarray*}
		X=Z(\xi,r)=Z(w\cdot (\xi',r'))=w\cdot Z(X',r')=w\cdot X'.
	\end{eqnarray*}
	Thus $X=w\cdot X'$. Since the estimator map is injective, $W$ acts on $\mathcal{A}_{\mathcal{E}}$ by permutations 
	(see Lemma \ref{ncdknknfeknk}). This means that there is $\sigma\in \mathbb{S}_{m}$ such that $(w\cdot X')(x_{k})=X'(x_{\sigma(k)})$ 
	for all $k\in \{1,...,m\}$. Therefore $X(x_{k})=X'(x_{\sigma(k)})$ for all $k\in \{1,...,m\}$, which completes the proof.
\end{proof}

\begin{corollary}
	Suppose that the estimator map $F$ is injective. Let $f$ be a K\"{a}hler function on $N$ such that 
	$f(p)=\mathbb{E}_{K(\phi(p))}(X)=\mathbb{E}_{K'(\phi'(p))}(X')$ for all $p\in N$, 
	where $\phi,\phi'\in \textup{Aut}(N,g)^{0}$, $X,X'\in \mathcal{A}_{\mathcal{E}}$, and where $K$ and $K'$ are 
	the natural extensions of some compatible projections. Then the images of $X$ and $X'$ coincide. 
\end{corollary}
\begin{proof}
	This follows from Lemmas \ref{nfeknkrenknfk} and \ref{nefkwnefnknkn}, and Theorem \ref{newnknfkenkndvknfk}(b).
\end{proof}

\begin{definition}[\textbf{Spectrum of a K\"{a}hler function}]\label{nfeknkenkenkf}
	Suppose that the estimator map $F$ is injective. 
	Let $f(p)=\mathbb{E}_{K(\phi(p))}(X)$ be a K\"{a}hler function on $N$, where 
	$X\in \mathcal{A}_{\mathcal{E}}$, $\phi\in \textup{Aut}(N,g)^{0}$ and 
	$K$ is the natural extension of some compatible projection. 
	We shall call the image of $X$ the \textit{spectrum} of $f$, and denote it by $\textup{spec}(f)$. Thus
	\begin{eqnarray*}
		\textup{spec}(f):=\{X(x_{1}),...,X(x_{m})\}. 
	\end{eqnarray*}
\end{definition}

	We now focus our attention on two concrete examples. Let $\textup{Herm}_{n}$ denote the space 
	of Hermitian matrices of size $n$, and let $ \mathscr{K}(\mathbb{P}_{n}(1))$ be the space of 
	K\"{a}hler functions on $\mathbb{P}_{n}(1)$. Recall that the map 
	$\textup{Herm}_{n+1}\to \mathscr{K}(\mathbb{P}_{n}(1))$, $H\mapsto f_{H}$, is a 
	linear bijection, where $f_{H}([z]):= \tfrac{\langle z,Hz\rangle}{\langle z,z\rangle}.$ Below, we use 
	$\textup{spec}(H)$ to denote the spectrum (in the classical sense) of $H\in \textup{Herm}_{n+1}$.

\begin{proposition}\label{neeknereknfk}
	Let $\mathbb{P}_{n}(1)$ be the complex projective space of holomorphic sectional curvature $c=1$, 
	regarded as a regular torification of $\mathcal{P}_{n+1}^{\times}$. 
	Then, for every $H\in \textup{Herm}_{n+1}$, we have $\textup{spec}(f_{H})=\textup{spec}(H)$. 
\end{proposition}
\begin{proof}
	Let $f$ be any K\"{a}hler function on $\mathbb{P}_{n}(1)$. Let $\moment:\mathbb{P}_{n}(1)\to \mathbb{R}^{n}$ 
	be the momentum map defined by $\moment([z])=-\tfrac{4\pi}{\langle z,z\rangle}(|z_{1}|^{2},...,|z_{n}|^{2})$. 
	By Proposition \ref{nenekfnsddnekn}, there are 
	$\phi\in \textup{Aut}(\mathbb{P}_{n}(1))^{0}$ and $(\xi,r)\in \mathbb{R}^{n}\oplus 
	\mathbb{R}$ such that $f([z])=\moment^{\xi}(\phi([z]))+r$ for all $[z]\in \mathbb{P}_{n}(1)$. 
	Let $Z$ be the kernel of $\moment$ relative to $K([z])(x_{k})=\tfrac{|z_{k}|^{2}}{\langle z,z\rangle}.$
	By definition of a kernel, we have $f([z])=\mathbb{E}_{K(\phi([z]))}(Z(\xi,r))$, and thus the spectrum of 
	$f$ is the image of $Z(\xi,r)$, which is given, in view of Example \ref{nfewnkefnkwnken}, by:
	\begin{eqnarray}\label{nekkwnkfnek}
		\textup{spec}(f)=\textup{image of}\,\,Z(\xi,r)=\{-4\pi\xi_{1}+r,...,-4\pi\xi_{n}+r,r\}. 
	\end{eqnarray}
	On the other hand, $\phi$ must be of the form $\phi([z])=[Uz]$ for some 
	$U\in SU(n+1),$ and thus Proposition \ref{neknkefnwknfekn} implies that $f=f_{H}$, where 
	$H=U^{*}DU$ and $D=\textup{diag}(-4\pi\xi_{1}+r,...,-4\pi\xi_{n}+r,r)$ (obvious notation). 
	Clearly, the spectrum of $H$ coincides with \eqref{nekkwnkfnek}. The result follows.
\end{proof}

	We now focus our attention on the case $\mathcal{E}=\mathcal{B}(n)$. 
	Let us identify the unite sphere $S^{2}\subset \mathbb{R}^{3}$ with $\mathbb{P}_{1}(\tfrac{1}{n})$ via the map 
	$\psi:S^{2}\to \mathbb{P}_{1}(\tfrac{1}{n})$ defined by 
	\begin{eqnarray}\label{nefknekfnknk}
		\psi(x,y,z)=\big[\cos(\beta/2)e^{i\alpha},\sin(\beta/2)\big],
	\end{eqnarray}
	where $S^{2}$ is parametrized by spherical coordinates:
	\begin{eqnarray*}
		\left \lbrace
		\begin{array}{lll}
			x &=& \cos(\alpha)\sin(\beta), \\
			y &=& \sin(\alpha)\sin(\beta), \\
			z &=& \cos(\beta), 
		\end{array}
		\right. 
		 \,\,\,\,\,\,\,\,\,\,\,\alpha\in [0,2\pi],\,\, \beta\in [0,\pi].
	\end{eqnarray*}
	Below we will regard $S^{2}$ as a regular torification of $\mathcal{B}(n)$. Under the identification 
	$S^{2}\cong \mathbb{P}_{1}(\tfrac{1}{n})$, 
	the natural extension $K$ of Example \ref{nekwnkefnk} and the momentum map $\moment$ of Example \ref{nfekwnknfekkn} read
	\begin{eqnarray}\label{jfekwnknefkn}
		\left\lbrace
		\begin{array}{lll}
			K(x,y,z)(k)    &=& 2^{-n}\binom{n}{k}(1+z)^{k}(1-z)^{n-k}, \\[0.4em]
			\moment(x,y,z) &=& -2\pi n(1+z),
		\end{array}
		\right.
	\end{eqnarray}
	where $(x,y,z)\in S^{2}$ and $k\in \{0,1,...,n\}$. (By convention $0^{0}=1$, and so 
	$K(0,0,1)=\delta_{n}$ and $K(0,0,-1)=\delta_{0}$, where $\delta_{j}:\Omega\to \mathbb{R}$ is defined by $\delta_{j}(i)=1$ if 
	$i=j$, $0$ otherwise.) The K\"{a}hler metric on $S^{2}\cong \mathbb{P}_{1}(\tfrac{1}{n})$ is $ng$, 
	where $g$ is the Riemannian metric on $S^{2}$ induced from the Euclidean metric on $\mathbb{R}^{3}$. The isometry group 
	of $(S^{2},ng)$ is the orthonormal group $O(3)$. Since $S^{2}$ is compact, the identity component of its group 
	of holomorphic isometries coincides with the identity component of the group of isometries of $S^{2}$ 
	(see Proposition \ref{jefknkefknk}). Thus $\textup{Aut}(S^{2},ng)^{0}=O(3)^{0}=SO(3)$. 
	Regarding K\"{a}hler functions, we have the following result.

\begin{proposition}\label{nfekwnkefnknkn}
	The space of K\"{a}hler functions on $\mathbb{P}_{1}(\tfrac{1}{n})\cong S^{2}$ is spanned by the coordinate functions 
	$x,y,z$ and the constant function $1$. If $\alpha,\beta,\gamma,\delta$ are real numbers, then 
	the spectrum of $f(x,y,z)=\alpha x+\beta y+\gamma z+\delta$ is given by:
	\begin{eqnarray}\label{nekwwnknefk}
		\textup{spec}(f)=\bigg\{\dfrac{2}{n}\|(\alpha,\beta,\gamma)\|\bigg(-\dfrac{n}{2}+k\bigg)+\delta\,\,\bigg\vert\,\,k=0,1,...,n
		\bigg\},
	\end{eqnarray}
	where $\|(\alpha,\beta,\gamma)\|$ is the Euclidean norm of $(\alpha,\beta,\gamma)\in \mathbb{R}^{3}$. In particular, 
	if $\|(\alpha,\beta,\gamma)\|=\tfrac{n}{2}=:j$ and $\delta=0$, then $\textup{spec}(f)=\{-j,-j+1,...,j-1,j\}$.
\end{proposition}
\begin{proof}
	Let $f:S^{2}\to \mathbb{R}$ be any K\"{a}hler function. By Proposition \ref{nenekfnsddnekn}, there are $U=(U_{ij})\in SO(3)$ and 
	$(\xi,r)\in \mathbb{R}^{2}$ such that $f=\moment^{\xi}\circ \phi_{U}+r$, where $\phi_{U}$ is the holomorphic isometry 
	on $S^{2}$ defined by $\phi_{U}(p)=Up$. In view of \eqref{jfekwnknefkn}, we have 
	\begin{eqnarray}\label{nekwnkefnkn}
		f(x,y,z)&=&-2\pi n\xi (1+U_{31}x+U_{32}y+U_{33}z)+r=\alpha x +\beta y +\gamma z+\delta, 
	\end{eqnarray}
	where 
	\begin{eqnarray}\label{nfeknknefkd}
	\left\lbrace
		\begin{array}{lll}
		\alpha &= & -2\pi n \xi U_{31}, \\
		\beta  &= &-2\pi n \xi U_{32},  \\
		\gamma &= &-2\pi n \xi U_{33},  \\
		\delta &=&-2\pi n \xi +r.
		\end{array}
	\right.
	\end{eqnarray}
	Note that 
	\begin{eqnarray*}
		\|(\alpha,\beta,\gamma)\|^{2}=4\pi^{2}n^{2}\xi^{2}(U_{31}^{2}+U_{32}^{2}+U_{33}^{2})=4\pi^{2}n^{2}\xi^{2}
	\end{eqnarray*}
	and hence $\|(\alpha,\beta,\gamma)\|=2\pi n|\xi|$. It follows from \eqref{nekwnkefnkn} that every K\"{a}hler function 
	on $S^{2}$ is a linear combination of $1,x,y,z$. By choosing appropriately $U$ and $(\xi,r)$, 
	one sees that the functions $1,x,y,z$ are all K\"{a}hler functions on $S^{2}$. Therefore 
	the space of K\"{a}hler functions on $S^{2}$ is precisely the vector space spanned by $1,x,y,z$.

	Let $Z$ be the kernel of $\moment$ relative to $K$. 
	Thus $f(p)=\mathbb{E}_{K(\phi_{U}(p))}(Z(\xi,r))$ for all $p\in S^{2}$ and the spectrum of $f$ is the image of $Z(\xi,r)$. 
	By Example \ref{jkwwjefjwkdjk}, we have $Z(\xi,r)(k)=-4\pi\xi k+r$ for all $k=0,1,...,n$, and since 
	$\delta =-2\pi n \xi +r$ (see \eqref{nfeknknefkd}), we get 
	\begin{eqnarray}\label{nfekwnkfenkn}
		Z(\xi,r)(k)=4\pi \xi (n/2-k)+\delta.
	\end{eqnarray}
	Note that $Z(\xi,r)(k)=Z(-\xi,r)(n-k)$ for all $k=0,1,...,n$, which 
	implies that $\textup{spec}(f)$ does not depend on the sign of $\xi.$ It follows that 
	\begin{eqnarray}
		\textup{spec}(f) = \textup{image of}\,\,Z(|\xi|,r)= \{4\pi |\xi| (n/2-k)+\delta\,\,\vert\,\,k=0,1,...,n\}.
	\end{eqnarray}
	The rest of the proof follows from this and the formula $\|(\alpha,\beta,\gamma)\|=2\pi n|\xi|$.
\end{proof}

\section{Cram\'er-Rao equality}\label{nfeknkefnknkn} 

	
	In this section, we prove an analogue of a result from information geometry known as the 
	\textit{Cram\'er-Rao equality} (see Theorem \ref{nceknkenknknk} and Lemma \ref{nenkenkfnknk}), 
	and discuss two applications. 

	We assume throughout that $\Phi:\mathbb{T}^{n}\times N\to N$ is a K\"{a}hler toric manifold and a regular torification of an 
	exponential family $\mathcal{E}$ defined over a finite set $\Omega=\{x_{1},...,x_{m}\}$, as in Setting \ref{mnekmnwknefknk}. 
	Unless we specify otherwise, we do not require the estimator map $F$ to be injective. 
	Let $\moment:N\to \mathbb{R}^{n}$ be a momentum map associated to the torus action.

\subsection{Cram\'er-Rao equality.}

	Given a probability function $p:\Omega\to [0,1]$ and a random variable $X:\Omega\to \mathbb{R}$, 
	we will denote by $\mathbb{V}_{p}(X)$ the variance of $X$ with respect to $p$. Thus 
	\begin{eqnarray*}
		\mathbb{V}_{p}(X)=\mathbb{E}_{p}\big((X-\mathbb{E}_{p}(X))^{2}\big).
	\end{eqnarray*}
	Given a Riemannian manifold $(M,h)$, we shall use $\textup{grad}^{h}(f)$ to denote the Riemannian 
	gradient of a function $f:M\to \mathbb{R}$ (thus $h(\textup{grad}^{h}(f)(p),u)=df(u)$ for all $u\in TM$). 
	If $u\in TM$, then $\|u\|:=\sqrt{h(u,u)}$.

\begin{theorem}[\textbf{Cram\'er-Rao equality, toric version}]\label{nceknkenknknk}
	Let the hypotheses be as in Setting \ref{mnekmnwknefknk}. Let $f(p)=\mathbb{E}_{K(\phi(p))}(X)$ be a K\"{a}hler function on $N$, where 
	$K$ is the natural extension of some compatible projection, $\phi\in \textup{Aut}(N,g)^{0}$ and $X\in \mathcal{A}_{\mathcal{E}}$. 
	The following identity holds for all $p\in N$:
	\begin{eqnarray}\label{nkwwnkenfkn}
		\mathbb{V}_{K(\phi(p))}(X)=\|\,\textup{grad}^{g}(f)(p)\,\|^{2}.
	\end{eqnarray}
\end{theorem}

	The proof is broken up into a few lemmas. 

\begin{lemma}\label{nckndwkenknkennk}
	Suppose that $(M,h,\nabla)$ is a dually flat manifold. Let $g$ be the K\"{a}hler metric on $TM$ 
	associated to Dombrowski's construction (see Section \ref{nkwwnkwnknknfkn}), and let $\pi:TM\to M$ denote the canonical projection. 
	Given a smooth function $f:M\to \mathbb{R}$, we have 
	\begin{eqnarray*}
	\|\textup{grad}^{g}(f\circ \pi)(u)\|=\|\textup{grad}^{h}(f)(\pi(u))\|
	\end{eqnarray*}
	for all $u\in TM$. 
\end{lemma}
\begin{proof}
	Let $(x_{1},...,x_{n})$ be a coordinate system on $U\subseteq M$, and let $(q,r)=
	(q_{1},...,q_{n},r_{1},...,r_{n})$ be the corresponding 
	coordinates on $\pi^{-1}(U)$ (see Section \ref{nkwwnkwnknknfkn}). With standard notation, 
	we have $\textup{grad}^{h}(f)=\sum_{i,j=1}^{n}h^{ij}\tfrac{\partial f}{\partial x_{i}}\tfrac{\partial}{\partial x_{j}}$ 
	and hence 
        \begin{eqnarray}\label{neknkneknkn}
		\|\textup{grad}^{h}(f)\|^{2}=\sum_{i,j=1}^{n}h^{ij}\dfrac{\partial f}{\partial x_{i}}\dfrac{\partial f}{\partial x_{j}}
	\end{eqnarray}
	on $U$. Without loss of generality, we may assume that $(x_{1},...,x_{n})$ are affine coordinates with respect to $\nabla$. 
	In this case, the local expression of $g$ in the coordinates $(q,r)$ is given by $g(q,r)=
	\big[\begin{smallmatrix}
		h(x)  &  0\\
		0     &   h(x)
	\end{smallmatrix}
	\big]$, where $h(x)$ is the matrix representation of $h$ in the affine coordinates $x=(x_{1},...,x_{n})$.
	From this it follows that $g^{ij}=g^{i+n\,j+n}=h^{ij}\circ \pi$ if $i,j\in \{1,...,n\}$ and $g^{ij}=0$ otherwise. 
	Applying \eqref{neknkneknkn} with $g$ and $f\circ \pi$ in place of $h$ and $f$, respectively, we find 
        \begin{eqnarray}\label{nfekwknrvknk}
		\|\textup{grad}^{g}(f\circ \pi)\|^{2}=\sum_{i,j=1}^{n}(h^{ij}\circ \pi)
		\dfrac{\partial (f\circ \pi)}{\partial q_{i}}\dfrac{\partial (f\circ \pi)}{\partial q_{j}}
		+\sum_{i,j=1}^{n}(h^{ij}\circ \pi)\dfrac{\partial (f\circ \pi)}{\partial r_{i}}\dfrac{\partial (f\circ \pi)}{\partial r_{j}}.
	\end{eqnarray}
	Clearly $\tfrac{\partial (f\circ \pi)}{\partial q_{i}}=\tfrac{\partial f}{\partial x_{i}}\circ \pi$ 
	and $\tfrac{\partial (f\circ \pi)}{\partial r_{i}}=0$, and so \eqref{nfekwknrvknk} can be rewritten as 
        \begin{eqnarray*}
		\|\textup{grad}^{g}(f\circ \pi)\|^{2}=
		\bigg(\sum_{i,j=1}^{n}h^{ij}\dfrac{\partial f}{\partial x_{i}}\dfrac{\partial f}{\partial x_{j}}\bigg)\circ \pi=
		\|\textup{grad}^{h}(f)\circ \pi\|^{2},
	\end{eqnarray*}
	from which the result follows. 
\end{proof}

\begin{lemma}[\textbf{Cram\'er-Rao equality, classical version}]\label{nenkenkfnknk}
	Let $\mathcal{E}$ be an exponential family defined over a finite set $\Omega=\{x_{0},...,x_{m}\}$, as in 
	Setting \ref{mnekmnwknefknk}(a). Let $f:\mathcal{E}\to \mathbb{R}$ be a function of the 
	form $f(p)=\mathbb{E}_{p}(X)$, where $X\in \mathcal{A}_{\mathcal{E}}$ and $p\in \mathcal{E}$. Then 
	\begin{eqnarray*}
		\mathbb{V}_{p}(X)=\|\textup{grad}^{h_{F}}(f)(p)\|^{2}
	\end{eqnarray*}
	for all $p\in \mathcal{E}$.
\end{lemma}
\begin{proof}
	Let $f(p)=\mathbb{E}_{p}(X)$, where $X\in \mathcal{A}_{\mathcal{E}}$. By definition of $\mathcal{A}_{\mathcal{E}}$, 
	there are real numbers $a_{0},....,a_{n}$ such that $X=a_{0}+a_{1}F_{1}+...+a_{n}F_{n}$, and thus
	$f(p)=a_{0}+\sum_{i=1}^{n}a_{i}\eta_{i}$, where $\eta_{i}:\mathcal{E}\to \mathbb{R}$ is defined by $\eta_{i}(p)=\mathbb{E}_{p}(F_{i})$. 
	Taking the gradient with respect to $h_{F}$, we get 
	$\textup{grad}^{h_{F}}(f)=\sum_{i=1}^{n}a_{i}\textup{grad}^{h_{F}}(\eta_{i})$, and thus
	\begin{eqnarray*}
		\|\textup{grad}^{h_{F}}(f)\|^{2}=\sum_{i,j=1}^{n}a_{i}a_{j}
		h_{F}(\textup{grad}^{h_{F}}(\eta_{i}),\textup{grad}^{h_{F}}(\eta_{j})). 
	\end{eqnarray*}
	In the context of information geometry, it is proven that (see, e.g., \cite{Amari-Nagaoka}):
	\begin{enumerate}[(1)]
	\item $\textup{grad}^{h_{F}}(\eta_{i})=\tfrac{\partial}{\partial \theta_{i}}$ for all $i\in \{1,...,n\}$.
	\item $(h_{F})_{ij}(p)=\mathbb{E}_{p}\big((F_{i}-\mathbb{E}_{p}(F_{i})(F_{j}-\mathbb{E}_{p}(F_{j}))\big)$ for all $p\in \mathcal{E}$ and 
		$i,j\in \{1,...,n\}$, where $(h_{F})_{ij}=h_{F}(\tfrac{\partial}{\partial \theta_{i}},\tfrac{\partial}{\partial \theta_{j}})$.
	\end{enumerate}
	Thus 
	\begin{eqnarray*}
		\|\textup{grad}^{h_{F}}(f)\|^{2} &=& \sum_{i,j=1}^{n}a_{i}a_{j}(h_{F})_{ij}= \mathbb{E}_{p}\bigg[\sum_{i,j=1}^{n}a_{i}a_{j}
		(F_{i}-\mathbb{E}_{p}(F_{i})(F_{j}-\mathbb{E}_{p}(F_{j}))\bigg]\\
		&=&\mathbb{E}_{p}\bigg[\bigg(\sum_{i=1}^{n}a_{i}(F_{i}-\mathbb{E}_{p}(F_{i}))\bigg)^{2}\bigg] = \mathbb{V}_{p}(X-a_{0}).
	\end{eqnarray*}
	It follows that $\|\textup{grad}^{h_{F}}(f)\|^{2}=\mathbb{V}_{p}(X-a_{0})$. 
	Since $\mathbb{V}_{p}(X-a_{0})=\mathbb{V}_{p}(X)$, the lemma follows. 
\end{proof}

\begin{proof}[Proof of Theorem \ref{nceknkenknknk}]
	Let $f(p)=\mathbb{E}_{K(\phi(p))}(X)$ be as in Theorem \ref{nceknkenknknk}. 
	First we observe that it is enough to prove the result when $\phi$ is the identity map. Indeed, 
	the fact that $\phi$ is an isometry implies that 
	\begin{eqnarray*}
		\|\textup{grad}^{g}(f)(p)\|^{2}=\|\textup{grad}^{g}(\widetilde{f})(\phi(p))\|^{2}
	\end{eqnarray*}
	for all $p\in N$, where $\widetilde{f}:N\to \mathbb{R}$ is the function defined 
	by $\widetilde{f}(p)=\mathbb{E}_{K(p)}(X)$. Thus \eqref{nkwwnkenfkn} holds for all $p\in N$ 
	if and only if $\mathbb{V}_{K(q)}(X)=\|\,\textup{grad}^{g}(\widetilde{f})(q)\,\|^{2}$ holds for all $q=\phi(p)\in N$. 
	Because of this, it suffices to assume that $f=\widetilde{f}$. 

	Let $\pi=\kappa\circ \tau:T\mathcal{E}\to \mathcal{E}$ be a toric factorization such that $K$ is the natural extension of $\kappa$, 
	and let $\overline{f}:\overline{\mathcal{E}}\to \mathbb{R}$ be the function defined by $\overline{f}(p)=\mathbb{E}_{p}(X)$. 
	Note that $f=\overline{f}\circ K$ on $N$ and $f\circ \tau=\overline{f}\circ \pi$ on $T\mathcal{E}$.
	Let $\widetilde{g}$ be the K\"{a}hler metric on $T\mathcal{E}$ associated to Dombrowski's construction.	
	By Lemmas \ref{nckndwkenknkennk} and \ref{nenkenkfnknk}, we have 
        \begin{eqnarray}\label{neknkwnefknk}
		\|\textup{grad}^{\widetilde{g}}(\overline{f}\circ \pi)(u)\|^{2}=\|\textup{grad}^{h_{F}}(\overline{f})(\pi(u))\|^{2}=
		\mathbb{V}_{\pi(u)}(X)=\mathbb{V}_{K(\tau(u))}(X) 
	\end{eqnarray}
	for all $u\in T\mathcal{E}$. On the other hand, the fact that $\tau$ is locally a Riemannian isometry implies that
        \begin{eqnarray*}
		\textup{grad}^{g}(f)(\tau(u))=\tau_{*}\,\textup{grad}^{\widetilde{g}}(f\circ \tau)(u)=
		\tau_{*}\,\textup{grad}^{\widetilde{g}}(\overline{f}\circ \pi)(u),
	\end{eqnarray*}
	and so, 
	\begin{eqnarray*}
		\|\textup{grad}^{g}(f)(\tau(u))\|^{2}=\|\textup{grad}^{\widetilde{g}}(\overline{f}\circ \pi)(u)\|^{2}.
	\end{eqnarray*}
	It follows from this and \eqref{neknkwnefknk} that 
        \begin{eqnarray*}
		\|\textup{grad}^{g}(f)(\tau(u))\|^{2}=\mathbb{V}_{K(\tau(u))}(X)
	\end{eqnarray*}
	for all $u\in T\mathcal{E}$. Since $\tau:T\mathcal{E}\to N^{\circ}$ is surjective, 
	$\|\textup{grad}^{g}(f)(p)\|^{2}=\mathbb{V}_{K(p)}(X)$ for all $p\in N^{\circ}$. The theorem follows from this and the 
	fact that $N^{\circ}$ is dense in $N$. 
\end{proof}

\subsection{Application I: spectrum and critical values.}

	Given a real-valued function $f$ on $N$, we shall denote by $\textup{Crit}(f)$ the set of critical values of $f$. 
	Thus a real number $\lambda$ is in $\textup{Crit}(f)$ if and only if there is $p\in N$ such that 
	$f(p)=\lambda$ and $f_{*_{p}}=0$. 

\begin{proposition}\label{neknkendkwnkn}
	Let the hypotheses be as in Setting \ref{mnekmnwknefknk}. Suppose that the estimator map $F$ is injective. 
	If $f$ is a K\"{a}hler function on $N$, then 
	\begin{eqnarray}\label{nfeknkenfnk}
		\textup{Crit}(f)\subseteq \textup{spec}(f).
	\end{eqnarray}
\end{proposition}

	Before we carry out the proof, we make a few technical remarks and establish some notation.
	Given a probability function $p:\Omega\to [0,1]$, we will denote by $\textup{supp}(p)$ 
	the set $\{x\in \Omega\,\,\vert\,\,p(x)\neq 0\}$. We shall call $\textup{supp}(p)$ the \textit{support} of $p$. 
	The following result is immediate.

\begin{lemma}\label{nfeknkrnefknk}
	Let $p:\Omega\to [0,1]$ be a probability function and $X:\Omega\to \mathbb{R}$ a random variable. Then 
	$\mathbb{V}_{p}(X)=0$ if and only if $X$ is constant on $\textup{supp}(p)$. 
	In this case, $X(x)=\mathbb{E}_{p}(X)$ for all $x\in \textup{supp}(p)$.
\end{lemma}

\begin{proof}[Proof of Proposition \ref{neknkendkwnkn}]
	Let $f(p)=\mathbb{E}_{K(\phi(p))}(X)$ be a K\"{a}hler function on $N$, where 
	$K$ is the natural extension of some compatible projection, 
	$\phi\in \textup{Aut}(N,g)^{0}$ and $X\in \mathcal{A}_{\mathcal{E}}$. Let $\lambda\in \textup{Crit}(f)$ be arbitrary. 
	By definition, there is $p\in N$ such that $f(p)=\lambda$ and $f_{*_{p}}=0$. The latter condition implies that 
	$\textup{grad}^{g}(f)(p)=0$. By the Cram\'er-Rao equality, $\mathbb{V}_{K(\phi(p))}(X)=0$. 
	By Lemma \ref{nfeknkrnefknk}, $X(x)=\mathbb{E}_{K(\phi(p))}(X)=f(p)=\lambda$ for all $x\in \textup{supp}(K(\phi(p)))$. 
	Thus $\lambda$ is in image of $X$, which means that $\lambda\in \textup{spec}(f)$. 
\end{proof}

	In general, the inclusion \eqref{nfeknkenfnk} is strict, as the following example shows. 

\begin{example}\label{nfeknkenknkefn}
	Suppose that $\mathcal{E}=\mathcal{B}(n)$ and $N=S^{2}\cong \mathbb{P}_{1}(\tfrac{1}{n})$. Let $f$ be a K\"{a}hler function on 
	$S^{2}$. By Proposition \ref{nfekwnkefnknkn}, there are $u\in \mathbb{R}^{3}$ 
	and $r\in \mathbb{R}$ such that $f(p)=\langle u,p\rangle+r$ for all $p\in S^{2}$, where $\langle\,,\, \rangle$ is the Euclidean 
	product on $\mathbb{R}^{3}$. The spectrum of $f$ is the set $\{\lambda_{k}\,\,\vert\,\,k=0,1,...,n\}$, where 
	\begin{eqnarray*}
		\lambda_{k}=\tfrac{2}{n}\|u\|\big(-\tfrac{n}{2}+k\big)+r
	\end{eqnarray*}
	(see \eqref{nekwwnknefk}). Note that $\lambda_{0}=-\|u\|+r$ and $\lambda_{n}=\|u\|+r$. 
	For simplicity, suppose that $u\neq 0.$ Then $f_{*_{p}}=0$ if and only if 
	$p=\pm\tfrac{u}{\|u\|}$. The corresponding critical values are given by $f(\tfrac{u}{\|u\|})=\lambda_{n}$ and 
	$f(-\tfrac{u}{\|u\|})=\lambda_{0}$. Thus $\textup{Crit}(f)=\{\lambda_{0},\lambda_{n}\}$. 
\end{example}

	We end this section with the an important special case. 
%

\begin{proposition}\label{nfeknkrnefknk}
	Let $\mathbb{P}_{n}(1)$ be the complex projection space of holomorphic sectional curvature $c=1$, regarded as a torification 
	of $\mathcal{P}_{n+1}^{\times}$. Then for every K\"{a}hler 
	function $f$ on $\mathbb{P}_{n}(1)$, we have $\textup{Crit}(f)=\textup{spec}(f)$. 
\end{proposition}
\begin{proof}
	In view of Proposition \ref{neknkendkwnkn}, we need only show that every $\lambda$ in $\textup{spec}(f)$ is a critical value of $f$.  
	So let $\lambda\in \textup{spec}(f)$ be arbitrary. Suppose that $f(p)=\mathbb{E}_{K(\phi(p))}(X)$ for all $p\in \mathbb{P}_{n}(1)$, 
	where $K$ is the natural extension of some compatible projection, $\phi\in \textup{Aut}(\mathbb{P}_{n}(1))$ and $X$ 
	is a random variable on $\Omega$. By definition of $\textup{spec}(f)$, there is $x\in \Omega$ such that $\lambda=X(x)$. 
	Let $\delta_{x}:\Omega\to \mathbb{R}$ be the probability function defined by 
	$\delta_{x}(y)=1$ if $x=y$, 0 otherwise. Then $\lambda=\mathbb{E}_{\delta_{x}}(X)$. Since 
	$K:\mathbb{P}_{n}(1)\to \overline{\mathcal{P}_{n+1}^{\times}}=\mathcal{P}_{n+1}$ is surjective, 
	there is $p\in \mathbb{P}_{n}(1)$ such that $K(\phi(p))=\delta_{x}$, and so
	$\lambda=\mathbb{E}_{K(\phi(p))}(X)=f(p)$. Moreover, it follows from the Cram\'er-Rao equality and Lemma \ref{nfeknkrnefknk} that 
	$\|\textup{grad}(f)(p)\|^{2}=\mathbb{V}_{K(\phi(p))}(X)=\mathbb{V}_{\delta_{x}}(X)=0$, which shows that 
	$f_{*_{p}}=0$. Therefore $\lambda$ is a critical value of $f$. 
\end{proof}


\subsection{Application II: fixed points and Dirac functions.}
	
	In this section, we turn our attention to the relationship between the fixed points of the torus action and 
	the Dirac functions in $\overline{\mathcal{E}}$. 
	Given $p\in N$, we will use $G_{p}$ to denote the isotropy subgroup of $p$ for the torus action, that is, 
	$G_{p}=\{a\in \mathbb{T}^{n}\,\,\vert\,\,\Phi(a,p)=p\}$. We recall that $G_{p}$ is a closed Lie subgroup of 
	$\mathbb{T}^{n}$ whose Lie algebra is the set of $\xi\in \mathbb{R}^{n}=\textup{Lie}(\mathbb{T}^{n})$ 
	satisfying $\xi_{N}(p)=0$, where $\xi_{N}$ is the fundamental vector field on $N$ associated to $\xi$. By definition, 
	a point $p\in N$ is a \textit{fixed point} if $G_{p}=\mathbb{T}^{n}$. 

\begin{lemma}\label{nfeknkfenknkn}
	Let $Z:\mathbb{R}^{n}\oplus \mathbb{R}\to \mathcal{A}_{\mathcal{E}}$ be the kernel of the momentum map $\moment:N\to \mathbb{R}^{n}$ 
	relative to $K$, and let $Z_{0}:\mathbb{R}^{n}\to \mathcal{A}_{\mathcal{E}}$ be the map defined by $Z_{0}(\xi)=Z(\xi,0)$.
	If $\xi \in \mathbb{R}^{n}$ and $p,q\in N$, then 
	\begin{enumerate}[(a)]
	\item $\mathbb{V}_{K(p)}(Z_{0}(\xi))=\|\xi_{N}(p)\|^{2}$.
	\item $\textup{Lie}(G_{p})=\{\xi\in \mathbb{R}^{n}\,\,\vert\,\,Z_{0}(\xi)\,\,\textup{is constant on}\,\,\textup{supp}(K(p))\}$. 
	\item If $\textup{supp}(K(p))=\textup{supp}(K(q))$, then $G_{p}=G_{q}$. 
	\end{enumerate}
\end{lemma}
\begin{proof}
	Before we get started, let us establish some notation. 
	Given a smooth function $f:N\to \mathbb{R}$, we will denote by $X_{f}$ and $\textup{grad}^{g}(f)$ the Hamiltonian vector 
	field and Riemannian gradient of $f$, respectively. Since $N$ is a K\"{a}hler manifold, we have 
	$X_{f}=-J\,\textup{grad}^{g}(f)$, where $J:TN\to TN$ is the complex structure. 

	(a) If $f=\moment^{\xi}$, where $\xi\in \mathbb{R}^{n}$, then $X_{f}=\xi_{N}$, and so $\xi_{N}=-J\,\textup{grad}^{g}(\moment^{\xi})$.
	Taking the norm, we get $\|\xi_{N}(p)\|=\|\textup{grad}^{g}(\moment^{\xi})(p)\|$
	for all $p\in N$, where we have use the fact that $J$ is an isometry on each tangent space $T_{p}N$. 
	It follows from this and the Cram\'er-Rao equality that 
	$\mathbb{V}_{K(p)}(Z_{0}(\xi))=\|\xi_{N}(p)\|^{2}$. This shows (a). 

	(b) This follows from (a) and Lemma \ref{nfeknkrnefknk}. 

	(c) If $\textup{supp}(K(p))=\textup{supp}(K(q))$, then (b) implies that $\textup{Lie}(G_{p})=\textup{Lie}(G_{q})$. 
	Because $N$ is a symplectic toric manifold, every isotropy subgroup for the torus action 
	is connected (see \cite{Delzant}, Lemme 2.2). Thus $G_{p}$ and $G_{q}$ are connected Lie subgroups of $\mathbb{T}^{n}$ 
	with the same Lie algebra, and thus they are equal. 
\end{proof}

	We shall say that a probability function $p:\Omega\to [0,1]$ is a \textit{Dirac function} if $\textup{supp}(p)=\{x\}$ 
	for some $x\in \Omega$. In this case, we write $p=\delta_{x}$. (Thus $\delta_{x}(y)=1$ if $x=y$, $0$ otherwise.)

\begin{proposition}[\textbf{Fixed points and Dirac functions}]\label{ndknkefnknk}
	Let the hypotheses be as in Setting \ref{mnekmnwknefknk}. 
	Let $K$ be the natural extension of some compatible projection. Assume that 
	$\mathcal{A}_{\mathcal{E}}$ separates the points of $\Omega$. Given $p\in N$, the following conditions are equivalent:
	\begin{enumerate}[(a)]
	\item $p$ is a fixed point for the torus action.
	\item $K(p)=\delta_{x}$ for some $x\in \Omega.$
	\end{enumerate}
\end{proposition}
\begin{proof}
	$(a)\Rightarrow (b)$ Suppose that $p\in N$ is a fixed point. Thus $G_{p}=\mathbb{T}^{n}$ and $\textup{Lie}(G_{p})
	=\mathbb{R}^{n}$. By Lemma \ref{nfeknkfenknkn}(b), $Z_{0}(\xi)=Z(\xi,0)$ is constant 
	on $\textup{supp}(K(p))$ for all $\xi\in \mathbb{R}^{n}$. Recall that if $r\in \mathbb{R}$, then $Z(\xi,0)+r=Z(\xi,r)$ (see 
	Proposition \ref{nfeknekfnkn}). 
	Thus $Z(\xi,r)$ is constant on $\textup{supp}(K(p))$ for all $(\xi,r)\in \mathbb{R}^{n}\oplus \mathbb{R}$. 
	Since the image of $Z$ is $\mathcal{A}_{\mathcal{E}}$, 
	we obtain that every $X$ in $\mathcal{A}_{\mathcal{E}}$ is constant on $\textup{supp}(K(p))$. Because $\mathcal{A}_{\mathcal{E}}$ 
	separates the points of $\Omega$, this forces $\textup{supp}(K(p))$ to be a singleton, say $\{x\}$. 
	Thus $K(p)=\delta_{x}$. 

	$(b)\Rightarrow (a)$ Suppose that $K(p)=\delta_{x}$ for some $x\in \Omega.$ By Lemmas \ref{nfeknkrnefknk} and 
	\ref{nfeknkfenknkn}(a), we have $\|\xi_{N}(p)\|=0$ for all $\xi\in \mathbb{R}^{n}$, which means that the Lie algebra of $G_{p}$ 
	coincides with the Lie algebra of $\mathbb{T}^{n}$. Since $G_{p}$ is connected (\cite{Delzant}, Lemme 2.2), this implies 
	$G_{p}=\mathbb{T}^{n}$, and shows that $p$ is a fixed point.
\end{proof}

\begin{example}
	Suppose that $N=\mathbb{P}_{n}(1)$ and $\mathcal{E}=\mathcal{P}_{n+1}^{\times}$. Let $\Phi$ and $K$ 
	be as in Example \ref{neknkenfkn}. 
	Then the fixed points for the torus action are $p_{1}=[1,0,...,0]$,..., $p_{n+1}=[0,...,0,1]$, and we have 
	$K(p_{j})=\delta_{x_{j}}$ for all $j\in\{1,...,n+1\}$. 
\end{example}

\appendix 

\section{Transformation groups}\label{nknkvnkfnkfnk}
	In this section, we review some basic results on transformation groups. 
	Our purpose is to get a basic understanding of the Lie group structure of the group of holomorphic isometries of a 
	connected K\"{a}hler manifold. The presentation is as self-contained as possible. For more advanced reading, 
	we refer to \cite{Kobayashi-transformation} and \cite{palais}.

	Let $\Phi:G\times M\to M$ be a Lie group action of a Lie group $G$ on a manifold $M$. 
	Let $\mathfrak{g}=\textup{Lie}(G)$ be the Lie algebra of $G$. Recall that the 
	\textit{fundamental vector field} of $\xi\in \mathfrak{g}$, denoted by $\xi_{M}$, is the vector field on $M$ which is defined 
	by $\xi_{M}(p)=\tfrac{d}{dt}\,\big\vert_{0}\,\Phi(\textup{exp}(t\xi),p)$,
	where $p\in M$ and $\textup{exp}\,:\,\mathfrak{g}\rightarrow G$ is the standard exponential map. 
	If $\mathfrak{X}(M)$ denotes the Lie algebra of vector fields on $M$ 
	endowed with the usual Lie bracket, then the map $f:\mathfrak{g}\to \mathfrak{X}(M)$, defined by $f(\xi)=\xi_{M}$, is an 
	anti-homomorphism of Lie algebras, that is, it is a linear map satisfying 
	\begin{eqnarray}\label{nfeknwkefnkwnk}
		f([\xi,\eta])=-[f(\xi),f(\eta)]
	\end{eqnarray}
	for all $\xi,\eta\in \mathfrak{g}$ (see, e.g., \cite{Ortega}). 

	Given $g\in G$, let $\Phi_{g}:M\to M$, $p\mapsto \Phi(g,p)$. Recall that $\Phi$ is \textit{effective} 
	if $\Phi_{g}=\textup{Id}_{M}$ implies $g=e$, 
	where $e$ is the identity element $G$. Equivalently, $\Phi$ is effective if $\cap_{p\in M}G_{p}=\{e\}$, where 
	$G_{p}=\{g\in G\,\,\vert\,\,\Phi(g,p)=p\}$ is the 
	isotropy subgroup of $p\in M$.

\begin{lemma}\label{wkjkjkefjwkjk}
	Let $\Phi:G\times M\to M$ be a Lie group action of a Lie group $G$ on a manifold $M$. 
	If $\Phi$ is effective, then the map $\textup{Lie}(G)\to \mathfrak{X}(M),$ 
	$\xi\mapsto \xi_{M}$, is injective. 
\end{lemma}
\begin{proof}
	Let $\xi\in \mathfrak{g}=\textup{Lie}(G)$ be such that $\xi_{M}\equiv 0$. Thus $\xi_{M}(p)=0$ for all $p\in M$, which means 
	that $\xi$ is in the Lie algebra of $G_{p}$ for all $p\in M$. Because $G_{p}$ is a Lie subgroup 
	of $G$, its Lie algebra coincides with the set $\{\eta\in \mathfrak{g}\,\,\vert\,\,\textup{exp}(t\eta)\in G_{p}\,\,\forall\,
	t\in \mathbb{R}\}$. Thus $\textup{exp}(t\xi)\in \cap_{p\in M}G_{p}=\{e\}$ for 
	all $t\in \mathbb{R}$. It follows that 
	$\textup{exp}(t\xi)=e$ for all $t\in \mathbb{R}$. Taking the derivative with respect to $t$ at $0$ yields $\xi=0.$
	The lemma follows. 
\end{proof}

	The following lemma is a direct consequence of Lemma \ref{wkjkjkefjwkjk} and \eqref{nfeknwkefnkwnk}. 

\begin{lemma}\label{nfdnkneknwkdnfkn}
	Let $\Phi:G\times M\to M$ be an effective Lie group action of a Lie group $G$ on a manifold $M$. 
	Let $A=\{\xi_{M}\,\,\vert\,\,\xi\in \textup{Lie}(M)\}\subset \mathfrak{X}(M)$, 
	endowed with minus the usual Lie bracket of vector fields. Then the map $\textup{Lie}(G)\to A$, $\xi\mapsto \xi_{M}$ 
	is a Lie algebra isomorphism. 
\end{lemma}

	Now we specialize to the case in which $G$ is a subgroup of the group $\textup{Diff}(M)$ of diffeomorphims of the manifold $M$. 
	More precisely:

\begin{setting}\label{ncekwndeknknk}
	Let $M$ be a manifold and $G$ a subgroup of $\textup{Diff}(M)$. We assume that:
	\begin{itemize}
	\item There is a smooth structure that makes $G$ into a Lie group. 
	\item The topology of $G$ is the compact-open topology (see below). 
	\item The natural action of $G$ on $M$ is smooth. (By ``natural action" we mean the evaluation map $G\times M\to M$, 
		$(g,p)\mapsto g(p)$.)
	\end{itemize}
\end{setting}	

	We continue to denote by $\mathfrak{g}$ the Lie algebra of $G$ and by $\textup{exp}:\mathfrak{g}\to G$ the 
	corresponding exponential map. If $\xi\in \mathfrak{g}$, then $\xi_{M}$ is the fundamental vector 
	field of $\xi$ associated to the natural action of $G$ on $M$. 

\begin{remark}\label{nfekwnkenfkwndkn}
	Let $G\subset \textup{Diff}(M)$ be as in Setting \ref{ncekwndeknknk}.
	\begin{enumerate}[(a)]
	\item If $K$ and $H$ are Lie groups and $f:K\to H$ is a continuous group homomorphism, then $f$ is automatically 
	smooth (see, e.g., \cite{Lee}, Problem 20-11).
	An immediate consequence is that if $K$ is a topological group, that is, a group endowed with a topology such that the 
	multiplication and inversion maps are continuous, 
	then there is only one smooth structure that makes $K$ into a Lie group. 
	It follows from this that the smooth structure on $G$ is unique. 
	\item The natural action of $G$ on $M$ is effective.
	\end{enumerate}
\end{remark}
	Given a complete vector field $X$ on $M$, we will denote by $\varphi^{X}:\mathbb{R}\times M\to M$ the flow of $X$. 
	Given $t\in \mathbb{R}$, we use $\varphi_{t}^{X}$ to denote the map $M\to M$, $p\mapsto \varphi^{X}(t,p)$. 
\begin{lemma}\label{nknenfkwnkfnk}
	Let $G\subset \textup{Diff}(M)$ be as in Setting \ref{ncekwndeknknk}. If $\xi\in \mathfrak{g}$, then 
	$\textup{exp}(t\xi)=\varphi_{t}^{\xi_{M}}$ for all $t\in \mathbb{R}$. 
\end{lemma}
\begin{proof}
	For simplicity, let $\Phi$ denote the natural action of $G$ on $M$. Since $\Phi$ is smooth by hypothesis, 
	the flow of $\xi_{M}$ is given by $\varphi^{\xi_{M}}_{t}(p)=
	\Phi(\textup{exp}(t\xi),p)$. Since $\Phi$ is the evaluation map, $\Phi(\textup{exp}(t\xi),p)=\textup{exp}(t\xi)(p)$. Thus 
	$\varphi^{\xi_{M}}_{t}(p)=\textup{exp}(t\xi)(p)$. Since this holds for every $p\in M$, $\varphi^{\xi_{M}}_{t}=\textup{exp}(t\xi)$.  
\end{proof}

	Before we proceed with the properties of $G$, let us establish some notation. Given two topological spaces $X$ and $Y$, 
	we will denote by $C(X,Y)$ the set of continuous maps from $X$ to $Y$. 

\begin{definition}[\textbf{Compact-open topology}]
	Let $X$ and $Y$ be two topological spaces and $F\subseteq C(X,Y)$. Given a compact subset $K$ of $X$ and an open set $U\subseteq Y$, 
	let $W(K,U):=\{f\in F\,\big\vert\,f(K)\subseteq U\}$. The familly of all sets of the form $W(K,U)$ is a subbases for a topology on 
	$F$ called \textit{compact-open topology}. 
\end{definition}

	Note that if $F\subseteq C(X,Y)$, then the compact-open topology on $F$ coincides with the relative topology induced from 
	the compact-open topology on $C(X,Y)$. 

	It is well-known that if $X$ is a topological space and $(Y,d)$ is a metric space, 
	then a sequence $(f_{n})_{n\in \mathbb{N}}$ of continuous maps $f_{n}\in C(X,Y)$ 
	converges to $g\in C(X,Y)$ with respect to the compact-open topology if and only if $(f_{n})_{n\in \mathbb{N}}$ 
	converges uniformly on compact sets to $g$, that is, 
	if for every $\varepsilon>0$ and every compact set $K\subset X$, there exists $N\in \mathbb{N}$ such that $n\geq N$ 
	implies $d(f_{n}(x),g(x))<\varepsilon$ for all $x\in K$ 
	(see, e.g., \cite{Willard}). Because of this, the compact-open topology is also called the \textit{topology of 
	compact convergence}.

\begin{lemma}\label{nfekwnkenknkndk}
	Let $G\subset \textup{Diff}(M)$ be as in Setting \ref{ncekwndeknknk}. Let $X$ be a vector field on $M$. The following are equivalent:
	\begin{enumerate}[(1)]
		\item There exists $\xi\in \mathfrak{g}$ such that $X=\xi_{M}$.
		\item $X$ is complete and $\varphi_{t}^{X}\in G$ for all $t\in \mathbb{R}$. 
	\end{enumerate}
\end{lemma}
\begin{proof}
	Let $X$	be a vector field on $M$. 

	\noindent $(1)\Rightarrow (2)$ If $X=\xi_{M}$ for some $\xi\in \mathfrak{g}$, then $X$ is complete 
	(since fundamental vector fields are complete) and by Lemma \ref{nknenfkwnkfnk}, 
	$\varphi^{X}_{t}=\textup{exp}(t\xi)\in G$ for all $t\in \mathbb{R}$. 

	\noindent $(2)\Rightarrow (1)$ Suppose that $X$ is complete and 
	$\varphi_{t}^{X}\in G$ for all $t\in \mathbb{R}$. Thus the map $\alpha_{X}:\mathbb{R}\to G$, $t\mapsto 
	\varphi_{t}^{X}$ is well-defined. Note that $\alpha_{X}$ is a group homomorphism, 
	with $\mathbb{R}$ regarded as a Lie group under addition. We claim that $\alpha_{X}$ is continuous. 
	To see this, recall that the family of sets of the form $W(K,U)=\{\varphi\in G\,\,\vert\,\,\varphi(K)\subseteq U\}$, with 
	$K\subset M$ compact and $U\subset M$ open, is a subbases for the compact-open topology on $G$. 
	Thus it suffices to show that $\alpha_{X}^{-1}(W(K,U))$ is open in $\mathbb{R}$ for every compact $K\subset M$ and every 
	open set $U\subset M$. So fix one such $W(K,U)$. If $\alpha_{X}^{-1}(W(K,U))=\emptyset$, there is nothing to prove. 
	Suppose $\alpha_{X}^{-1}(W(K,U))\neq \emptyset$. If $\alpha_{X}^{-1}(W(K,U))$ were not open, there would be a 
	real number $t\in \alpha_{X}^{-1}(W(K,U))$ and a 
	sequence $(t_{n})_{n\in \mathbb{N}}$ of real numbers $t_{n}$ such that $t_{n}\to t$ 
	and $t_{n}\not\in \alpha_{X}^{-1}(W(K,U))$ for all $n\in \mathbb{N}$. 
	By definition of $\alpha_{X}$ and $W(K,U)$, this means that 
	$\varphi_{t}^{X}(K)\subset U$ and $\varphi_{t_{n}}^{X}(K)\not\subseteq U$ for all $n\in \mathbb{N}$. 
	In particular, for every $n\in \mathbb{N}$, there is 
	$k_{n}\in K$ such that $\varphi_{t_{n}}^{X}(k_{n})\not\in U$. Because $K$ is compact, we 
	can assume without loss of generality that $k_{n}\to k$ for some $k\in K$. Then 
	$\varphi_{t_{n}}^{X}(k_{n})=\varphi^{X}(t_{n},k_{n})\to 
	\varphi^{X}(t,k)\in \varphi_{t}^{X}(K)\subseteq U$. Since $U$ is open, this implies that 
	there exists $N\in \mathbb{N}$ such that $\varphi_{t_{n}}^{X}(k_{n})\in U$ 
	for all $n\geq N$, a contradiction. This conclude the proof of the claim. 

	It follows from the claim that $\alpha_{X}$ is a continuous group homomorphism between Lie groups. 
	This implies that $\alpha_{X}$ is smooth (see Remark \ref{nfekwnkenfkwndkn}), which in turn implies that 
	$\alpha_{X}$ is a one-parameter subgroup of $G$. Thus there exists $\xi\in \mathfrak{g}$ 
	such that $\alpha_{X}(t)=\textup{exp}(t\xi)$ for all $t\in \mathbb{R}$. 
	It follows, in view of Lemma \ref{nknenfkwnkfnk}, that $\varphi_{t}^{X}=\varphi_{t}^{\xi_{M}}$ for all $t\in \mathbb{R}$, 
	which implies $X=\xi_{M}$.  
\end{proof}

	Combining the results of Lemmas \ref{nfdnkneknwkdnfkn} and \ref{nfekwnkenknkndk}, we obtain the following result.
\begin{proposition}\label{necknknknfsknfk}
	Let $G\subset \textup{Diff}(M)$ be as in Setting \ref{ncekwndeknknk}. Let $A_{G}$ be the set of 
	complete vector fields $X$ on $M$ satisfying 
	$\varphi_{t}^{X}\in G$ for all $t\in \mathbb{R}$. 
	\begin{enumerate}[(1)]
		\item $A_{G}$ endowed with minus the usual Lie bracket of vector fields is a finite-dimensional Lie algebra. 
		\item The map $\mathfrak{g}\to A_{G}$, $\xi \mapsto \xi_{M}$, is a Lie algebra isomorphism, where $A_{G}$ is endowed with 
			minus the usual Lie bracket of vector fields.
		\item Let $X\in A_{G}$. Under the identification of Lie algebras 
			$\mathfrak{g}\cong A_{G}$, we have $\exp(tX)=\varphi_{t}^{X}$ for all $t\in \mathbb{R}$.  
	\end{enumerate}
\end{proposition}

	Very often, it is possible to identify $G$ with a ``concrete Lie group", as the following result shows.

\begin{lemma}[\textbf{Abstract versus concrete transformation groups}]\label{nknkdnknskndknk}
	Let $G\subset \textup{Diff}(M)$ be as in Setting \ref{ncekwndeknknk}. Let $K$ be a compact Lie group 
	and $\Phi:K\times M\to M$ an effective Lie group action such that $\Phi_{k}\in G$ for all $k\in K$. 
	Let $K'=\{\Phi_{k}\,\,\vert\,\,k\in K\}.$ 
	\begin{enumerate}[(i)]
		\item $K'$ is topologically closed in $G$ (thus $K'$ is a closed Lie subgroup of $G$).
		\item If any of the following holds, then the map $K\to K'$, $k\mapsto \Phi_{k}$ is a Lie group 
			isomorphism, where $K'$ is regarded as a closed Lie subgroup of $G$.
			\begin{enumerate}[(a)]
			\item $G$ is compact.
			\item There exists a $G$-invariant Riemannian metric $g$ on $M$. 			
			\end{enumerate}
	\end{enumerate}
\end{lemma}
\begin{proof}
	(i) Let $(k_{n})_{n\in \mathbb{N}}$ be a sequence of points $k_{n}\in K$ such that $(\Phi_{k_{n}})_{n\in \mathbb{N}}$ 
	converges to some $\psi$ in $G$ with respect to the topology of 
	compact convergence. Because compact convergence implies pointwise convergence, we have $\Phi_{k_{n}}(p)\to \psi(p)$
	for all $p\in M$. Since $K$ is compact, we may assume without loss of generality that $k_{n}\to k$ for some $k\in K$.
	It follows from this and the continuity of $\Phi$ that $\psi(p)=\Phi_{k}(p)$ for all $p\in M$. Thus 
	$\psi=\Phi_{k}\in K'$. This shows that $K'$ is closed in $G$. 

	(ii) For simplicity, let $\phi:K\to K'$, $k\mapsto \Phi_{k}$. If $G$ is compact, then one can construct a $G$-invariant Riemannian 
	metric $g$ on $M$ by standard averaging method (see \cite[Corollary B.35]{Guillemin}), 
	so it suffices to assume (b). Cleary $\phi$ is an isomorphism of groups. 
	Let $(k_{n})_{n\in \mathbb{N}}$ be a sequence of points $k_{n}\in K$ such that 
	$k_{n}\to k$ for some $k\in K$. By continuity of $\Phi$, we have $\Phi_{k_{n}}(p)\to \Phi_{k}(p)$ for all $p\in M$, 
	which means that $(\Phi_{k_{n}})_{n\in \mathbb{N}}$ converges pointwise to $\Phi_{k}$. 
	It is well-known that if $(\varphi_{n})_{n\in \mathbb{N}}$ is a sequence of isometries of a Riemannian manifold 
	that converges pointwise to an isometry $\varphi$, then $(\varphi_{n})_{n\in \mathbb{N}}$ converges to $\varphi$ for the
	compact-open topology (see \cite{Kobayashi-Nomizu}, Lemma 5, Chapter 1 and Theorem 3.10, Chapter 4, or \cite{helgason}, 
	Chapter IV, Lemma 2.4). Since $\Phi_{k_{n}}$ and $\Phi_{k}$ are isometries by (b), it follows that 
	$(\Phi_{k_{n}})_{n\in \mathbb{N}}$ converges to $\Phi_{k}$ with respect to the compact-open topology. This shows that 
	$\phi$ is continuous. Since $\phi$ is a continuous group homomorphism between Lie groups, it is smooth. Finally, it 
	is well-known that a Lie group homomorphism is a Lie group isomorphism if and only if 
	it is bijective (see \cite{Lee}, Corollary 7.6). Since $\phi$ is bijective, it is a Lie group isomorphism. 
\end{proof}

\begin{example}
	Let $S^{2}\subset \mathbb{R}^{3}$ be the unite sphere endowed with the round metric induced 
	from $\mathbb{R}^{3}$, and let $\textup{Isom}(S^{2})$ be the isometry group of
	$S^{2}$ endowed with the compact-open topology. Then $\textup{Isom}(S^{2})$ is a Lie group satisfying the conditions of 
	Setting \ref{ncekwndeknknk} (this follows from the theorem of Myers and Steenrod, see below). 
	On the other hand, the natural action of the orthogonal group $O(3)$ on $S^{2}$ induces a group isomorphism 
	\begin{eqnarray}\label{newnewknefknk}
		O(3)\to \textup{Isom}(S^{2}). 
	\end{eqnarray}
	It follows from Propositon \ref{nknkdnknskndknk} that \eqref{newnewknefknk} is actually a Lie group isomorphism.
\end{example}


	We now turn our attention to the isometry group of a Riemannian manifold. 
	Let $(M,g)$ be a Riemannian manifold. An isometry of $(M,g)$ is a 
	smooth diffeomorphism $\varphi:M\to M$ satisfying $\varphi^{*}g=g$. 
	We will denote by $\textup{Isom}(M)$ the group of isometries of $M$. 

	The following result is due to Myers and Steenrod \cite{Myers} (see also \cite{Kobayashi-transformation} 
	or \cite{Kobayashi-Nomizu} for a modern proof). 
\begin{theorem}\label{dkfjkejfkegjkr}
	Let $M$ be a connected Riemannian manifold. Then 
	$\textup{Isom}(M)$ is a Lie group with respect to the compact-open 
	topology, whose natural action on $M$ is smooth. If $M$ is compact, then $\textup{Isom}(M)$ 
	is also compact.
\end{theorem}

%
\begin{remark}\label{memdnrfkefneknk}
	\textup{}
	\begin{enumerate}[(a)]
		\item $\textup{Isom}(M)$ satisfies the conditions of Setting \ref{ncekwndeknknk}. 
		\item By Proposition \ref{necknknknfsknfk}, the Lie algebra of 
			$\textup{Isom}(M)$ can be identified with the set of complete vector fields $X$ on $M$ 
			satisfying $\varphi_{t}^{X}\in \textup{Isom}(M)$ for all $t\in \mathbb{R}$, where the 
			Lie bracket is minus the usual Lie bracket of vector fields. 
	\end{enumerate}
\end{remark}

	Recall that a vector field $X$ on a Riemannian manifold $(M,g)$ is said to be \textit{Killing} if its flow consists of 
	isometries. We will use the following notation.
	\begin{itemize}
		\item $\mathfrak{i}(M)$ is the vector space of Killing vector fields on $M$.
		\item $\mathfrak{i}_{0}(M)$ is the set of Killing vector fields on $M$ that are complete. 
	\end{itemize}

	Thus we will identify the Lie algebra of $\textup{Isom}(M)$ with $\mathfrak{i}_{0}(M)$, where $\mathfrak{i}_{0}(M)$ 
	is endowed with minus the usual Lie 
	bracket of vector fields. Under this identification, the exponential map 
	is the map $\mathfrak{i}_{0}(M)\to \textup{Isom}(M)$, $X\mapsto \varphi^{X}_{1}$. 

\begin{remark}\label{nfekdwnkenkfnkn}
	On a complete Riemannian manifold $(M,g)$, Killing vector fields are complete, 
	that is, $\mathfrak{i}_{0}(M)=\mathfrak{i}(M)$ (see \cite{Kobayashi-Nomizu}). 
\end{remark}

	We now focus our attention on K\"{a}hler manifolds. Let $X$ and $Y$ be complex manifolds. 
	We denote by $C(X,Y)$ the set of continuous maps between $X$ and $Y$, endowed with the compact-open topology.
	Let $\mathscr{O}(X,Y)$ be the set of holomorphic maps from $X$ to $Y$. 

\begin{lemma}\label{lkdnknfeknkn}
	Let $(f_{n})_{n\in \mathbb{N}}$ a sequence of holomorphic maps 
	$f_{n}:X\to Y$. Suppose that there is a map $f:X\to Y$ such that 
	$(f_{n})_{n\in \mathbb{N}}$ converges uniformly to $f$ on every compact 
	subset of $X$. Then $f$ is holomorphic. Therefore $\mathscr{O}(X,Y)$ is closed in $C(X,Y)$ 
	with respect to the compact-open topology.
\end{lemma}
\begin{proof}[Sketch of proof]
	This follows from the following classical result (see \cite[Corollary 2.2.4]{hormander}): 
	if $(f_{n})_{n\in\mathbb{N}}$ is a sequence of holomorphic functions $f_{n}:U\to \mathbb{C}$ defined 
	on an open set $U\subseteq \mathbb{C}^{m}$ that converges uniformly on compact sets to a function 
	$f:U\to \mathbb{C}$, then $f$ is holomorphic on $U$. 
\end{proof}

	Now let $N$ be a K\"{a}hler manifold with K\"{a}hler metric $g$. Let 
	$\textup{Aut}(N)$ be the group of biholomorphisms of $N$ and $\textup{Aut}(N,g)=\textup{Aut}(N)\cap \textup{Isom}(N)$. 
	The following result is an immediate consequence of Lemma \ref{lkdnknfeknkn}.

\begin{proposition}\label{nfekwnknefknk}
	Let $N$ be a connected K\"{a}hler manifold with K\"{a}hler metric $g$. Then $\textup{Aut}(N,g)$ is 
	closed in $\textup{Isom}(N)$ with respect to the compact-open topology. Therefore $\textup{Aut}(N,g)$ is a closed Lie subgroup of 
	$\textup{Isom}(N)$ whose natural action on $N$ is smooth. 
\end{proposition}

	In view of Remark \ref{memdnrfkefneknk}, the Lie algebra of $\textup{Aut}(N,g)$ is the set of complete Killing vector fields $X$ on $M$ 
	whose flows consist of holomorphic transformations, and the corresponding Lie bracket is minus the usual Lie bracket for vector fields.
	Thus, if $J$ denotes the complex structure of $N$, we have 
	\begin{eqnarray*}
		\textup{Lie}\big(\textup{Aut}(N,g)\big)=\{X\in \mathfrak{i}_{0}(M)\,\,\vert\,\,\mathscr{L}_{X}J=0\},
	\end{eqnarray*}
	(here we use $\mathscr{L}_{X}$ to denote the Lie derivative in the direction of $X$).

	When $N$ is compact, it can be proven that every Killing vector field $X$ on $N$ satisfies $\mathscr{L}_{X}J=0$ 
	(see \cite{moroianu_2007}, Lemma 8.7 and Proposition 15.5), which implies that the Lie algebras of 
	$\textup{Aut}(N,g)$ and $\textup{Isom}(N)$ coincide. It follows immediately that 

\begin{proposition}\label{jefknkefknk}
	In Proposition \ref{nfekwnknefknk}, if $N$ is compact, then the identity components of $\textup{Aut}(N,g)$ and 
	$\textup{Isom}(N)$ are equal, that is, $\textup{Aut}(N,g)^{0}=\textup{Isom}(N)^{0}$. 
\end{proposition}

\begin{footnotesize}\bibliography{bibtex}\end{footnotesize}
\end{document}